\documentclass[11pt]{amsart}
\usepackage[colorlinks=true, pdfstartview=FitV, linkcolor=blue, 
citecolor=blue]{hyperref}

\usepackage{amssymb,amsmath,amscd}
\usepackage{graphicx}
\usepackage{bbm}
\usepackage{a4wide}
\usepackage{pstricks}

\theoremstyle{plain}
\newtheorem{theorem}{Theorem}[section]
\newtheorem{prop}[theorem]{Proposition}
\newtheorem{lemma}[theorem]{Lemma}
\newtheorem{coro}[theorem]{Corollary}
\newtheorem{fact}[theorem]{Fact}
\theoremstyle{definition}
\newtheorem{defin}[theorem]{Definition}
\newtheorem{remark}[theorem]{Remark}
\newtheorem{example}[theorem]{Example}

\newcommand{\ts}{\hspace{0.5pt}}
\newcommand{\nts}{\hspace{-0.5pt}}

\newcommand{\CC}{\mathbb{C}\ts}
\newcommand{\RR}{\mathbb{R}\ts}

\newcommand{\ZZ}{{\ts \mathbb{Z}\ts}}

\newcommand{\TT}{\mathbb{T}}
\newcommand{\MM}{\mathbb{M}}
\newcommand{\NN}{\mathbb{N}}
\newcommand{\XX}{\mathbb{X}}
\newcommand{\YY}{\mathbb{Y}}
\newcommand{\cA}{\mathcal{A}}
\newcommand{\cB}{\mathcal{B}}
\newcommand{\cE}{\mathcal{E}}
\newcommand{\cH}{\mathcal{H}}
\newcommand{\cI}{\mathcal{I}}

\newcommand{\cO}{\mathcal{O}}

\newcommand{\cR}{\mathcal{R}}
\newcommand{\cT}{\mathcal{T}}

\newcommand{\vL}{\varLambda}
\newcommand{\vS}{\varSigma}
\newcommand{\vU}{\varUpsilon}

\newcommand{\aaa}{n_{a}}

\newcommand{\one}{\mathbbm{1}}

\newcommand{\ii}{\ts\mathrm{i}\ts}
\newcommand{\ee}{\mathrm{e}}

\newcommand{\dd}{\, \mathrm{d}}

\newcommand{\bs}{\boldsymbol}
\newcommand{\exend}{\hfill $\Diamond$}

\newcommand{\defeq}{\mathrel{\mathop:}=}

\newcommand{\myfrac}[2]{\frac{\raisebox{-2pt}{$#1$}}
      {\raisebox{0.5pt}{$#2$}}}

\DeclareMathOperator{\dens}{dens}
\DeclareMathOperator{\diag}{diag}
\DeclareMathOperator{\vol}{vol}
\DeclareMathOperator{\Mat}{Mat}
\DeclareMathOperator{\card}{card}

\DeclareMathOperator{\supp}{supp}

\begin{document}

\title[Pair correlations via renormalisation] {Renormalisation of pair
  correlation measures \\[1mm] for primitive inflation rules and 
  absence \\[1mm]   of absolutely continuous diffraction}

\author{Michael Baake}

\author{Franz G\"{a}hler}

\author{Neil Ma\~{n}ibo}
\address{Fakult\"at f\"ur Mathematik, Universit\"at Bielefeld, \newline
\hspace*{\parindent}Postfach 100131, 33501 Bielefeld, Germany}
\email{$\{$mbaake,gaehler,cmanibo$\}$@math.uni-bielefeld.de}

\begin{abstract}
  The pair correlations of primitive inflation rules are analysed via
  their exact renormalisation relations. We introduce the inflation
  displacement algebra that is generated by the Fourier matrix of the
  inflation and deduce various consequences of its structure.
  Moreover, we derive a sufficient criterion for the absence of
  absolutely continuous diffraction components, as well as a necessary
  criterion for its presence. This is achieved via estimates for the
  Lyapunov exponents of the Fourier matrix cocycle of the inflation
  rule. We also discuss some consequences for the spectral measures of
  such systems.  While we develop the theory first for the classic
  setting in one dimension, we also present its extension to primitive
  inflation rules in higher dimensions with finitely many prototiles
  up to translations.
\end{abstract}

\maketitle

\section{Introduction}

The spectral structure of substitution systems gives valuable insight
into such systems and their mutual relations. However, with our
present knowledge, it is still rather far to a classification in
sufficient generality. While the general Pisot substitution
conjecture, despite great progress in recent years (see \cite{Aki,KLS}
and references therein), is still open, the class of constant-length
substitutions is essentially understood, at least on an algorithmic
level. In fact, building on \cite{Q}, Bartlett \cite{Bart} presented a
general method how to determine the spectral measure of maximal type
computationally, for any given primitive constant-length
substitution. Remarkably, this approach also works in higher
dimensions, and various general results have been derived from it.

From the viewpoint of diffraction theory, it is also possible to
derive the spectral type, because the required spectral measures can
be realised as the restriction of certain diffraction measures to a
fundamental domain \cite{BLvE}. More generally, as long as one works
with systems with pure point spectrum, it does not matter whether one
considers the dynamical or the diffraction spectrum, as pure
pointedness of one implies the other, without restriction to
constant-length substitutions; see \cite{LMS,BL,BLvE}.

In general, the situation is less favourable when considering
substitutions that are not of constant length.  For instance, beyond
the much-studied Pisot substitution case, one additional topological
obstacle emerges when the substitution matrix has eigenvalues of
modulus $\lvert \lambda \rvert \geqslant 1$ other than the leading
one.  Here, it generally matters \cite{CS} whether one considers the
symbolic dynamical system, under the $\ZZ$-action of the shift, or the
geometric one, the latter defined via tiles (intervals) of natural
length and studied under the natural translation action of $\RR$. As
was realised and demonstrated in \cite{renex}, the geometric version
possesses an exact renormalisation identity for the pair correlation
measures of the system. This gives access to some spectral properties
that, to our knowledge, are presently not available on the symbolic
level. Substitutions of constant length are special in the sense that
the two viewpoints, symbolic and geometric, coincide, which ultimately
is the reason for their better accessibility.

In this paper, we develop the renormalisation approach for the
geometric setting of primitive substitutions in more generality,
building on previous work on several classes of examples
\cite{renex,Neil,BGM}. Moreover, we extend the approach to inflation
tilings of finite local complexity (FLC) in Euclidean spaces of
arbitrary dimension. In fact, we also take first steps to go beyond
the FLC case. We employ the diffraction theory approach and derive
consequences for the spectral measures where presently possible.  Our
particular interest is the derivation of sufficient criteria for the
absence of absolutely continuous diffraction and spectral
measures, as well as necessary criteria for their presence.
This is motivated by the rare occurrence of such components,
which to date is essentially limited to Rudin--Shapiro-type sequences
and their generalisations; compare \cite{NF-Hadamard,TAO,CG,CGS} and
references therein.\smallskip

It will be instrumental for our analysis that we formulate various
aspects on the symbolic level, while the core of our analysis rests on
the natural geometric realisation in order to profit from the inherent
self-similarity in the form of exact renormalisation relations; see
\cite{BS1,BS2} for related results on spectral measures of
one-dimensional systems via the analysis of matrix Riesz
products. Here, to make the distinction between the types of dynamical
systems as transparent as possible, we will speak of
\emph{substitution rules} on the symbolic side, but of \emph{inflation
  rules} on its geometric counterpart, thus following the notation and
terminology of \cite{TAO,TAO2}. For convenience, various results are
briefly recalled from there. Rather than repeating the proofs, we
provide precise references instead.

The paper is organised as follows. After recalling some general facts
about unbounded, but translation-bounded measures on $\RR$ and their
Fourier transforms in Section~\ref{sec:prelim}, we extend the notions
and results of \cite{renex} on the classic Fibonacci case to general
primitive inflation rules in Section~\ref{sec:general}. Here, we
introduce the \emph{inflation displacement algebra} and derive some of
its general properties in relation to the \emph{Fourier matrix} of the
inflation and its cocycle, which will later help to understand the
Lyapunov spectrum of this cocycle, in Section~\ref{sec:abelian}.  Some
additional material on the corresponding Kronecker product algebra
\cite{renex,BFGR} is gathered in an appendix. We then introduce the
\emph{pair correlation functions} and their renormalisation relations,
which are a consequence of the inflation structure. They lead to an
infinite-dimensional system \eqref{eq:RE1} of linear equations, with
nevertheless an essentially unique solution
(Theorem~\ref{thm:unique-one}).

A measure-theoretic reformulation leads to the \emph{pair correlation
  measures} and their Fourier transforms, with specific access to
their spectral components. In Theorem~\ref{thm:points}, the structure
of the pure point part is made explicit in analogy to the
Bombieri--Taylor approach, compare \cite{Lenz}, while the ensuing
analysis of the absolutely continuous parts forms the core of our
paper. Here, via the introduction of certain Lyapunov exponents for
the Fourier matrix cocycle, we derive an effective sufficient
criterion for the absence of absolutely continuous components in the
diffraction measure of the inflation system (Theorem~\ref{thm:1D}),
together with a general upper bound for the maximal exponent in
Theorem~\ref{thm:geq}. This also provides a necessary criterion for
the presence of absolutely continuous diffraction components in
Corollary~\ref{coro:ac-cond}.

In Section~\ref{sec:abelian}, we apply the general theory to the class
of Abelian bijective substitutions of constant length. With some input
from group representation and character theory, we can derive the
known absence \cite{Q,Bart} of absolutely continuous spectral
components, both in the diffraction and in the dynamical sense, in an
independent way in Theorems~\ref{thm:pos-abelian} and
\ref{thm:extend}. This is also illustrated with several examples,
which are selected to highlight the relations between substitutions,
Fourier matrices and Lyapunov exponents.

Finally, in Section~\ref{sec:higher}, we extend our approach to
inflation tilings in higher dimensions with finitely many prototiles
up to translations. As we shall briefly indicate, this already admits
the treatment of some tilings without finite local complexity
\cite{BG-block}.  The key observation for primitive inflation tilings
with finitely many translational prototiles is that the approach with
the Fourier matrices readily generalises and leads to the essential
separation of the geometric structure of the tiling in Euclidean space
from the combinatorial data of the inflation rule (and its geometry
via the pair correlation measures). This ultimately leads to a
criterion for the absence of absolutely continuous spectral components
in Theorem~\ref{thm:higher-D}, which is the extension of
Theorem~\ref{thm:1D} to this situation. We demonstrate the
effectiveness of our criterion by treating two examples, namely binary
block substitutions (in Section~\ref{sec:block}) and the planar
Godr\`{e}che--Lan\c{c}on--Billard tiling (in
Section~\ref{sec:LB}). The latter is a planar non-Pisot inflation
tiling, built with the Penrose rhombuses, and is shown, via the
aforementioned method, to have an essentially singular continuous
diffraction spectrum.

\section{Radon measures and Eberlein
  convolutions}\label{sec:prelim}

Here, we recall some notions and results on unbounded measures that we
shall need throughout.  A (complex) \emph{Radon measure} on $\RR^d$ is
a continuous linear functional on the space $C_{\mathsf{c}} (\RR)$ of
continuous functions with compact support, the latter equipped with
the inductive limit topology. By the general Riesz--Markov theorem,
Radon measures correspond to regular Borel measures on $\RR^d$, and we
shall use this connection several times. Note that these need not be
finite measures. If $\mu$ is a measure, its twisted counterpart
$\widetilde{\mu}$ is defined via
$\widetilde{\mu} (g) = \overline{\mu (\widetilde{g}\ts )}$ for
$g\in C_{\mathsf{c}} (\RR)$, where
$\widetilde{g} (x) \defeq \overline{g (-x)}$. Moreover, given a mapping
$f$ of $\RR^d$ into itself, the \emph{push-forward} of $\mu$, denoted
by $f \! . \mu$, is defined by
$\bigl( f \! .\mu\bigr) (g) \defeq \mu (g\circ \nts f)$ where
$g \in C_{\mathsf{c}} (\RR^d)$, often called \emph{test function} from
now on.

The convolution of two finite measures $\mu$ and $\nu$ is defined as
\[
    \bigl( \mu * \nu \bigr) (g) \, = 
    \int_{\RR^d} \int_{\RR^d} g (x+y) \dd \mu (x) \dd \nu (y) \ts .
\]
When $f$ is a \emph{linear} map on $\RR^d$, one has the relation
\[
     f \! . (\mu * \nu) \, = \, (f \! . \mu) * ( f \! . \nu) \ts ,
\]
as follows from a simple calculation; compare
\cite[Lemma~2.3]{BG-block}.  Let us fix an averaging sequence
$\cR = (R_n)^{}_{n\in\NN}$ of compact sets $R_n \subset \RR^d$ with
$R^{}_n \subset R^{\circ}_{n+1}$ and $\bigcup_{n\in\NN} R_n = \RR^d$.
We will assume throughout that $\cR$ is a van Hove sequence; see
\cite{TAO,BM} for details. Now, with $\mu|^{}_{R_n}$ denoting the
restriction of $\mu$ to $R_n$, the \emph{Eberlein} (or
volume-averaged) \emph{convolution} of two translation-bounded
measures, relative to $\cR$, is defined as
\[
    \mu \circledast \nu \, \defeq \, \lim_{n\to\infty}
    \frac{\mu |^{}_{R_n} \! * \nu |^{}_{R_n} }{\vol (R_n)} \ts ,
\]
provided the limit exists (we shall not consider any other
situation below). 

A Radon measure $\mu$ is called \emph{positive definite}, if
$\mu (g * \widetilde{g}\ts ) \geqslant 0 $ holds for all
$g\in C_{\mathsf{c}} (\RR^d)$.  A positive and positive definite
measure is automatically translation bounded \cite[Prop.~4.4]{BF}. An
important instance of this is the \emph{autocorrelation measure}
$\gamma$ of a translation-bounded measure $\omega$, defined as
\begin{equation}\label{eq:def-auto}
   \gamma \, = \, \gamma^{}_{\omega} \, \defeq \, \omega
   \circledast \widetilde{\omega} \ts .
\end{equation}
Provided the Eberlein convolution exists, which will be true in all
cases studied below, $\gamma$ is a positive definite measure, and in
many later situations, it is also a positive measure; see
\cite[Sec.~8.5]{TAO} and references therein for more.  For the study
of spectral properties, we are then interested in the Fourier
transform of $\gamma$, denoted as $\widehat{\gamma}$, which is known
as the \emph{diffraction measure} of $\omega$; see \cite{Hof,BL-rev}
as well as \cite[Ch.~9]{TAO} for general background.

There are several possibilities to define and analyse the Fourier
transform of a measure.  This is a non-trivial issue, see \cite{MS}
for a systematic exposition, and part of our later analysis will rely
on the existence of the Fourier transform.  We use a standard
version of the Fourier transform \cite{Rudin,BF} that, for integrable
functions on $\RR^d$ viewed as Radon--Nikodym densities, reads
\[
    \widehat{f} (k) \, = \int_{\RR^d} \ee^{-2 \pi \ii kx} 
     f(x) \dd x  \ts ,
   \]
   where $kx$ denotes the standard inner product between $k$ and $x$
   in $\RR^d$.  Any positive definite measure is Fourier
   transformable. The Fourier transform of a positive definite measure
   is a positive measure; see \cite[Ch.~I.4]{BF} or
   \cite[Sec.~8.6]{TAO} for details. Moreover, Fourier transform is
   continuous on the class of positive and positive definite measures.

\begin{lemma}\label{lem:transformable}
  Let\/ $\mu$, $\nu$ be translation-bounded measures such that\/
  $\mu\circledast \widetilde{\nu}$ as well as\/
  $\mu\circledast\widetilde{\mu}$ and\/
  $\nu\circledast\widetilde{\nu}$ exist, all with respect to the same
  averaging sequence\/ $\cR$. Then, $\mu\circledast \widetilde{\nu}$
  is a translation-bounded and transformable measure, as is\/
  $\widetilde{\mu}\circledast\nu$.
\end{lemma}
\begin{proof}
  Observe first that
  $\widetilde{\mu}\circledast\nu = \widetilde{\mu
    \circledast\widetilde{\nu}}$,
  wherefore it suffices to prove the claim for
  $\mu \circledast\widetilde{\nu}$. Now, as a variant of the (complex)
  polarisation identity, one verifies that
\begin{equation}\label{eq:polar}
   \mu\circledast \widetilde{\nu} \, = \, \myfrac{1}{4}
   \sum_{\ell=1}^{4} \ii^{\ell} (\mu + \ii^{\ell} \nu) \circledast
   (\mu + \ii^{\ell} \nu)\!\widetilde{\phantom{T}} ,
\end{equation}
where all measures on the right-hand side exist due to our
assumptions.  Consequently, $\mu\circledast \widetilde{\nu}$ is a
complex linear combination of four positive definite measures, each of
which is transformable. Moreover, due to our assumptions, these four
measures are translation bounded, so  $\mu\circledast
\widetilde{\nu}$ is translation bounded and transformable as well.
\end{proof}

If $\omega$ is a positive definite measure, $\widehat{\omega}$ is a
well-defined positive measure that has a unique Lebesgue decomposition
$\widehat{\omega} = \widehat{\omega}^{}_{\mathsf{pp}} +
\widehat{\omega}^{}_{\mathsf{cont}}$ into a pure point measure, with a
supporting set that is at most countable, and a continuous one.  On
the level of $\omega$ itself, this corresponds to the Eberlein
decomposition
$\omega = \omega^{}_{\mathsf{sap}} +
\omega^{}_{\text{$0$-}\mathsf{wap}}$ into a strongly almost periodic
measure, whose Fourier transform is
$\widehat{\omega}^{}_{\mathsf{pp}}$, and a null-weakly almost periodic
one; see \cite{MS,TAO2} for background.  Unfortunately, for the
further decomposition
$\widehat{\omega}^{}_{\mathsf{cont}} =
\widehat{\omega}^{}_{\mathsf{sc}} + \widehat{\omega}^{}_{\mathsf{ac}}$
into the singular continuous and absolutely continuous parts, no
general counterpart in the Eberlein decomposition is known at
present. However, some special cases have recently been analysed by
Strungaru \cite{S} that look promising.

\section{Primitive inflation rules in one
  dimension}\label{sec:general}

Let $\cA = \{ a^{}_{1}, \ldots , a^{}_{\aaa}\}$ be our alphabet with
$\aaa$ symbols or letters, and let
$\varrho \! : \, a_{i} \mapsto \varrho (a_{i})$ be a primitive
substitution rule with substitution matrix
$M_{\varrho} = (M_{ij})^{}_{1\leqslant i,j \leqslant \aaa}$, where
\[
    M_{ij} \, \defeq \, \text{number of letters of type $a_{i}$
                    in $\varrho (a_{j})$}
\]
as usual; compare \cite{Q,TAO}. We will also use the notation
$\varrho = \bigl(\varrho (a^{}_{1}), \ldots , \varrho (a^{}_{\aaa})
\bigr)$ for $\varrho$.  Let $\lambda^{}_{\mathrm{PF}}$ denote the
Perron--Frobenius (PF) eigenvalue of $M$, with the usual
interpretation that the corresponding right eigenvector, in
statistical normalisation, provides the relative frequencies of the
letters in a fixed point of $\varrho$ (or of a suitable power of it),
and that the left eigenvector contains the natural prototile lengths
(up to a common overall factor) for the corresponding geometric
inflation rule; see \cite[Ch.~4]{TAO} and references therein for
background. In short, each inflation step consists in first expanding
each tile by a factor of $\lambda$ and then dissecting it into tiles
of the original size, in the order specified by $\varrho$.  By slight
abuse of notation, we use the symbol $\varrho$ both for the (symbolic)
substitution rule and for its partner, the (geometric) inflation rule.

Starting from a fixed point tiling $\cT$ of $\RR$ under the inflation
rule (or one of its powers, if necessary), the corresponding compact
\emph{hull} is defined as
$\YY = \overline{ \{ t + \cT \mid t \in \RR \} }$, with the closure
being taken in the local topology. By standard results, see \cite{TAO}
and references therein, one obtains a topological dynamical system
$( \YY, \RR)$ that is strictly ergodic.\footnote{Note that there is
  another topological dynamical system, denoted by $(\XX, \ZZ)$, which
  emerges from the shift action on the \emph{symbolic} hull $\XX$, the
  latter obtained as the orbit closure of a symbolic fixed point of
  $\varrho$ or a suitable power of it; see \cite{TAO} for more. Also
  this system is strictly ergodic.} In other words, there is just one
way to put an invariant probability measure $\mu$ on it, which is the
one induced by the patch frequencies, and the resulting
measure{\ts}-theoretic dynamical system is denoted by
$(\YY, \RR, \mu)$. Given any element from $\YY$, which is a tiling of
$\RR$ by $\aaa$ (possibly coloured) intervals, there is a
corresponding Delone set $\vL$ obtained by taking the left endpoints
of all tiles of $\cT$. If the tile lengths are not distinct, we
distinguish them by colour, and do the same for the points. Then, one
has a unique decomposition $\vL = \bigcup_{i=1}^{\aaa} \vL_i$, so that
the (coloured) tiling $\cT$ and the (coloured) point set $\vL$ are
\emph{mutually locally derivable} from each other, or MLD for short;
see \cite{TAO} for the concept and more background. It is clear that
$\cT$ and $\vL$ define topologically conjugate dynamical systems on
the orbit closure, which we tacitly identify from now on for ease of
presentation.

\subsection{The inflation displacement algebra (IDA)}

Given $\varrho$, we now assume that we have chosen prototiles of
natural length, hence proportional to the entries of the left PF
eigenvector of $M_{\varrho}$. For standardisation, we shall often take
the shortest interval to have unit length. Now, define the
\emph{displacement matrix}
$T=(T_{ij})^{}_{1\leqslant i,j \leqslant \aaa}$ with set-valued
entries
\begin{equation}\label{eq:def-T}
     T_{ij} \, \defeq \, \{ \text{all relative positions of intervals
     of type $a_{i}$ in the patch $\varrho (a_{j})$} \} \ts .
\end{equation}
For simple examples of non-constant length, we refer to
 \cite[Sec.~4.2]{BFGR} and \cite[Sec.~3.1]{BGM}.
Here and below, relative positions are always defined via the left
endpoints of the tiles (intervals) or patches. Note that
$\varrho (a_{j})$ is a level{\ts}-$1$ supertile. With these
definitions, we have
$\card (T) \defeq \bigl( \card (T_{ij}) \bigr)_{1\leqslant i,j \leqslant
  \aaa} = M_{\varrho}$.
Let us also define the total set $S_{T} \defeq \bigcup_{i,j} T_{ij}$ of
all relative positions of prototiles in level{\ts}-$1$
supertiles. Since they are all non-negative by construction, we may
write $S_{T}$ as an ordered set,
\[
     S_{T} \, = \, \{ x^{}_{1}, \ldots , x^{}_{m} \} \ts ,
\]
with $0=x^{}_{1} < x^{}_{2} < \ldots < x^{}_{m}$ and some $m\in\NN$.

\begin{defin}\label{def:B-mat}
  Let $T_{ij}$ be the displacement sets from Eq.~\eqref{eq:def-T}.
  Then, for $k\in\RR$, the \emph{Fourier matrix}
  $B(k) = \bigl( B_{ij} (k) \bigr)_{1 \leqslant i,j \leqslant \aaa}$
  of $\varrho$ is defined by
\[
    B_{ij} (k) \, =  \sum_{t\in T_{ij}} \ee^{2\pi\ii tk} .
\]
\end{defin}  

Clearly, one has $\overline{B(k)} = B(-k)$ for all $k\in\RR$.  Note
that $B(0) = M_{\varrho}$, and that we have a decomposition of the
form
\begin{equation}\label{eq:def-digit}
   B(k) \, =  \sum_{x\in S_{T}} \ee^{2\pi\ii k x} D_{x}
\end{equation}
with integer matrices $D_{x}$ that satisfy
$\sum_{x\in S_{T}} D_{x} = M_{\varrho}$, as is clear from setting
$k=0$. Note that $D_{x} = D_{y}$ for $x\ne y$ is possible.  Since at
most one prototile of a patch can have its left endpoint at a given
position, it is clear that any $D_x$ can only have entries $0$ and
$1$, namely
\[
   D_{x,ij} \, = \, \begin{cases}
   1, & \text{if $\varrho (a_{j})$ contains a tile of type $a_{i}$
         at position $x$}, \\
   0, & \text{otherwise} . \end{cases}
\]
These matrices are a generalisation of what is known as \emph{digit
  matrices} in constant-length substitutions \cite{Vince,NF},
wherefore we adopt the name here as well; compare also
\cite[Ch.~VIII]{Q}, where they appear as \emph{instruction matrices}.

Let us next consider the $\CC$-algebra $\cB$ that is generated by the
one-parameter matrix family $\{ B(k) \mid k\in \RR \}$.  Since the
algebra $\cB$ is also a finite-dimensional vector space over $\CC$, it
is automatically closed (in any of the matrix norms, which are all
equivalent in this finite-dimensional setting).

\begin{fact}\label{fact:IDA}
  The\/ $\CC$-algebra\/ $\cB$ that is generated by the matrix family\/
  $\{ B(k) \mid k\in \RR \}$ equals the $\CC$-algebra\/ $\cB_{\nts D}$
  that is generated by the digit matrices\/ $\{ D_{x} \mid x\in
  S_{T}\}$.
\end{fact}

\begin{proof}
  The inclusion $\cB \subseteq \cB_{\nts D}$ is immediate from
  Eq.~\eqref{eq:def-digit}. For the converse, let us first assume that
  the digit matrices $ D_{x}$ with $ x \in S^{}_{T}$ are distinct. Then,
  our claim follows from the observation that we can certainly choose
  $m=\card (S^{}_{T})$ distinct numbers $k\in\RR$, say
  $\{ k^{}_{1}, \ldots , k^{}_{m} \}$, such that the corresponding
  vectors
  $ \bigl(\ee^{2\pi\ii k^{}_{\nts \ell} x} \bigr)_{x\in S^{}_{T}}$ with
  $1 \leqslant \ell \leqslant m$ are linearly independent. This gives
  a set of equations of the form \eqref{eq:def-digit} that can now be
  solved for the matrices $D_{x}$, with $x\in S_{T}$, as linear
  combinations in $B(k^{}_{\ell})$.  Consequently,
  $\cB_{\nts D} \subseteq \cB$ and $\cB = \cB_{\nts D}$.

  If the digit matrices $D_{x}$ with $x\in S_{T}$ fail to be distinct,
  a smaller collection of numbers $k^{}_{\ell}$ suffices for an
  analogous argument.
\end{proof}

\begin{defin}
  The matrix algebra\/ $\cB$ of a primitive inflation rule\/ $\varrho$
  is called the \emph{inflation displacement algebra} (IDA) of
  $\varrho$.
\end{defin}

By construction, $\cB$ is a subalgebra of the full matrix algebra
$\mathrm{Mat} (\aaa,\CC)$.  Recall that $\cB$ is \emph{irreducible} if
the only subspaces of $\CC^{\aaa}$ that are invariant under the entire
algebra $\cB$ are the trivial subspaces, $\{ 0 \}$ and $\CC^{\aaa}$.
If there are others, $\cB$ is called \emph{reducible}. More generally,
a matrix family (finite or infinite) is called irreducible if only the
trivial subspaces are invariant, and reducible
otherwise.\footnote{This notion of irreducibility is to be
  distinguished from the one for non-negative matrices used elsewhere
  in this paper. Since this will always be clear from the context, we
  stick to the standard terminology.}  Note that a matrix family is
irreducible if and only if the algebra generated by it is.  We shall
later see various (classes of) examples.

\begin{remark}
  The commutant $\cB'$ of $\cB\subseteq\mathrm{Mat} (\aaa,\CC)$ is
  defined as
\[
    \cB' \, = \, \{ A \in \mathrm{Mat} (\aaa,\CC) \mid
    [A,B] = 0 \text{ for all } B\in \cB \}
\]
and is again a subalgebra of $\mathrm{Mat} (\aaa,\CC)$. By Schur's
lemma for the field $\CC$, we know that $\cB$ irreducible implies
$\cB' = \CC \one$, while the converse is generally false. If, however,
our IDA $\cB$ is closed under taking Hermitian conjugation (which
turns it into a finite-dimensional \mbox{$C^{*}$-algebra}), von Neumann's
bi-commutant theorem states that $\cB=\cB'' = \mathrm{Mat}(\aaa,\CC)$,
and irreducibility of $\cB$ follows. Unfortunately, the IDA rarely is
a $*$-algebra, so irreducibility has to be decided by other means.
However, in view of Burnside's theorem and the fact that our alphabet
has at least cardinality $2$, irreducibility in our situation is
\emph{equivalent} to showing that $\cB = \mathrm{Mat} (\aaa,\CC)$;
compare \cite{LR} and references therein.  \exend
\end{remark}

\begin{remark}
  When a constant-length substitution $\varrho$ is bijective, meaning
  that every column of the word vector
  $\bigl(\varrho (a_i ) \bigr)_{1\leqslant i \leqslant \aaa}$ is a
  permutation of the $\aaa$ letters, all digit matrices $D_x$ are
  permutation matrices. This permits to compute the dimension of
  $\cB = \cB^{}_{G}$ for some groups $G$ via the decomposition of the
  permutation representation $\varPhi$ and some character theory. In
  particular, when the group $G$ generated by the columns of $\varrho$
  is isomorphic to the full symmetric (or permutation)
  group\footnote{We use $\vS_n$ for the symmetric or permutation group
    of $n$ symbols.} $\vS_{\aaa}$, one has
\[
   \cB^{}_{\vS_{\aaa}}  \simeq \, \Mat (\aaa-1, \CC) \oplus \CC \ts .
\]
This follows from the fact that $\varPhi$ splits as the direct sum of
the trivial and the standard representation of $\vS^{}_{\aaa}$,
meaning $\varPhi = 1 \oplus U_{\mathsf{st}}$; compare \cite{JL}.

Furthermore, when $G \simeq A^{}_{\aaa}$, the subgroup of even
permutations, it can be shown that
$\cB^{}_{\! A_{\aaa}} \! \simeq \ts\cB^{}_{\vS_{\aaa}}\nts$.  This
follows from the fact that $U_{\mathsf{st}}$ does not split when
restricted to $A^{}_{\aaa}\nts$, because its character satisfies
$\chi^{}_{\mathsf{st}} (g) \ne 0$ for some
$g\in \vS^{}_{\aaa} \!\setminus A^{}_{\aaa}$, where $A^{}_{\aaa}$ is
an index-$2$ subgroup of $\vS^{}_{\nts \aaa}$; see
\cite[Prop.~20.13]{JL}.

A subgroup for which $U_{\mathsf{st}}$ does split is
$G \simeq D^{}_4 \subset \vS^{}_4$, where we have
$\dim (\cB^{}_{D_4}) =6$; see Example~\ref{ex:D4-IDA} below for a
substitution with such an IDA.  \exend
\end{remark}

Before we continue, we need some result on the relation between IDAs
for $\varrho$ and its powers. Assume that $\varrho$ is primitive, with
substitution matrix $M_{\varrho}$ and PF eigenvalue $\lambda$.  Let
$B(k)$ be the Fourier matrix from Definition~\ref{def:B-mat} for
$\varrho$, and denote the corresponding matrix for $\varrho^{n}$ by
$B^{(n)} (k)$, so $B^{(1)} (k) = B (k)$.  A simple calculation with
the displacements of $\varrho^2$ versus $\varrho$ shows that
$B^{(2)} (k) = B (k) \, B(\lambda \ts k)$ holds; we shall return to
this point in more generality and detail in
Section~\ref{sec:higher}. Inductively, one has
\[
   B^{(n+1)} (k) \, = \, B (k) \, B^{(n)} (\lambda \ts k)
\]
for any $n\in\NN$, and thus also the matrix Riesz product type
relation
\begin{equation}\label{eq:B-powers}
   B^{(n+1)} (k) \, = \, B (k) \, B(\lambda \ts k) 
    \cdots B (\lambda^{n} \ts k) \ts .
\end{equation}
Note that $B^{(n)} (k)$ defines a \emph{matrix cocycle} \cite{Viana}
over the
dilation dynamical system defined by $k \mapsto \lambda \ts k$ on
$\RR_{+}$, which will play a central role in our later spectral
analysis.  We summarise the relations as follows.

\begin{fact}\label{fact:F-matrix}
  Let\/ $B (k)$ be the Fourier matrix of\/ $\varrho$ from
  Definition~\textnormal{\ref{def:B-mat}}. Then, for arbitrary\/
  $n\in\NN$, the Fourier matrix of\/ $\varrho^n$ is given by\/
  $B^{(n)} (k) = B(k) \ts B (\lambda \ts k) \cdots B
  (\lambda^{n-1} k)$.  \qed
\end{fact}

Eq.~\eqref{eq:B-powers} has the following consequence for the IDA
of $\varrho^{n}$, denoted by $\cB^{(n)}$.

\begin{lemma}\label{lem:power-alg}
  Let\/ $\varrho$ be a primitive substitution over a finite alphabet
  with\/ $\aaa$ letters, and consider the corresponding inflation rule
  with $($fixed$\ts\ts )$ natural prototile lengths.  If\/
  $m,n \in \NN$ with\/ $m | n$, one has\/
  $\cB^{(n)}\nts \subseteq \cB^{(m)}$. In particular,
  $\cB^{(n)} \subseteq \cB^{(1)} = \cB$ for all\/ $n\in\NN$.

  Moreover, if there is a\/ $q\in\NN$ such that\/
  $\cB^{(n)} = \mathrm{Mat} (\aaa,\CC)$ holds for all\/ $n \geqslant q$,
  one has\/ $\cB^{(n)} = \mathrm{Mat} (\aaa,\CC)$ for all\/ $n\in\NN$.
\end{lemma}

\begin{proof}
  Let $\lambda$ be the PF eigenvalue of $M_{\varrho}$, which implies
  that $\lambda^{n}$ is the corresponding one of $M_{\varrho^{n}}$.
  The first claim is trivial for $n=m$, so let $n>m$ and set
  $\ell = n/m$, which is an integer $\geqslant 2$ due to our
  assumptions.  As a direct consequence of Eq.~\eqref{eq:B-powers},
  one now derives
\[
   B^{(n)} (k) \, = \,  B^{(\ell\ts m)} (k) \, = \,
   B^{(m)} (k)\, B^{(m)} (\lambda^{m}\ts k) 
    \cdots B^{(m)} (\lambda^{(\ell - 1) m}\ts k) \ts .
\]
This relation entails that the Fourier matrices of $\varrho^{\ts n}$
are elements of $\cB^{(m)}$, hence $\cB^{(n)} \subseteq \cB^{(m)}$ as
claimed.

Now, if $\cB^{(n)} = \mathrm{Mat} (\aaa,\CC)$ holds for all\/
$n \geqslant q$, we may choose $n^{\prime} \defeq q !$, so
$m | n^{\prime}$ holds for all $1 \leqslant m \leqslant q$. The second
assertion then is a consequence of the first.
\end{proof}

The result of Lemma~\ref{lem:power-alg} looks odd at first sight, as
one might expect the IDA $\cB^{(n)}$ to be independent of
$n$. However, this is generally not the case, as the next example
demonstrates.

\begin{example}\label{ex:D4-IDA}
Consider the alphabet $\cA = \{ a,b,c,d \}$ and the constant-length
substitution $\varrho$ defined by
\[
    \mbox{ \Large $\left[\begin{smallmatrix}
    a \\ b \\ c \\ d \end{smallmatrix}\right]$}
    \quad\stackrel{\varrho}{\longmapsto}\quad
    \mbox{\Large $\left[\begin{smallmatrix} 
    ad \\ b\ts c \\ da \\ c\ts b 
    \end{smallmatrix}\right]$}
    \quad\stackrel{\varrho}{\longmapsto}\quad
    \mbox{\Large $\left[\begin{smallmatrix} 
    adcb \\ bcda \\ cbad \\ dabc
    \end{smallmatrix}\right]$}\vspace{3pt}
\]
where we also wrote the second iteration. Now, in the first step, the
columns display the letter permutations $(c\ts d)$ and $(a d b c)$,
which multiplicatively generate a group isomorphic with $D_{4}$. Thus,
the corresponding IDA is $6$-dimensional, and isomorphic with
$\CC \oplus \CC \oplus \Mat (2,\CC)$. In the second step, the columns
show the (non-trivial) letter permutations $(ad)(b\ts c)$,
$(a c)(b d)$ and $(a b)(c d)$, which only generate Klein's $4$-group,
$C_{2} \times C_{2}$, which is Abelian. Here, the IDA is then
$4$-dimensional, and isomorphic to
$\CC \oplus \CC \oplus \CC \oplus \CC$. More generally,
$\cB^{(2n)} = \cB^{(2)}$ and $\cB^{(2n+1)}=\cB^{(1)}$ for all
$n\in\NN$.  \exend
\end{example}

We say that $\varrho$ admits a \emph{substitutional root} if a
substitution $\sigma$ exists such that $\varrho = \sigma^{n}$ for some
$n\geqslant 2$. A necessary condition for this to happen is that the
corresponding substitution matrices satisfy
$M^{}_{\nts \varrho} = M^{n}_{\sigma}$. This condition is not
sufficient, as one can see from $\varrho = (aba,ab)$, which has the
square root $(ab,a)$, versus $\varrho' = (aab,ab)$, which has no
root. Nevertheless, $M_{\varrho} = M_{\varrho'}$, and the two
substitutions even generate the same hull. The following consequence
of Lemma~\ref{lem:power-alg} is immediate.

\begin{coro}\label{coro:roots}
  Let\/ $\varrho$ be a primitive inflation rule with irreducible
  IDA. If\/ $\sigma$ is a substitutional root of\/ $\varrho$, the IDA
  of\/ $\sigma$, when realised with the matching tile lengths, is
  irreducible as well.  \qed
\end{coro}

Let us next state one general sufficient criterion for the
irreducibility of an IDA.

\begin{prop}\label{prop:irreducible}
  Let\/ $\varrho$ be a primitive substitution over a finite alphabet
  with\/ $\aaa\geqslant 2$ letters.  If the natural prototile lengths
  are distinct, the IDA of\/ $\varrho$ is\/
  $\cB = \mathrm{Mat}(\aaa,\CC)$ and hence irreducible.
\end{prop}

\begin{proof}
  If $\varrho$ is primitive, we know that the hull defined by it is
  minimal, and each element of it is linearly repetitive; see
  \cite{TAO} and references therein for background. In particular,
  there is a number $\zeta > 0$ such that every (connected) legal
  patch of length $\geqslant \zeta$ contains at least one copy of 
  each prototile. Denote the natural prototile lengths by
  $\ell^{}_{1}, \ldots , \ell^{}_{\aaa}$, where we may assume that the
  letters of the alphabet are ordered such that
  $\ell^{}_{1} > \ell^{}_{2} > \ldots > \ell^{}_{\aaa} > 0$.

  Define
  $\triangle_{\min} = \min \{\ell^{}_{1} - \ell^{}_{2}, \ell^{}_{2} -
  \ell^{}_{3}, \ldots , \ell^{}_{\aaa} - \ell^{}_{\aaa-1} \}$.
  Then, we pick an integer $q$ such that
  $\lambda^{q} \triangle_{\min} > \zeta$, with $\lambda$ the PF
  eigenvalue of $M_{\varrho}$ as before, and consider
  $\varrho^{\ts q}$. This power of $\varrho$ now has the property that
  the corresponding level{\ts}-$1$ supertile of type $i$, which is the
  patch $\varrho^{\ts q} (a_{i})$, is longer than that of type $i+1$
  by more than $\zeta$, and this holds for all
  $1 \leqslant i \leqslant \aaa-1$. If we draw the level{\ts}-$1$
  supertiles in a stack on top of each other, with coinciding left
  endpoints, we see that each supertile now has an `overhang' of
  length $>\zeta$ over the next one below it. We can now determine the
  digit matrices $D^{(q)}_{x}$ of $\varrho^{\ts q}$ as follows.

  For each $1\leqslant i \leqslant \aaa$, by our above repetitivity
  argument, at least one $x\in T^{(q)}_{i,1}$ exists with
  $x > \lambda^{q}\ts \ell^{}_{2}$, and the corresponding digit matrix
  is $D^{(q)}_{x} = E^{}_{i,1}$, the standard elementary
  matrix. Consequently, all $E^{}_{i,1}$ are in $\cB^{(q)}$.  Next,
  for each $1\leqslant i \leqslant \aaa$, there exists at least one
  $x\in T^{(q)}_{i,2}$ with $x > \lambda^{q} \ell^{}_{3}$, and we have
  $\bigl( D^{(q)}_{x}\bigr)_{i,2} = 1$.  Here, we do not know whether
  the first column of $D_{x}^{(q)}$ contains only zeros, as there
  could be some coincidences between $\varrho(a^{}_{1})$ and
  $\varrho(a^{}_{2})$. However, if
  $\bigl( D^{(q)}_{x}\bigr)_{j,1} = 1$ for some $j$, we may form
  differences with the elementary matrix $E^{}_{j,1}$, which we
  already know to be in $\cB^{(q)}$. So, also all $E^{}_{i,2}$ are in
  $\cB^{(q)}$.

  Proceeding inductively in the row number, we see (after finitely
  many steps) that all matrices $E^{}_{i,j}$ must be in $\cB^{(q)}$,
  wherefore we get $\cB^{(q)} = \mathrm{Mat} (\aaa,\CC)$. Since our
  argument with the sufficiently long overhangs applies to all powers
  $\varrho^{q^{\ts \prime}}$ with $q^{\ts \prime} \geqslant q$, the
  second assertion of Lemma~\ref{lem:power-alg} implies that
  $\cB = \mathrm{Mat} (\aaa,\CC)$ as well, which proves the main
  claim.

  Since we are working over $\CC^{\aaa}$ with $\aaa\geqslant 2$,
  irreducibility follows from Burnside's theorem, as $\CC$ is
  algebraically closed; see \cite{Lang,LR}.
\end{proof}

In view of Eq.~\eqref{eq:def-digit} and the ensuing discussion of
$\cB$ versus $\cB_{D}$, the following is immediate.

\begin{coro}
  Let\/ $\varrho$ be a primitive substitution over a finite alphabet
  with irreducible IDA\/ $\cB$. Then, for each\/ $\varepsilon > 0$,
  the complex algebra generated by the matrices\/
  $\{ B (k) \mid 0 \leqslant k < \varepsilon \}$ is again\/ $\cB$, and
  hence irreducible as well. Moreover, even the finite matrix family\/
  $\{ B(k) \mid k\in J \}$ is irreducible, provided that\/ $J$
  contains at least\/ $r = \lvert S_{T}\rvert$ distinct values of\/
  $k$ that are rationally independent.  \qed
\end{coro}

The point here is that, if $\cB$ is irreducible, no matrix family
$\{ B (k) \mid 0 \leqslant k < \varepsilon \}$ with $\varepsilon > 0$
can possess a non-trivial invariant subspace. This can be viewed as a
first step towards establishing a stronger irreducibility notion, as
needed for a version of Furstenberg's theorem to represent extremal
Lyapunov exponents; compare \cite[Ch.~6]{Viana} and
\cite[Thm.~2.3]{David}.

Note that the IDA of a primitive inflation rule $\varrho$ is
\emph{not} an MLD invariant, see \cite[Sec.~5.2]{TAO} for background,
and neither is its irreducibility. Since the latter is an important
tool, we illustrate this phenomenon with a paradigmatic example.

\begin{example}\label{ex:TM}
The classic \mbox{Thue{\ts}--Morse} (TM) rule
\[
     \varrho^{}_{\mathrm{TM}} : \quad  1 \mapsto  1 \bar{1}
     \; , \quad \bar{1} \mapsto \bar{1}  1
\]
defines a substitution of constant length over the binary alphabet
$\{ 1,\bar{1}\}$, which can thus also be read as an inflation rule for
two prototiles of unit length; compare \cite[Secs.~4.6 and
10.1]{TAO}. It possesses a (self-explanatory) \emph{bar swap symmetry}
in the sense of \cite{renex}, which implies the IDA
$\cB^{}_{\mathrm{TM}}$ to be reducible. In fact,
\[
   \cB^{}_{\mathrm{TM}} \, = \, \biggl\{
   \begin{pmatrix} \alpha & \beta \\
   \beta & \alpha \end{pmatrix} \Big| \,
   \alpha, \beta \in \CC \biggr\}
\]
is a two-dimensional \emph{commutative} subalgebra of 
$\Mat (2,\CC)$; see \cite[Sec.~4.1]{renex} for details.

Next, observe that any sequence in the TM hull is composed of
overlapping words of the form $1 \bar{1}^{\ell} 1$ with
$\ell \in \{ 0,1,2\}$, where the overlap with the preceding (ensuing)
word is always $1$.  These are nothing but the three right-collared
\emph{return words} for the letter $1$ of the TM system; see
\cite{Durand} for background. Consequently, each $1$ in the sequence
is the first letter of one of the right-collared words (in obvious
notation)
\[
    a \, = \, 1 \bar{1} \bar{1} |_{1} \; , \quad
    b \, = \, 1 \bar{1} |_{1} \quad \text{or} \quad
    c \, = \, 1 |_{1} \, ,
\]
so that there is a simple rule between the TM hull and the derived
$abc$ hull with natural tile lengths, which is local in both
directions.  This makes the two hulls MLD as tiling spaces, where $a$,
$b$ and $c$ represent prototiles of lengths $3$, $2$ and $1$,
respectively.

Now, it is easy to check that we inherit an inflation rule for the
new prototiles, namely
\[
   \varrho^{\ts\prime} : \quad a \mapsto abc \; , \quad
   b \mapsto ac \; , \quad c \mapsto b \; ,
\]
which satisfies the conditions of Proposition~\ref{prop:irreducible},
and thus possesses $\Mat (3,\CC)$ as its IDA, which is irreducible.
Note that, for the inflation $\varrho^{\ts\prime}$, the positions
(left endpoints) of the tiles $a$, $b$ and $c$ in any fixed element of
the hull taken together coincide with the positions of all intervals
of type $1$ in the corresponding element of the original TM hull.
\exend
\end{example}

\begin{example}\label{ex:pd}
  Closely related is the \emph{period doubling} substitution
  $\varrho^{}_{\mathrm{pd}} = (AB,AA)$ on the alphabet $\{ A,B\}$.  It
  defines a subshift that is a factor of the TM system, with a
  globally $2:1$ factor map \cite[Thm.~4.7]{TAO}. In the notation of
  Example~\ref{ex:TM}, the latter is given by the sliding block map
  defined via $\psi (1 \bar{1}) = \psi (\bar{1} 1) = A$ and
  $\psi (11) = \psi (\bar{1} \bar{1}) = B$. It is not difficult to
  check that the IDA of $\varrho^{}_{\mathrm{pd}}$ is generated by the
  digit matrices
\[
   D_{0} \, = \, \begin{pmatrix} 1 & 1 \\ 0 & 0 \end{pmatrix}
   \quad \text{and} \quad
   D_{1} \, = \, \begin{pmatrix} 0 & 1 \\ 1 & 0 \end{pmatrix} . 
\]
Since $D_0 D_1 = D_0$ and 
$D_{1} D_{0} = \left( \begin{smallmatrix} 0 & 0 \\ 1 & 1 
\end{smallmatrix} \right)$, one sees that $\cB_{\mathrm{pd}}$ is
a three-dimensional algebra, namely
\[
   \cB_{\mathrm{pd}} \, = \, \biggl\{
   \begin{pmatrix} \alpha  & \alpha + \gamma \\
    \beta + \gamma & \beta \end{pmatrix} \Big| \,
   \alpha, \beta, \gamma \in \CC \biggr\} ,
\]
which is non-commutative, but still reducible, with non-trivial
invariant subspace $\CC \binom{1}{-1}$.

As in our previous example, we can use the return words
\[
    a \, = \, A|^{}_{A} \quad \text{and} \quad
    b \, = \, AB|^{}_{A}
\]
to construct an inflation rule with distinct tile lengths that defines
an MLD system, then with an irreducible IDA.  \exend
\end{example}

\begin{remark}
  Examples \ref{ex:TM} and \ref{ex:pd} can be extended to general
  constant-length substitutions over a binary alphabet
  $\cA = \{ a, b\}$ as follows. If $\varrho = (w_{a}, w_{b} )$, with
  $w_{a}$ and $w_{b}$ being words of the same length, is primitive and
  aperiodic, the corresponding IDA is either $\cB_{\mathrm{pd}}$ or
  $\cB^{}_{\mathrm{TM}}$. The latter case occurs if and only if
  $\varrho$ is bijective (meaning that $w_{a}$ and $w_{b}$ differ at
  every position). Nevertheless, in all these examples, the induced
  inflation rule for the new alphabet based on the return words has an
  irreducible IDA, but defines a tiling system that is MLD with
  the previous one.  \exend
\end{remark}

The structure of the IDA becomes more complex, and more interesting,
for larger alphabets and in higher dimensions. Here, we discuss one
example, and return to IDAs later, in Section~\ref{sec:abelian} and
in Example~\ref{ex:block}.

\begin{remark}
  For the Rudin--Shapiro substitution \cite{Q,TAO}, as defined by
  $\varrho^{}_{\mathrm{RS}} = (02,32,01,31)$ on $\cA=\{ 0,1,2,3 \}$,
   a return word encoding does not  lead to tiles of distinct length.
   Indeed, the eight right-collared return words for the letter $0$ are
\[
     01|_0 \ts , \; 02|_0 \ts , \; 0131|_0 \ts , \; 0232|_0 \ts , \;
%  ab|_a \ts , \; ac|_a \ts , \; abdb|_a \ts , \; acdc|_a \ts , \;
     013132|_0 \ts , \; 01313231|_0 \ts , \;  
%  abdbdc|_a \ts , \; abdbdcdb|_a \ts , \;
     02323132|_0 \ts , \; 0232313231|_0 \ts , 
%  acdcdbdc|_a \ts , \; acdcdbdcdb|_a \ts ,
\]   
   and using them to set up the new alphabet $\cA'=\{ a,b,c, \ldots, h \}$,
   in the same order, one gets the induced substitution
   $\varrho^{\ts \prime}_{\mathrm{RS}} =
      (d, ba, g, bca, ha, he, bcfa, bcfe)$,
%   (3, 10, 6, 120, 70, 74, 1250, 1254)$, 
   which is again primitive.  Its substitution matrix is
\[
     M \, = \, \mbox{\small $ \begin{pmatrix} 
     0 & 1 & 0 & 1 & 1 & 0 & 1 & 0 \\
     0 & 1 & 0 & 1 & 0 & 0 & 1 & 1 \\
     0 & 0 & 0 & 1 & 0 & 0 & 1 & 1 \\
     1 & 0 & 0 & 0 & 0 & 0 & 0 & 0 \\
     0 & 0 & 0 & 0 & 0 & 1 & 0 & 1 \\
     0 & 0 & 0 & 0 & 0 & 0 & 1 & 1 \\
     0 & 0 & 1 & 0 & 0 & 0 & 0 & 0 \\
     0 & 0 & 0 & 0 & 1 & 1 & 0 & 0 \end{pmatrix}$ }
\]   
with eigenvalues $2$, $\pm \sqrt{2}$, $-1$ and $0$ 
(the last with multiplicity $4$).
The PF left eigenvector is $(2,2,4,4,6,8,8,10)$
for the natural interval lengths, with length $2$
for the shortest interval to reflect the meaning of $a$
in the original version, while the corresponding right eigenvector
$\frac{1}{16}(4,4,2,2,1,1,1,1)^T$
codes the letter frequencies as usual.

It is interesting to note that, in the geometric realisation with
natural interval lengths, we thus have a system that is not of
constant length, but shows absolutely continuous diffraction (and thus
also spectral) components. This follows from the fact that the
Rudin--Shapiro system and this geometric return word system are MLD,
hence lead to topologically conjugate dynamical systems under the
translation action of $\RR$.
  
The IDA of the induced inflation rule is irreducible, which has an
interesting consequence on the way the AC spectrum is encoded in the
Fourier matrix cocycle. In particular, we do no longer have a
$k$-independent subspace with unitary dynamics as in the original
version \cite{renex,Neil}, but a $k$-dependent equivariant
family. This would deserve further exploration, in particular
via extending some results in this direction from \cite{David}.
  \exend
\end{remark}

For the appropriate treatment of pair correlations, one has to go one
step beyond the IDA in considering the real algebra generated by the
matrices $\bs{A} (k) = B(k) \otimes \overline{B(k)}$. Since we need
rather little of this extension below, we summarise the basic
properties in an appendix.

\subsection{Pair correlation functions and measures}

Since the structure of the correlation measures for periodic examples
is clear, we can restrict our attention to primitive inflation rules
that are \emph{aperiodic} in the sense of \cite[Def.~4.13]{TAO}.  Note
that, in one dimension, the hull of a primitive inflation rule is
either periodic or aperiodic. Let
$\vL = \vL_{1} \cup \dots \cup \vL_{\aaa}$ be a fixed point of the
aperiodic, primitive inflation rule $\varrho$ (or of $\varrho^{\ts q}$
for some $q\in\NN$, which defines the same hull), where each interval
of type $a_{i}$ carries a marker point of type $i$ at its left
endpoint.  As in \cite{renex}, we define $\nu^{}_{ij} (z)$ as the
relative frequency of the occurrence of distance $z$ between a point of
type $i$ (left) and one of type $j$ (right). These quantities exist
uniformly due to unique ergodicity, and they are constant on the hull
due to minimality.  In fact, for any $\vL$ from the hull, one has the
same relation,
\begin{equation}\label{eq:nu-def}
     \nu^{}_{ij} (z) \, = \, \frac{\dens \bigl( \vL_i \cap 
        (\vL_j - z) \bigr)}{\dens (\vL )} \ts ,
\end{equation}
which entails $\nu^{}_{ij} (0) = 0$ for $i\ne j$ and
$\sum_{i=1}^{\aaa} \nu^{}_{ii} (0) = 1$.

Clearly, we also have
\begin{equation}\label{eq:gen-symm}
    \nu^{}_{ij} (-z) \, = \, \nu^{}_{ji} (z)
\end{equation}
for all $i,j$ and all $z$. Moreover, we know that
\begin{equation}\label{eq:supports}
   \nu^{}_{ij} (z) \, > \, 0 
   \quad \Longleftrightarrow \quad
    \nu^{}_{ij} (z) \, \ne \, 0
   \quad \Longleftrightarrow \quad
   z \in S_{ij} \defeq \vL_{j} - \vL_{i} \ts ,
\end{equation}
where the Minkowski differences $\vL_{j} - \vL_{i}$ are again the same
for each element of the hull, whence $S_{ij}$ is well defined. Note
that the first equivalence is clear by definition. One direction of
the second equivalence follows from the geometric constraint of the
tiling, while the other is another consequence of minimality (and
hence repetitivity).

\begin{figure}[t]
\begin{pspicture}(9,1.5)
%\psset{arrowsize=6pt,arrowinset=0}
\psline[linewidth=1pt,linestyle=dotted]{|-|}(0,1.1)(3.5,1.1)
\psline[linestyle=solid]{|-|}(0.95,1.1)(2.15,1.1)
\psline[linewidth=1pt,linestyle=dotted]{|-|}(4.95,1.1)(9.05,1.1)
\psline[linestyle=solid]{|-|}(5.8,1.1)(7.4,1.1)
\psline[linestyle=solid](0.95,0.3)(5.8,0.3)
\rput(0.5,0.8){\large $x$}
\rput(5.4,0.78){\large $y$}
\rput(3.5,0){\large $z$}
\end{pspicture}
\caption{If two tiles (solid intervals) at distance $z$ have offsets
  $x$ and $y$ within their covering supertiles (dotted intervals), the
  latter have distance $z+x-y$.\label{fig:one}}
\end{figure}
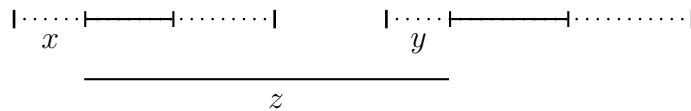

In \cite{renex}, the pair correlation functions $\nu^{}_{ij} (z)$ for
the example of the Fibonacci chain are shown to satisfy a set of
exact, linear \emph{renormalisation equations}, which we now extend to
primitive inflation rules in full generality as follows.

\begin{lemma}\label{lem:frequ-eqs}
  Let\/ $\YY$ be the tiling space defined by an aperiodic, primitive
  inflation rule\/ $\varrho$, with a fixed set of\/ $\aaa$ prototiles
  of natural length. Given some\/ $\vL\in\YY$, let\/ $\nu^{}_{ij} (z)$
  be the relative frequency of occurrence of a tile of type $i$
  $($left$\ts\ts )$ and one of type $j$ $($right$\ts\ts )$ at
  distance\/ $z$ between their left endpoints, which exists and is
  independent of\/ $\vL$. Then, these coefficients are non-negative
  and satisfy the renormalisation equations
\begin{equation}\label{eq:RE1}
    \nu^{}_{mn} (z) \, = \,\myfrac{1}{\lambda} \sum_{i,j = 1}^{\aaa}
    \, \sum_{ x \in T_{m i}} \, \sum_{ y \in T_{n j}}
    \nu^{}_{ij} \Bigl( \myfrac{z+x-y}{\lambda} \Bigr),
  \end{equation}
  for $1\leqslant m,n\leqslant \aaa$, together with the symmetry
  relation~\eqref{eq:gen-symm} and the support
  condition~\eqref{eq:supports}. 
\end{lemma}

\begin{proof}[Sketch of proof]
  The derivation works in complete analogy to the Fibonacci case in
  \cite{renex}, employing local recognisability, as illustrated in
  Figure~\ref{fig:one}.  In fact, it is not necessary that we start
  from a fixed point, because the pair correlation functions exist,
  and are the same for every element of the hull $\vL$, which is
  minimal. The derivation only requires the recognition of the unique
  level{\ts}-$1$ supertile to which any individual tile belongs.
\end{proof}

\begin{remark}\label{rem:periodic}
  The derivation of Eq.~\eqref{eq:RE1} relies on the unique
  identification of the covering supertile for each tile, but it does
  not require this process to be local. Therefore, the statement of
  Lemma~\ref{lem:frequ-eqs} can also be extended to primitive inflation
  rules that define a periodic hull. In such a case, one consistently
  marks the supertiles in each element of the hull, for which one has
  more than one choice, and proceeds with the otherwise unchanged
  proof. It is not difficult to check that the outcome does not depend
  on the actual decomposition chosen.  \exend
\end{remark}

For some aspects, as exploited in \cite{renex}, we need to understand
precisely to what extent the renormalisation equations~\eqref{eq:RE1}
determine the frequencies.  Let us thus look at these relations from
scratch, for which we first need to recall some results from the
theory of matrices with non-negative entries.

A matrix $M\in \Mat (d,\ZZ)$, written as
$M = (m^{}_{ij})^{}_{1 \leqslant i,j \leqslant d}$, is called
\emph{non-negative} if $m^{}_{ij} \geqslant 0$ for all $i,j$. We
assume the reader to be familiar with the classic notions of
irreducibility and primitivity of such matrices, and with the classic
theorems due to Perron and Frobenius; compare \cite[Sec.~2.4]{TAO} and
references given there, or \cite[Ch.~13]{G} for a detailed
account. Moreover, we shall need some other properties that are best
stated via the \emph{normal form}, $M_{\mathsf{nf}}$, of a
non-negative matrix $M$. Following \cite[Ch.~13.4]{G}, it is given by
\begin{equation}\label{eq:normal-form}
\renewcommand{\arraystretch}{1.2}
   M_{\mathsf{nf}} \, = \,
   \left(\begin{array}{ccc@{\;}c@{\;}c@{\;}c@{\;\,}c}
   M_{1} &   & \bs{0} &\vline&  && \\
     &  \ddots && \vline && \bs{0} & \\
   \bs{0} &  & M_{r} & \vline & & & \\
   \hline
   M_{r+1,1} & \cdots &  M_{r+1,r} && M_{r+1} && \bs{0} \\
    \vdots &  & && \!\!\! \ddots  & \ddots & \\
   M_{s,1} & &  \cdots && & M_{s,s-1} & M_{s}
   \end{array}\right)
\renewcommand{\arraystretch}{1}
\end{equation}
where $s\geqslant r \geqslant 1$ and all $M_{i}$ are
indecomposable,\footnote{A square matrix $M$ is called
  \emph{decomposable} if it can be brought to the block-triangular
  form $M' = \left(
  \begin{smallmatrix} A & 0 \\ B & C \end{smallmatrix} \right)$
via simultaneous permutations of its rows and columns, 
and \emph{indecomposable} otherwise \cite{G}.}  
non-negative square matrices, and where, if $s>r$, each
sequence $M_{r+\ell,1}, \dots, M_{r+\ell,r+\ell-1}$ contains at least
one non-zero matrix. Such a normal form can always be achieved by a
suitable permutation of the coordinates, hence by the corresponding
simultaneous permutation of the rows and columns of $M$. It is
essentially unique, up to obvious permutations of entire blocks,
which clearly do not change the values of $r$ and $s$.

In particular, we need the following result, which is a simple extension
of \cite[Thm.~13.7]{G}.
\begin{lemma}\label{lem:normal-form}
  Consider a non-negative matrix\/ $M$ in normal form\/
  $M_{\mathsf{nf}}$ according to Eq.~\eqref{eq:normal-form}, and let\/
  $\lambda$ be an eigenvalue of\/ $M$. Then, $M$ has a corresponding
  strictly positive eigenvector, meaning that all entries are
  positive, if and only if
\begin{enumerate}\itemsep2pt
\item $M_{i}$ has eigenvalue\/ $\lambda$ for every\/ $1\leqslant
  i\leqslant r$;
\item No\/ $M_{j}$ with\/ $r < j \leqslant s$ has\/ $\lambda$
  as an eigenvalue.
\end{enumerate}
In this situation, one has\/ $\lambda > 0$, and the eigenspace 
of\/ $\lambda$ is one-dimensional if and only if\/ $r=1$.  \qed
\end{lemma}

Having exact renormalisation relations for the pair correlation
functions in Lemma~\ref{lem:frequ-eqs}, where the support of these
functions also follows from the inflation construction, it is a
natural question to what extent these relations, taken on their own,
determine the $\nu^{}_{ij}$. This is tantamount to asking what the
solution space of \eqref{eq:RE1} is, and how it depends on the
prescribed support of the $\nu^{}_{ij}$. The latter point is critical,
due to the occurrence of scaled arguments, and increasing the support
might increase the solution space; compare \cite{renex}.

\begin{theorem}\label{thm:unique-one}
  Assume that the geometric data needed to write down
  Eq.~\eqref{eq:RE1} are taken from a $($periodic or
  aperiodic$\ts\ts )$ primitive inflation rule\/ $\varrho$ with
  inflation multiplier\/ $\lambda$ as explained above.  Let\/
  $\nu^{}_{ij}$, with\/ $1\leqslant i,j \leqslant \aaa$, be
  real-valued functions with\/
  $\supp (\nu^{}_{ij}) \subseteq \Delta^{}_{ij} \subset \RR$, where
  the\/ $\Delta^{}_{ij}$ are given point sets such that\/
  $\Delta \defeq \bigcup_{i,j} \Delta^{}_{ij}$ is locally finite and
  contains\/ $0$.  Then, the solution space of the linear
  renormalisation equations \eqref{eq:RE1} with functions of this type
  is finite-dimensional. When\/ $S^{}_{ij} \subseteq \Delta^{}_{ij}$
  for all\/ $1\leqslant i,j \leqslant \aaa$, this solution space is
  non-trivial.

  In particular, when\/ $\Delta^{}_{ij} = S^{}_{ij}$ for all\/
  $1\leqslant i,j \leqslant \aaa$, the solution space is
  one-dimensional, so there is precisely one solution with\/
  $\sum_{i=1}^{\aaa} \nu^{}_{ii} (0) = 1$, which is actually strictly
  positive. This solution automatically satisfies the symmetry
  relation~\eqref{eq:gen-symm}.  If\/ $\Delta^{}_{ij}$ is further
  restricted to a true subset of\/ $S^{}_{ij}$ for at least one index
  pair, the solution space becomes trivial.
\end{theorem}

\begin{proof}
  Observe first that, since $\lambda > 1$, the set of
  equations with
\[
    \lvert z \rvert  \, \leqslant \, c \, \defeq \,\myfrac{1}{\lambda - 1}\,
    \sup \{ x-y \mid x \in T^{}_{m i} \ts , \, y \in T^{}_{n j} 
    \ts , \, 1 \leqslant i,j,m,n \leqslant \aaa \}
\]
forms a closed subsystem of linear equations. When the total support
is inside a fixed locally finite set $\Delta\subset\RR$, we know that
$\Delta \cap [-c,c]$ is a finite set, wherefore the closed subsystem
comprises \emph{finitely} many equations only (there is at least one,
since $0\in\Delta$ by assumption).

Observe next that we only need to prove the dimensionality claim for
this subsystem. Indeed, since $\Delta\subset\RR$ is locally finite,
there is a smallest $z\in\Delta$ with $z> c$, and the values of the
correlation coefficients at $z$ are then uniquely determined, either
from the support constraint or from Eq.~\eqref{eq:RE1}, where only
arguments with modulus $< z$ (and hence $\leqslant c$ in this case)
occur on the right-hand side. Proceeding inductively with growing
modulus of $z$, one sees that the correlation coefficients are
uniquely determined for all $z>c$. An analogous argument works for all
$z\in\Delta$ with $z<-c$.

Consequently, the solution space dimension of the entire system, with
functions $\nu^{}_{ij}$ with $\supp (\nu^{}_{ij}) \subseteq \Delta$,
equals that of the closed, finite subsystem. This proves the first
claim, while the underlying inflation tiling space with its properties
guarantees at least one solution if the supports are large enough;
compare Eq.~\eqref{eq:supports}.

Indeed, when $\Delta^{}_{ij} = S^{}_{ij} = \vL^{}_{j} - \vL^{}_{i}$,
we know from Eq.~\eqref{eq:supports} that a strictly positive solution
exists, with $\supp (\nu^{}_{ij}) = \Delta^{}_{ij}$ for all
$1\leqslant i,j \leqslant \aaa$. Consider the subset of equations of
\eqref{eq:RE1} that emerge from inserting $z=0$ on the left-hand
side. Since $\vL_{i} \cap \vL_{j} = \varnothing$ for $i\ne j$, one has
$0\in S_{ij}$ if and only if $i=j$. This means that we must have
$\nu^{}_{ij} (0) = 0$ for all $i\ne j$. Moreover, if
$x \in T^{}_{m i}$ and $y \in T^{}_{n j}$, one has
$0\leqslant x < \lambda \ts \ell_{i}$ and
$0\leqslant y < \lambda \ts \ell_{j}$ and hence the inequality
\[
    - \ell_{j} \, < \, \frac{x-y}{\lambda} \, < \, \ell_{i}
\]
from the geometry of the prototiles. So, for $x \ne y$, the point
$\frac{x-y}{\lambda}$ cannot be an element of $S_{ij}$, and
$\nu^{}_{ij}$ must vanish there.  We thus remain with the relations
\[
    \nu^{}_{mm} (0) \, = \, \myfrac{1}{\lambda}
    \sum_{i=1}^{\aaa} \card (T^{}_{m i}) \, \nu^{}_{ii} (0) \ts ,
\]
for $1\leqslant m \leqslant \aaa$, which together give the eigenvalue
equation
\begin{equation}\label{eq:subblock}
    M_{\varrho} \begin{pmatrix} \nu^{}_{11} (0) \\
    \vdots \\ \nu^{}_{\aaa\aaa} (0) \end{pmatrix} \, = 
    \, \lambda \begin{pmatrix} \nu^{}_{11} (0) \\
    \vdots \\ \nu^{}_{\aaa\aaa} (0) \end{pmatrix} .
\end{equation}
By primitivity of $\varrho$, the non-negative matrix $M_{\varrho}$ is
primitive, and the eigenspace of $\lambda = \lambda^{}_{\mathrm{PF}}$
is one-dimensional. Moreover, there is an eigenvector with
$\nu^{}_{ii} (0) > 0$ for all $i$, which can be normalised as
$\sum_{i=1}^{\aaa} \nu^{}_{ii} (0) = 1$. If we now multiply both sides
of Eq.~\eqref{eq:RE1} by $\lambda$, we may interpret it as an
eigenvector equation for eigenvalue $\lambda$, where we have finitely
many vector components of the form $\nu^{}_{ij} (z)$ with
$z \in S^{}_{ij}$ and $\lvert z \rvert \leqslant c$ from above.  The
right-hand side of Eq.~\eqref{eq:RE1} can now be seen as the
application of a non-negative matrix $A$ to this vector, where $A$ is
decomposable, because we have identified an irreducible subblock in
Eq.~\eqref{eq:subblock}.

Since we know from the underlying inflation $\varrho$ together with
Eq.~\eqref{eq:supports} that $A$ has $\lambda$ as an eigenvalue with a
strictly positive eigenvector corresponding to it, we deduce from
Lemma~\ref{lem:normal-form} that the eigenspace for $\lambda$ is
one-dimensional if no other invariant diagonal subblock of $A$ exists
with eigenvalue $\lambda$, meaning $r=1$ in
Lemma~\ref{lem:normal-form}.

Indeed, when there is some $0 < z \in S_{ij}$, we know that there
exists a legal patch of the form $a_{i} \ts w \ts a_{j}$, with $w$ a
finite word and $a_{i} \ts w$ coding a patch of length $z$. From the
structure of the primitive inflation $\varrho$, we know that
$a_{i} \ts w \ts a_{j}$ must then be a subword of
$\varrho^{k} (a^{}_{1})$ for some $k\in\NN$, where we may assume $k$
to be the minimal such power. Now, applying the renormalisation
equation $k$ times to $a^{}_{1}$ implies that $\nu^{}_{ij} (z)$ is
linked to $\nu^{}_{11} (0)$ on this level, as the patch under
consideration lies in this very level{\ts}-$k$
supertile. Consequently, $\nu^{}_{ij} (z)$ cannot belong to a
decoupling subset of components. An analogous argument holds for
negative $z$.

We thus see that the parameter $r$ of the normal form
$A_{\mathsf{nf}}$ of $A$ must be $r=1$, and the dimension of our
solution space is indeed $1$ when $\Delta^{}_{ij}=S^{}_{ij}$. The
symmetry is clear, while the final claim is now a simple consequence
of Eq.~\eqref{eq:supports}, which followed from the strict ergodicity
of our underlying dynamical system.
\end{proof}

\begin{remark}\label{rem:K-struc}
  One can rewrite Eq.~\eqref{eq:RE1} in a slightly different way that
  establishes a link to the generators of the $\RR$-algebra $\bs{\cA}$
  from the Appendix, as originally studied in \cite[Sec.~5.2]{BFGR} in
  the case of binary alphabets. If we use
  $\underline{\nu} = (\nu^{}_{11},\nu^{}_{12}, \ldots,
  \nu^{}_{\aaa\aaa})^{T}$ in the standard ordering for a double index,
  one can check that Eq.~\eqref{eq:RE1} can be rewritten as
\begin{equation}\label{eq:RE2}
   \lambda \, \underline{\nu} (z) \, = \sum_{x\in S_{T} - S_{T}}\!
   F_{\nts x} \, \underline{\nu} \left(\myfrac{z+x}{\lambda} \right),
\end{equation}
where we use (as before) the convention to set $\nu^{}_{ij} (z) = 0$
whenever $z$ is not in the admissible set $\Delta^{}_{ij}$. Note that
this version of the renormalisation equation can be read as an
eigenvector equation for a non-negative matrix.  \exend
\end{remark}

\subsection{Renormalisation for correlation measures and their
Fourier transforms}

Let us return to the hull $\YY$ defined by the primitive
inflation rule $\varrho$, and consider some $\vL \in \YY$.  The
Minkowski difference $\Delta=\vL - \vL$ is locally finite and the same
set for all elements $\vL \in \YY$.  We may thus consistently define
pure point measures
$\vU^{}_{\nts m\ts n} = \sum_{z\in\vL - \vL} \nu^{}_{m\ts n} (z) \,
\delta^{}_{z}$
as in \cite{renex}, called the \emph{pair correlation measures}. 
Now, Eqs.~\eqref{eq:gen-symm} and \eqref{eq:supports} imply
\begin{equation}\label{eq:props-one}
   \widetilde{\vU^{}_{\nts m\ts n}} = 
   \vU^{\phantom{\chi}}_{\! n\ts m }
   \quad \text{and} \quad
   \vU^{\phantom{\chi}}_{\! m\ts n} \geqslant 0 \ts .
\end{equation}
Moreover, each $\vU^{}_{\nts m\ts m}$ is a positive definite measure,
and, due to Eq.~\eqref{eq:nu-def}, see also \cite[Eq.~4.1]{BFGR}, one
actually has the representation
\begin{equation}\label{eq:corr-as-conv}
    \vU^{}_{\nts m \ts n} \, = \,
   \frac{ \widetilde{\delta^{}_{\! \vL_m}}\!
    \circledast \delta^{}_{\! \vL_n} }{\dens (\vL)}
\end{equation}
which could also be used to define the measures in the first place.

With this definition, Eq.~\eqref{eq:RE1} implies the
measure{\ts}-valued counterpart
\begin{equation}\label{eq:RE3}
   \vU^{}_{\nts m\ts n} \, = \, \myfrac{1}{\lambda} \sum_{i,j = 1}^{\aaa}
   \sum_{r\in T_{m i}} \sum_{s\in T_{n j}} \delta^{}_{s-r} *
   \bigl( f\nts \nts . \ts  \vU^{}_{\nts ij}\bigr) ,
\end{equation}
as originally derived in \cite[Lemma~4.2]{BFGR} for the binary case.
Here, $f(x) \defeq \lambda x$ and $f\! . \ts \mu$ is defined via
$\bigl( f\! . \ts \mu\bigr) (\cE) = \mu \bigl( f^{-1} (\cE)\bigr)$ for
Borel sets $\cE$, which matches with the definition via test functions
used earlier. Note that the $\vU^{}_{\nts i j}$ are the pair
correlation measures both of the entire hull $\YY$ and of each
individual member of $\YY$.

\begin{remark}\label{rem:diffraction}
  The pair correlation measures can be considered as the building
  blocks of diffraction theory as follows. Decompose $\vL \in \YY$ as
  $\vL = \dot{\bigcup}_{1 \leqslant i \leqslant \aaa} \vL_i$ into the
  distinct types of points, and consider the measure
  $\omega = \sum_{i=1}^{\aaa} u^{}_{i} \ts \delta^{}_{\! \vL_i}$ with
  weights $u^{}_{i} \in \CC$, which is translation bounded by
  construction. Its autocorrelation $\gamma = \gamma^{}_{u}$ according
  to Eq.~\eqref{eq:def-auto} exists, and reads
\[
    \gamma^{}_{u} \, = \, \dens (\vL)
    \sum_{i,j=1}^{\aaa} \overline{u^{}_{i}}
    \, \vU^{}_{ij} \ts u^{}_{j} \ts ,
\]
which also implies that one has
$\widehat{\gamma^{}_{u}} = \dens (\vL) \sum_{i,j} \overline{u^{}_{i}}
\, \widehat{\vU}^{}_{ij} \ts u^{}_{j}$.
In this sense, understanding $\vU$ and $\widehat{\vU}$ gives 
complete access to the diffraction measures of the dynamical system.
\exend
\end{remark}

Observe that, by Lemma~\ref{lem:transformable}, all measures
$\vU^{}_{\nts m n}$ are translation bounded and transformable.
As a result of Eq.~\eqref{eq:props-one}, we also know that
\begin{equation}\label{eq:props-two}
   \overline{\widehat{\vU^{}_{\nts m n}}} \, = \, 
   \widehat{\widetilde{ \vU^{}_{\nts m n}}} \, = \,
   \widehat{\vU^{}_{\nts n m}} \ts ,
\end{equation}
and each $\widehat{\vU^{}_{\nts m n}}$ is a positive definite measure,
with $\widehat{\vU^{}_{\nts m m}} \geqslant 0$ in addition for all $m$. 

Recalling from \cite[Lemma~2.5]{BGM} that $\widehat{f\nts\nts .\mu} =
\frac{1}{\lambda}\ts (f^{-1}\! .\ts \widehat{\mu})$ holds for any
transformable measure $\mu$, and applying Fourier transform and
the convolution theorem to \eqref{eq:RE3}, one finds
\begin{equation}\label{eq:RE-new}
   \widehat{\vU^{}_{\nts m n}} \, = \,  
   \myfrac{1}{\lambda^{2}} \sum_{i,j = 1}^{\aaa}
   \sum_{r\in T_{m i}} \sum_{s\in T_{n j}} 
   \ee^{-2\pi\ii (s-r) (.)} \bigl( f^{-1} \! . \ts 
   \widehat{\vU^{}_{\nts ij}}\bigr) ,
\end{equation}
which is the appropriate generalisation of the Fibonacci equations
from \cite{renex} to this more general situation. Let us rewrite
Eq.~\eqref{eq:RE-new} in matrix form to highlight its structure.
Employing column vector notation for the Kronecker product structure
discussed in Remark~\ref{rem:K-struc} and in the Appendix, with
\[
   \vU \, \defeq \,
   (\vU^{}_{\nts 11}, \vU^{}_{\nts 12} ,
     \ldots , \vU^{}_{1\aaa},
    \vU^{}_{\nts 21}, \vU^{}_{\nts 22}, \ldots ,
    \vU^{}_{\nts \aaa\aaa})^{T}
\]
and similarly for the Fourier transform, one now finds
\begin{equation}\label{eq:FT-scaling}
    \widehat{\vU} \, = \, \myfrac{1}{\lambda^{2}} \, \bs{A}(.)
    \bigl(f^{-1} \! . \ts \widehat{\vU} \bigr)
\end{equation}
with $\bs{A}(k) = B(k) \otimes \overline{B(k)}$ in complete analogy to
the Fibonacci example treated in \cite{renex}, and $B(k)$ as in
Definition~\ref{def:B-mat}; see also \cite[Prop.~4.3]{BFGR}. Several
properties of the matrices $\bs{A} (k)$ and the algebra generated by
them are collected in the Appendix.

Next, observe that each component $\widehat{\vU^{}_{\nts ij}}$ has a
unique decomposition into its pure point (\textsf{pp}) and continuous
(\textsf{c}) parts, where the support of the pure point part is (at
most) a countable set. The union of these countable sets over $i,j$
still is a countable set, and justifies the decomposition of the
measure vector $\widehat{\vU}$ as
$ \widehat{\vU} = \bigl(\widehat{\vU} \ts \bigr)_{\mathsf{pp}} +
\bigl(\widehat{\vU} \ts \bigr)_{\mathsf{c}}$.
This means that we have a decomposition
$\RR = \cE_{\mathsf{pp}} \, \dot{\cup} \, \cE_{\mathsf{c}}$ such that
\[
     \bigl(\widehat{\vU} \ts \bigr)_{\mathsf{pp}} \, = \,
     \widehat{\vU}\big|_{\cE_{\mathsf{pp}}}
     \quad \text{and} \quad 
     \bigl(\widehat{\vU} \ts \bigr)_{\mathsf{c}} \, = \,
     \widehat{\vU}\big|_{\cE_{\mathsf{c}}} \ts ,
\]    
where $\cE_{\mathsf{pp}}$ is a countable set. Without loss of
generality, we may also assume that
$\cE_{\mathsf{pp}} = f (\cE_{\mathsf{pp}} )$, for instance by
replacing $\cE_{\mathsf{pp}}$ with
$\cE^{\prime}_{\mathsf{pp}} \defeq \bigcup_{m\in\ZZ} f^m
(\cE_{\mathsf{pp}})$,
which is still a countable set, and $\cE_{\mathsf{c}}$ with
$\cE^{\prime}_{\mathsf{c}} \defeq \RR \setminus
\cE^{\prime}_{\mathsf{pp}}$,
which is then also invariant under $f$. This gives another valid
decomposition of $\RR$, which is better suited for our purposes.

Similarly, one can now further split the continuous component into its
singular continuous (\textsf{sc}) and absolutely continuous
(\textsf{ac}) parts, finally giving
\[
    \widehat{\vU} \, = \,
    \bigl(\widehat{\vU} \ts \bigr)_{\mathsf{pp}} +
    \bigl(\widehat{\vU} \ts \bigr)_{\mathsf{sc}} +
    \bigl(\widehat{\vU} \ts \bigr)_{\mathsf{ac}} \ts ,
\]
where each part is concentrated on a set that is a null set for the
other two parts. Note that we may assume, without loss of generality,
that the supporting sets are disjoint and invariant under the linear
mapping $f$. This follows constructively by an extension of our
previous argument to a decomposition of $\cE^{\prime}_{\mathsf{c}}$
into two sets, leading to
$\RR = \cE^{\prime}_{\mathsf{pp}} \, \dot{\cup} \,
\cE^{\prime}_{\mathsf{sc}} \, \dot{\cup} \,
\cE^{\prime}_{\mathsf{ac}}$ with
$\bigl(\widehat{\vU} \ts \bigr)_{\alpha} =
\widehat{\vU}\big|_{\cE^{\prime}_{\alpha}}$
and $ f (\cE^{\prime}_{\alpha}) = \cE^{\prime}_{\alpha}$ for all
$\alpha \in \{ \mathsf{pp}, \mathsf{sc}, \mathsf{ac} \}$.

\begin{lemma}\label{lem:type-recursion}
  The scaling relation \eqref{eq:FT-scaling} for\/
  $\widehat{\vU}$ holds for each of the three spectral types
  separately. In other words, one has
\[
    \bigl(\widehat{\vU} \ts \bigr)_{\mathsf{\alpha}} \, = \, 
    \myfrac{1}{\lambda^{2}} \,  \bs{A}(.)
    \bigl(f^{-1} \! . \ts \widehat{\vU} \ts\ts \bigr)_{\nts\alpha}
\]
  for each\/ $\alpha \in \{ \mathsf{pp}, \mathsf{sc},
  \mathsf{ac} \}$. Moreover, Eq.~\eqref{eq:props-two} holds
  for each spectral type separately.
\end{lemma}

\begin{proof}
  Note that the matrix function $\bs{A}(k)$ depends analytically on
  $k$, wherefore the measure vectors $\bs{A}(.) \ts \bs{\mu}$ and
  $\bs{\mu}$ are of the same spectral type; in particular,
  $\bigl(\bs{A} (.) \ts \bs{\mu}\bigr)_{\alpha} = \bs{A} (.) \bigl(
  \bs{\mu}\bigr)_{\alpha}$
  holds for each
  $\alpha \in \{ \mathsf{pp}, \mathsf{sc}, \mathsf{ac} \}$. Moreover,
  the map $\bs{\mu} \mapsto f^{-1}\! . \ts \bs{\mu}$, which is a
  simple dilation, does not change the spectral type either. The claim
  is now an exercise in restricting the measures on the left-hand and
  right-hand sides to the $f$-invariant supporting sets
  $\cE^{\prime}_{\alpha}$ together with their measure{\ts}-theoretic
  orthogonality; compare \cite[Prop.~8.4]{TAO}.
  
  The last claim follows from standard arguments.
\end{proof}

\subsection{Analysis of pure point part}\label{sec:pp-part}

Let us first take a closer look at the pure point part, which can be
written as
\[
    \bigl(\widehat{\vU} \ts \bigr)_{\mathsf{pp}} \, = 
    \, \dens (\vL) \sum_{k\in \cE^{\prime}_{\mathsf{pp}}}
    \bs{I} (k) \, \delta^{}_{k}
\]
with the (at most countable) set $\cE^{\prime}_{\mathsf{pp}}$
introduced earlier. Here, the extra factor $\dens (\vL)$ is introduced
to match the definition of $\bs{I} (k)$ as a relative (or dimensionless)
quantity with the interpretation of $\widehat{\gamma^{}_{u}} (0)$
according to Remark~\ref{rem:diffraction} and \cite[Prop.~9.2]{TAO}.
As before, we use a vector notation, with the intensity vector
$\bs{I}$, so
$\bigl(\widehat{\vU^{}_{ij}}\bigr)_{\mathsf{pp}} = \dens (\vL)
\sum_{k\in \cE^{\prime}_{\mathsf{pp}}} I_{ij} (k) \, \delta_{k}$.  As
a result of Eq.~\eqref{eq:props-two} and
Lemma~\ref{lem:type-recursion}, we have $I_{ii} (k) \geqslant 0$
together with
\[
   \overline{I_{ij} (k)} \, = \,
   I_{ji} (k) \, = \, I_{ij} (-k) \ts .
\]
A straight-forward calculation with point measures in the form of
weighted Dirac combs, in conjunction with a comparison of
coefficients, now shows that $\bs{I}$ must also satisfy the identity
\begin{equation}\label{eq:pp-recursion}
   \bs{I} (k) \, = \,  \lambda^{-2}\bs{A}(k) 
   \ts \bs{I} (\lambda\ts k)
\end{equation}
for all $k$, with the obvious understanding that we set 
$I_{ij} (k) = 0$ for any $k$ outside the supporting set 
$\cE^{\prime}_{\mathsf{pp}}$. In
particular, for $k=0$, Eq.~\eqref{eq:pp-recursion} entails the relation
\begin{equation}\label{eq:pp-at-0}
     \bs{A}(0) \ts \bs{I} (0) \, = \, \lambda^{2} \ts \bs{I} (0) \ts .
\end{equation}
Since $\bs{A}(0) = M_{\varrho} \otimes M_{\varrho}$, which has PF
eigenvalue $\lambda^{2}$, we recognise this as an eigenvalue equation
that is related to the frequencies of the prototiles, respectively the
density of the subsets of points of the corresponding type.  In fact,
the solution is unique up to an overall constant, and given by
\begin{equation}\label{eq:pp-dens}
      I_{ij} (0) 
      \, = \, \alpha^2 \dens(\vL_{i}) \ts \dens(\vL_{j}) \ts ,
\end{equation}
where $\alpha = \dens (\vL)^{-1}$ in our setting due to the definition of
the $\nu^{}_{ij} (z)$ as relative (and hence dimensionless)
frequencies.  Note that this contribution, which decouples from other
values of $k$, is always present, no matter whether our system has
non-trivial point spectrum or not.

Let us briefly look at the above scaling relations in a different
way. Defining the matrix
$\cI = (I^{}_{ij})^{}_{1\leqslant i,j \leqslant \aaa}$, one checks that
Eq.~\eqref{eq:pp-recursion} can be rewritten as
\[
    \cI (k) \, = \, \lambda^{-2}
    B(k) \, \cI (\lambda\ts k) B(k)^{\dag} ,
\]
where ${}^{\dag}$ denotes Hermitian conjugation. Clearly, this implies
the relation
\[
    \lvert \det (B(k)) \rvert^{2} \det(\cI (\lambda \ts k)) 
    \, = \, \lambda^{2 \aaa} \det (\cI (k)) \ts ,
\]
which has various consequences. In particular, $\det (\cI (k))$ will
usually vanish. Moreover, the matrix $\cI (0)$ has rank $1$, as
follows from Eq.~\eqref{eq:pp-dens}. In fact, a more general result is
true.

\begin{theorem}\label{thm:points}
  Let\/ $\varrho$ be a primitive inflation rule with hull\/ $\YY$.
  Then, for all\/ $k\in\RR$, there are
  numbers\/ $a^{}_{i} (k)$, with\/ $1 \leqslant i \leqslant d$, such
  that the identity\/
  $I_{ij} (k) = \overline{a^{}_{i} (k)} \, a^{}_{j} (k)$ holds for\/
  $1 \leqslant i,j \leqslant \aaa$. These numbers are the
  dimensionless Fourier--Bohr coefficients, defined as
\[
    a^{}_{j} (k) \, = \, \alpha \lim_{r\to\infty} \myfrac{1}{2 r}
    \sum_{x \in \vL_{j} \cap [-r + c, r + c]} \ee^{- 2 \pi \ii k x},
\]
with\/ $\alpha = \dens (\vL)^{-1}$. Here, the convergence is uniform
in\/ $c\in\RR$.  This coefficient is independent of the choice of\/
$\vL \in \YY$.  As a consequence, for all\/ $k\in \RR$, the matrix\/
$\, \cI (k)$ is Hermitian, positive semi-definite, and has rank at
most\/ $1$. Moreover, the set of\/ $k$ with\/ $a^{}_{j} (k) \ne 0$ for
some\/ $j$ is at most a countable set.
\end{theorem}

\begin{proof}
  Our system is strictly ergodic, and all eigenfunctions are
  continuous \cite{Q}. This implies the uniform existence of the
  Fourier--Bohr coefficients \cite{Hof,Lenz}, as well as their
  independence of the choice of $\vL$.

  For $I^{}_{ii} (k)$, the claim is now a consequence of
  \cite[Thm.~5]{Lenz} applied to $\delta^{}_{\! \vL_i}$, which gives
  $I^{}_{ii} (k) = \lvert a^{}_{i} (k) \rvert^2$, where $I^{}_{ii} (k) > 0$
  at most for countably many $k\in\RR$. For $I^{}_{ij} (k)$
  with $i\ne j$, we recall Eqs.~\eqref{eq:polar} and
  \eqref{eq:corr-as-conv} to obtain
\[
    \vU^{}_{ij} \, = \, \myfrac{1}{4} \sum_{\ell=1}^{4}
    \ii^{\ell} \left[\bigl( \omega^{}_{j} + \ii^{\ell}
    \omega^{}_{i} \bigr) \circledast \bigl( \omega^{}_{j}
    + \ii^{\ell} \omega^{}_{i} \bigr)^{\! \widetilde{\quad}}\right]
\]
with $\omega^{}_{i} = \delta^{}_{\! \vL_i}$. Since the contribution of
the term in square brackets to $I^{}_{ij} (k)$ is given by
$\lvert a^{}_{j} (k) + \ii^{\ell} a^{}_{i} (k) \rvert^2$, again by
\cite[Thm.~5]{Lenz}, one finds
$I^{}_{ij} (k) = \overline{a^{}_{i} (k)} \, a^{}_{j} (k)$ as claimed.
\end{proof}

\begin{remark}
  The result of Theorem~\ref{thm:points} can be viewed as a variant of
  the Bombieri--Taylor observation on the connection between
  exponential sums (or amplitudes) and intensities. Its original
  version refers to the relations
  $I_{ii} (k) = \lvert a_i (k) \rvert^2$. For primitive inflation
  rules, they extend to all components of $\cI (k)$ as stated.  \exend
\end{remark}

\begin{remark}\label{rem:phase}
  The representation of $\cI (k)$ as a rank-$1$ matrix in
  Theorem~\ref{thm:points} is only unique up to a phase, which means
  that replacing $a (k)$ by $\ee^{- 2 \pi \ii \vartheta (k)} a (k)$
  results in the same $\cI (k)$. When $\vL,\vL' \in \YY$ are
  translates of one another, say $\vL' = t+ \vL$, one has
  $\vartheta (k) = {kt}$, but more complicated phase
  functions show up in general, due to the structure of $\YY$.
  This is an interesting problem in its own right, and has been
  studied extensively in the physics literature; see
  \cite[Sec.~4.2]{Trebin} and references therein.

  Still, our recursion has the consequence that we also
  have
\[
    \lvert a_i (k) \rvert \, = \, \lambda^{-1}
    \Bigl\lvert \sum_j B_{ij} (k) \,
    \overline{a_j (\lambda k)} \Bigr\rvert ,
\]
which can be further analysed when $B(k)$ is invertible. It can
provide valuable insight on how the modulus of the amplitudes behaves
under the inward or outward iteration.  \exend
\end{remark}

\subsection{Analysis of absolutely continuous part}

Let us now take a closer look at
$\bigl( \widehat{\vU} \ts \bigr)_{\mathsf{ac}}$.  By the
Radon--Nikodym theorem, each component
$\bigl(\widehat{\vU^{}_{ij}}\bigr)_{\mathsf{ac}}$ is represented by a
measurable and locally integrable density function $h_{ij}$ relative
to Lebesgue measure, though this is generally not an element of
$L^{1} (\RR)$.

For the next result, we consider
$W \! = \CC^{\aaa} \otimes \CC^{\aaa}$ also as a \emph{real} vector
space, then of dimension $2 \ts \aaa^2$, and split it as
$W \! = W_{\! +} \oplus W_{\! -}$ into the eigenspaces of the
$\RR$-linear map $C$ on $W$ defined by
$x\otimes y \mapsto \overline{y} \otimes \overline{x}$, where
$\ts\overline{.\vphantom{a}}\ts$ is complex conjugation; see the
Appendix for more.

\begin{lemma}\label{lem:ac-rec}
  Let\/ $\bs{h}$ be the vector of Radon--Nikodym densities that
  represents\/ $\widehat{\vU}_{\mathsf{ac}}$. Then, one has the relation
\[
     \bs{h} (k) \, = \, \myfrac{1}{\lambda} \ts \bs{A} (k) 
     \, \bs{h} (\lambda\ts k) \ts ,
\]     
    which holds for Lebesgue-a.e.\ $k\in\RR$.

    Moreover, the relations\/ $h_{ij} (-k) = h_{ji} (k) = \overline{
      h_{ij} (k)}$ and\/ $h_{ii} (k) \geqslant 0$ hold for a.e.\
       $k\in\RR$ and all\/ $1 \leqslant i,j \leqslant \aaa$. In
    particular, $\bs{h} (k) \in W_{\! +}$ for a.e.\ $k\in\RR$.
\end{lemma}

\begin{proof}
  Let $g$ be a continuous function with compact support in $\RR$.
%  $S_{g} \subset \RR$.  
  We have to determine what the relation from
  Lemma~\ref{lem:type-recursion}, for $\alpha = \textsf{ac}$, implies
  for the Radon--Nikodym densities of
  $\bs{\mu} \defeq \widehat{\vU}_{\mathsf{ac}}$.  The
  left-hand side of Lemma~\ref{lem:type-recursion} clearly reads
\[
     \bs{\mu} (g) \, = \int_{\RR} g(k) \dd \bs{\mu} (k) \, =
     \int_{\RR} g(k)  \, \bs{h}(k) \dd k \ts .
\]
With $f(k) = \lambda\ts k$ as before, we can calculate the right-hand
side of Lemma~\ref{lem:type-recursion} as follows,
\[
\begin{split}
    \lambda^{-2}\bigl( \bs{A} (.) \, 
    (f^{-1} \! . \ts \bs{\mu}) \bigr) (g) \, & = \, 
      \lambda^{-2} \int_{\RR} g\bigl(f^{-1} (k)\bigr) 
      \, \bs{A} \bigl(f^{-1} (k)\bigr) \dd \bs{\mu} (k)\\[1mm]
    &  = \, \lambda^{-2} \int_{\RR}  g  
      \bigl( \tfrac{k}{\lambda} \bigr) 
      \bs{A} \!  \bigl( \tfrac{k}{\lambda} \bigr)\, \bs{h} (k) \dd k 
      \, = \, \myfrac{1}{\lambda} \int_{\RR} g(k) \, \bs{A} (k) \,
        \bs{h} (\lambda \ts k) \dd k  \ts ,
\end{split}
\]
where the first step is just working out the definition of the measure
on the left, while the last is the result of a change of variable
transformation. This is the important step, as it leads to the
cancellation of one factor of $\lambda$ in the denominator. Comparing
the two expressions above, and observing that the test function $g$
was arbitrary, leads to the first claim by means of standard arguments
for densities.

The second claim is a consequence of Eq.~\eqref{eq:props-two} and
Lemma~\ref{lem:type-recursion}
for the absolutely continuous parts, in conjunction with the
properties of the correlation measures $\vU^{}_{\nts ij}$ from
Eq.~\eqref{eq:props-one}. Since the action of the mapping $C$ on
$\bs{h}$ is given by
$ (C \bs{h})^{}_{ij} (k) = \overline{h^{}_{ji} (k)}$, the previous
claim implies $(C\bs{h}) (k) = \bs{h} (k)$ for a.e.\ $k\in\RR$, hence
$\bs{h} (k) \in W_{\! +}\ts $ as claimed.
\end{proof}

To continue, we need the following well-known decomposition property
of positive semi-definite, Hermitian matrices, which we prove for
convenience.

\begin{fact}\label{fact:decompose}
  Let\/ $H = (h_{ij})^{}_{1\leqslant i,j \leqslant d}\in \Mat (d,\CC)$
  be Hermitian and positive semi-definite, with rank\/ $m$.  Then, all
  diagonal elements of\/ $H$ are non-negative. If\/ $h_{ii} =0$ for
  some\/ $i$, one has\/ $h_{ij} = h_{ji} = 0$ for all\/
  $1 \leqslant j \leqslant d$.  
    
  Whenever\/ $H\ne 0$, there are\/ $m\geqslant 1$ Hermitian, positive
  semi-definite matrices\/ $H_1, \ldots , H_m$ of rank\/ $1$ such
  that\/ $H = \sum_{r=1}^{m} H_r$ together with\/ $H_r \ts H_s = 0$
  for\/ $r\ne s$.
\end{fact}

\begin{proof}
  By Sylvester's criterion, $H$ positive semi-definite means that all
  principal minors are non-negative, hence in particular all diagonal
  elements of $H$. Assume $h_{ii}=0$ for some $i$, and select any
  $j\in \{ 1, \ldots, d\}$.  By semi-definiteness in conjunction with
  Hermiticity, one finds
\[
     0 \, = \, h_{ii} \ts h_{jj} \, \geqslant \, h_{ij} \ts h_{ji}
     \, = \, \lvert h_{ij} \rvert^2  \, \geqslant \, 0 \ts ,
\]  
 which implies the second claim. 
   
 Employing Dirac's notation, the spectral theorem for Hermitian
 matrices asserts that one has
 $H = \sum_{i=1}^{d} | v_i \rangle \ts \lambda_i \ts \langle v_i |$,
 where the eigenvectors $|v_i \rangle$ can be chosen to form an
 orthonormal basis (so $\langle v_i | v_j \rangle = \delta_{i,j}$ and
 $|v_i\rangle \langle v_i |$ is a projector of rank $1$), while all
 eigenvalues are non-negative due to positive semi-definiteness. The
 rank of $H$ is the number of positive eigenvalues, counted with
 multiplicities. Ordering the eigenvalues decreasingly as
 $\lambda^{}_{1} \geqslant \lambda^{}_{2} \geqslant \cdots \geqslant
 \lambda^{}_{d} \geqslant 0$,
 one can choose $H_r = | v_r \rangle \ts \lambda_r \ts \langle v_r|$
 for $1 \leqslant r \leqslant m$, and the claim is obvious.
\end{proof}

In order to profit from Lemma~\ref{lem:ac-rec}, we now perform a
\emph{dimensional reduction} as follows.\footnote{As we outline in
  more detail in the Appendix, one can alternatively work with the
  recursion from Lemma~\ref{lem:ac-rec} directly. However, the
  dimensional reduction leads to a stronger result in the sense that
  we also get a representation of $\bs{h}$ that resembles the
  situation of the pure point part.} Define the matrix
$\cH (k) = \bigl( h_{ij} (k) \bigr)_{1 \leqslant i,j \leqslant \aaa}$,
which is Hermitian and positive semi-definite, for a.e.\
$k\in\RR$. Simply switching from a notation with vectors of length
$\aaa^2$ to a version with $\aaa \! \times \! \aaa$ matrices, the
renormalisation relation can be rewritten as
\begin{equation}\label{eq:H-iter}
    \cH (k) \, = \, \lambda^{-1} B(k) \ts \cH 
    (\lambda k) B^{\dagger} (k) \ts .
\end{equation}
If we decompose $\cH (k) =\sum_{i=1}^{m} \cH_i (k)$ as a sum of
Hermitian, positive semi-definite matrices of rank~$1$ according to
Fact~\ref{fact:decompose}, every term is of the form
\begin{equation}\label{eq:H-decompose}
  \cH^{}_i (k) \, = \,
  v^{(i)} (k) \, \bigl(v^{(i)}\bigr)^{\dagger} (k)
\end{equation}  
where each $v^{(i)} (k)$ is a vector of functions from
$L^2_{\mathrm{loc}} (\RR)$.  Now, one has
\[
  B(k) \cH (k) B^{\dagger} (k) \, =
  \sum_{i=1}^{m} B(k) \cH_i (k) B^{\dagger} (k) \ts ,
\]
and we can study the simpler iteration
\[
    v (k) \, = \, \myfrac{1}{\sqrt{\lambda}}
       \, B(k) \ts\ts v (\lambda k)
\]
instead of \eqref{eq:H-iter}. Iterating this relation, for a.e.\
$k\in\RR$, yields
\begin{equation}\label{eq:inward}
  v \Bigl( \myfrac{k}{\lambda^n} \Bigr) \, = \,
  \myfrac{1}{\lambda^{n/2}} \, B \Bigl(
  \myfrac{k}{\lambda^{n-1}} \Bigr) \cdots
  B \Bigl( \myfrac{k}{\lambda} \Bigr) v(k) \ts ,
\end{equation}
which we call the \emph{inward iteration} for $v(k)$, with given
$k\in\RR$.  When $B(k)$ is invertible, at least for a.e.\ $k\in\RR$,
we get
\[
     v (\lambda k) \, = \, \sqrt{\lambda} \, B^{-1} (k) \ts\ts v(k)
\]
and the corresponding \emph{outward iteration},
\begin{equation}\label{eq:iter-out}
     v ( \lambda^n k) \, = \, \lambda^{n/2}
     B^{-1} (\lambda^{n-1} k) \cdots B^{-1} (\lambda k)
     B^{-1} (k) \ts\ts v (k) \ts ,
\end{equation}
which holds for a.e.\ $k\in\RR$ and is of particular interest to
us. For the existence of the matrix inverses on a set of full measure,
it suffices that $\det \bigl( B(k)\bigr) \ne 0$ for \emph{some}
$k\in\RR$, because the determinant is an analytic function in $k$ and
thus can then at most have isolated zeros.

This way, we can consider the corresponding matrix cocycle, whose
Lyapunov exponents determine the asymptotic growth behaviour of $v$
for $k\to\infty$. It is sufficient to look at the extremal ones, which
are derived from the $B^{(n)} (k)$ previously defined in
Eq.~\eqref{eq:B-powers} and Fact~\ref{fact:F-matrix}.  Writing out
$B^{(n)} (k)$ and its inverse, these extremal exponents are given by
\[
\begin{split}
   \chi^{}_{\max} (k) \, & = \, \log\sqrt{\lambda} \, + \,
       \limsup_{n\to\infty} \myfrac{1}{n} \log \big\|
       B^{-1} (\lambda^{n-1} k) \cdots B^{-1} (k) \big\| \\
   \chi^{}_{\min} (k) \, & = \, \log\sqrt{\lambda} \, + \,
       \liminf_{n\to\infty}\myfrac{1}{n} \log \big\|
       B (k) B(\lambda k) \cdots B(\lambda^{n-1} k) \big\|^{-1} ;
     \end{split}
\]
compare \cite[Eq.~2.2]{Viana}.  This follows from the general
definition, compare \cite{Viana}, applied to the iteration from
\eqref{eq:iter-out} by a simple calculation, where the additional
logarithmic term reflects the multiplication by $\sqrt{\lambda}$ in
each iteration step.  Now, one also obtains
\[
     \chi^{}_{\min} (k) \, = \, \log\sqrt{\lambda} \, - \, \chi^{B} (k)
\]
where
\begin{equation}\label{eq:chi-B}
    \chi^{B} (k) \, \defeq \, \limsup_{n\to\infty} \frac{1}{n}
    \log \big\| B^{(n)} (k) \big \|
\end{equation}
is the maximal Lyapunov exponent of the Fourier matrix cocycle,
$B^{(n)} (k)$. This way, we can formulate one of our central results
in terms of the asymptotic behaviour of the Fourier matrices of
$\varrho^n$ for large $n$ as follows, which is the expected extension
of the result for a binary example from \cite{BFGR} to general
alphabets, but requires a slightly different proof. Also, we state
it in terms of the Fourier matrix cocycle, $B^{(n)} (k)$, as this
seems the most natural object.

\begin{theorem}\label{thm:1D}
  Let\/ $\varrho$ be a primitive inflation rule, with inflation
  multiplier\/ $\lambda$, and let\/ $B(k)$ be the corresponding
  Fourier matrix, with\/ $\det (B(k)) \ne 0$ for some\/ $k\in\RR$.  If
  there is an\/ $\varepsilon > 0$ such that\/
  $\chi^{B} (k) \leqslant \log \sqrt{\lambda} - \varepsilon$ holds for
  a.e.\ $k\in\RR$, where\/ $\chi^{B} (k)$ is the maximal Lyapunov
  exponent of the Fourier matrix cocycle, the diffraction measure of
  the system cannot have an absolutely continuous part.
\end{theorem}

\begin{proof}
  If such an $\varepsilon>0$ exists, there is a
  $0 < \delta \leqslant \varepsilon $ such that
  $\| v ( \lambda^n k)\| $, for a.e.\ $k$, is bounded below by
  $C \ts \ee^{\delta n }$ as $n \to\infty$, where the constant can
  depend on $k$, but is always positive, so this lower bound grows
  exponentially. This allows an application of \cite[Lemma~9.3]{BFGR}:
  Translation boundedness of the diffraction measure means
  $\int_{0}^{a \lambda^m} \| v(x) \|^{2}_{2} \dd x = \cO (a
  \lambda^m)$ for any $a>0$ as $m\to\infty$, while a positive exponent
  means that there is some $\eta > 1$ with
  $\int_{0}^{a \lambda^m} \| v (x) \|^{2}_{2} \dd x \geqslant c^{}_{v}
  \ts ( \eta \lambda )^m$. These conflicting conditions are compatible
  only when $\| v \|^{2}_{2} = 0$, so $v (k) = 0$ for a.e.\ $k\in\RR$.

  For any choice of the weight vector $u \in \CC^{\aaa}$, the
  absolutely continuous part of the corresponding diffraction measure,
  $\widehat{\gamma}^{}_{\mathsf{ac}}$, is a translation-bounded
  measure. If $u^{}_{j} = \delta^{}_{j,\ell}$ for some fixed index
  $\ell$, the Radon--Nikodym density of
  $\widehat{\gamma}^{}_{\mathsf{ac}}$ is the locally integrable
  function $h^{}_{\ell\ell} \geqslant 0$.  Since a finite sum of
  translation-bounded measures is still translation bounded,
  $h_s \defeq \sum_{\ell=1}^{\aaa} h^{}_{\ell\ell}$ represents a
  translation-bounded, positive measure.

Now, by Fact~\ref{fact:decompose} in conjunction with
Eqs.~\eqref{eq:H-iter} and \eqref{eq:iter-out}, there is an
integer $m\leqslant \aaa$ such that $h_s$ is of the form
\[
     h_s (k) \, = \sum_{\ell=1}^{\aaa} \sum_{i=1}^{m}
     \big\lvert v^{(i)}_{\ell} (k) \big\rvert^2 .
\]
Since each summand is non-negative, there can be no cancellation
between the terms, and the exponential growth of any of them would
violate translation boundedness as explained above. Consequently, we
must have $v^{(i)} (k) = 0$ for a.e.\ $k\in\RR$ and for all
$1 \leqslant i \leqslant m$, and hence
$\widehat{\vU}_{\mathsf{ac}} = 0$. This implies
$\widehat{\gamma}^{}_{\mathsf{ac}} = 0$ as claimed.
\end{proof}

In fact, as shown in \cite{BFGR}, $\chi^{}_{\min} (k) >0$ for a
subset of full measure of some interval of the form 
$\bigl[ \frac{\varepsilon}{\lambda}, \varepsilon\bigr]$ with
$\varepsilon >0$ is already enough to rule out an absolutely
continuous diffraction component. This leads to the following
consequence.

\begin{coro}\label{coro:1D}
  If\/ the primitive inflation rule\/ $\varrho$, with inflation
  multiplier\/ $\lambda$ and the determinant condition on\/ $B(k)$ as
  before, leads to a system that displays a non-trivial diffraction
  component of absolutely continuous type, one must have\/
  $\chi^{}_{\min} (k) \leqslant 0$ on a subset of\/ $\RR$
  of positive measure. When\/ $\chi^{}_{\min} (k)$ is constant
  for a.e.\ $k\in\RR$, this constant must be\/ $\leqslant 0$.
  \qed
\end{coro}

This means that we are in a particularly good situation whenever
Oseledec's theorem applies, see \cite[Ch.~4]{Viana} for backgropund,
as is the case for constant-length substitutions.

\begin{remark}
  In \cite{BS2}, a variant of the cocycle $B^{n}(k)$ for $S$-adic
  systems is considered, called the \emph{spectral cocycle}. There, the
  main objects are spectral measures $\sigma^{}_{\nts f}$ of a certain
  class of functions associated to a suspension flow.  One of the main
  results is the expression of the lower local dimension
  $\underline{d} {\ts} (\sigma^{}_{\nts f}, x)$ in terms of the
  Lyapunov exponent of the spectral cocycle.  It is shown that the
  singularity of these measures is implied by the Lyapunov exponent
  being strictly bounded from above by $\log\sqrt{\lambda_{\bs{a}}}$
  for a.e.~$k\in\mathbb{R}^{d}$, see \cite[Cor.~2.4]{BS2}, where
  $\lambda_{\bs{a}}$ is the generalisation of the PF eigenvalue for
  $S$-adic systems (which is actually also defined as a Lyapunov
  exponent).

  Flows generated by deterministic substitutions form a subclass of
  $S$-adic systems, for which a dynamical version of the singularity
  result in Theorem.~\ref{thm:1D} for almost every choice of roof
  function is shown in \cite[Cor.~2.6]{BS2}. The one-dimensional
  systems we cover in this work pertain to the specific case where the
  roof function is given by the left PF eigenvector of the
  substitution matrix of $\varrho$.  \exend
  \end{remark}

\subsection{Further consequences}

In theory, the Lyapunov exponent could still be negative, thus
signifying an exponential decay in the outward direction. However,
since the Radon--Nikodym densities $\bs{h}(k)$ can be recovered from the
vectors $v(k)$, one can rule out the existence of negative exponents
via measure{\ts}-theoretic arguments when the elements $\vL$ of
the geometric hull $\YY$ satisfy some additional properties.

\begin{prop}\label{prop:no-neg}
   Assume that the elements of\/ $\YY$ are Meyer sets, and 
   assume that\/ $\widehat{\gamma}^{}_{\mathsf{ac}} \ne 0$.
   Then, its Radon--Nikodym density cannot decay at infinity.
\end{prop}

\begin{proof}
  When the point set $\vL \in \YY$ under consideration is a Meyer set,
  it was recently shown by Strungaru \cite{S} that the absolutely
  continuous part of its diffraction measure is of the form
  $\widehat{\gamma}^{}_{\mathsf{ac}} = \widehat{\mu}$ where $\mu$ is a
  pure point measure with Meyer set support. Moreover, it follows from
  \cite[Cor.~11.1]{GLA} that $\widehat{\mu}$ is a strongly almost
  periodic measure. Consequently, for any $g\in C_{\mathsf{c}} (\RR)$,
  the convolution $f\defeq\widehat{\mu}* g$ is a Bohr almost periodic
  function. By starting from a non-negative $g$ with sufficiently
  small support, we make sure that $f\ne 0$.
   
  If the Radon--Nikodym density $h$ of $\widehat{\mu}$ decays at
  infinity, then so does $\widehat{\mu}$ as a measure, and hence also
  $f$.  On the other hand, for any $\varepsilon >0$, the
  $\varepsilon$-almost periods of $f$ are relatively dense. This only
  leaves $f=0$, which is a contradiction, and $h$ cannot decay at
  infinity.
\end{proof}

For primitive inflation rules that lead to a hull $\YY$ of Meyer sets
with non-trivial absolutely continuous diffraction,
we can neither have $\chi^{}_{\min} < 0$, by
Proposition~\ref{prop:no-neg}, nor $\chi^{}_{\min} >0$, by
Corollary~\ref{coro:1D}, which has the following rather strong
consequence. For its formulation, let $\chi^{B} (v,k) \defeq
\limsup_{n\to\infty} \frac{1}{n} \log \| B^{(n)} (k) \ts v \|$.

\begin{coro}\label{coro:Pisot}
  Assume that the hull\/ $\YY$ contains Meyer sets only, and that
  the diffraction measure\/ $\widehat{\gamma}^{ }_{\mathsf{ac}}$ is
  nontrivial. Then, $\chi^{B} (v,k) = \log\sqrt{\lambda}$ holds for
  a set of\/ $k \in \RR$ of positive measure, for all\/
  $v = v^{(i)} (k)$ as in\/ $\cH_i$ from Eq.~\eqref{eq:H-decompose}. In
  particular, this criterion applies whenever the inflation
  multiplier\/ $\lambda$ of the primitive substitution is a PV number.
  Whenever\/ $\chi^{B} (v,k)$ is constant a.e., this constant must
  be\/ $\log \sqrt{\lambda}$.
\end{coro}

\begin{proof}
  While the first claim is clear from the previous arguments, the
  second follows from the fact that the geometric realisation of a
  primitive Pisot (or PV) substitution always leads to a geometric
  hull with the Meyer property.

  Let us be more precise with the last point. Though we still do not
  know whether the Pisot substitution conjecture holds, we do know
  that the fixed point of a primitive PV inflation is a relatively
  dense subset of a model set, and hence a Meyer set; see \cite{Bernd}
  and references therein for the details. Since every hull of a
  primitive PV inflation is minimal and can be generated as the orbit
  closure of a fixed point (possibly under some power of the inflation
  rule), compare \cite[Sec.~4.2]{TAO}, the corresponding hull consists
  of Meyer sets only.

  The last claim of the corollary is clear.
 \end{proof}

In favourable situations, some of which will be discussed in the next
section, one can calculate the Lyapunov exponents explicitly, and then
apply Theorem~\ref{thm:1D} directly. Otherwise, one can look for upper
bounds to $\chi^{B} (k)$ in order to establish the estimate needed in
Theorem~\ref{thm:1D}.  Let us briefly explain one particularly useful
method that is based on a subadditivity argument; compare
\cite{Schmeling}.

Since the norm in Eq.~\eqref{eq:chi-B} is arbitrary, we may choose a
submultiplicative one, such as the spectral norm or the Frobenius norm.
Then, for any $m,n\in\NN$, one has
\[
    \| B^{(m+n)} (k) \| \, = \, \| B^{(m)} (k) B^{(n)} (\lambda^m k) \| 
    \, \leqslant \, \| B^{(m)} (k) \| \, \| B^{(n)} (\lambda^{m} k) \|
\]
and thus, with $L_n (k) \defeq \log \| B^{(n)} (k) \|$, also
\[
   L_{m+n} (k) \, \leqslant \, L_{m} (k) + L_{n} (\lambda^m k) \ts .
\]
Using arithmetic progressions, which correspond to averages along
sequences that are totally Bohr ergodic in the sense of
\cite{Schmeling}, one can invoke a standard argument known from
Fekete's subadditive lemma. Indeed, for a fixed $N\in\NN$, one has
\[
  \frac{L^{}_{q N +r} (k)}{q N + r} \, \leqslant \,
  \frac{q}{qN+r}  \biggl( \frac{1}{q} \sum_{\ell=0}^{q-1} L^{}_{N}
  (\lambda^{\ell N}k) \biggr)
   + \frac{L_r (\lambda^{q N} k)}{qN+r} \ts ,
\]
where $q\in\NN$ and $0\leqslant r < N$. As $q\to\infty$, under mild
conditions that are satisfied in our case, the first term with the
Birkhoff sum converges to the mean, namely $\frac{1}{N} \MM (L_N)$ as
detailed below, while the second converges to zero, both for a.e.\
$k\in\RR$.  One thus obtains the following bound for the maximal
Lyapunov exponent of $B^{(n)}$.  For further details of the proof, we
refer to \cite[Lemma~6.16]{BFGR} and the treatment in \cite{BHL}.

\begin{lemma}\label{lem:bound}
  Let\/ $B^{(n)} (.)$ be the Fourier matrix cocycle of a primitive
  inflation rule with inflation multiplier\/ $\lambda > 1$, and let\/
  $L_n (k) = \log \| B^{(n)} (k) \|$, for every\/ $n\in\NN$, be Bohr
  almost periodic. Then, for any\/ $N\in\NN$ and a.e.\ $k\in\RR$, one
  has
\[
   \chi^{B} (k) \, = \,
   \limsup_{n\to\infty} \myfrac{1}{n} \ts L_n (k)
   \, \leqslant \,  \myfrac{1}{N} \MM (L^{}_{\nts N}) \ts ,
\]
where\/ $\MM (f) \defeq \lim_{T \to\infty} \frac{1}{T} \int_{0}^{T}
f(t) \dd t$ is the \emph{mean} of the function $f$.  \qed
\end{lemma}

Note that the mean of an almost periodic function always exists, and
that Bohr almost periodicity can be relaxed by an application of
Sobol's theorem; compare \cite{BHL}. In the more restrictive situation
that $\lambda$ is an integer and $L^{}_{\nts N}$ is thus $1$-periodic,
one can determine the mean via Birkhoff's ergodic theorem
\cite[Thm.~2.30]{EW} whenever $L^{}_{\nts N}$ is integrable in the
Lebesgue sense, as explained in some detail in \cite[Sec.~6.3]{BHL}.

In general, observing that $L^{}_{\nts N} $ is a quasiperiodic
function, its mean can be evaluated as an integral over the $D$-torus,
where $D$ is the algebraic degree of $\lambda$. Indeed, each matrix
entry $B_{ij} (k)$ is quasiperiodic, and can thus be represented as a
section through a function $P_{ij} \bigl( \tilde{k} \bigr)$ that is
$1$-periodic in each of the (possibly several) variables; compare
\cite[Secs.~3.1 and 5.2]{BGM} for details. This works simutaneously
for the entire matrix $B(k)$, and analogously for $B^{(N)} (k)$.
Employing the Frobenius norm, one then has
\begin{equation}\label{eq:mean}
   \myfrac{1}{N}\ts \MM \bigl( \log \| B^{(N)} (.) 
   \|^2_{\mathrm{F}} \bigr)
   \, = \, \myfrac{1}{N} \int_{\TT^D} \log \Bigl(
   \sum_{i,j=1}^{\aaa} \big\lvert P^{(N)}_{ij} 
   \bigl( \tilde{k} \bigr)
   \big\rvert^2 \Bigr) \dd \tilde{k} \ts ,
\end{equation}
where the $1$-periodic trigonometric polynomials $P^{(N)}_{ij}$ are
the entries of $B^{(N)}$. Effectively, one can now calculate the right-hand
side of the upper bound from Lemma~\ref{lem:bound} for increasing $N$,
and test this against the threshold value of $\log \sqrt{\lambda}$
from Theorem~\ref{thm:1D}. It turns out that this works quite well in
concrete examples; see \cite{BGM,BG-block} and
Sections~\ref{sec:abelian} and \ref{sec:higher}.

It is reasonable to expect that
$\chi^{B} (v, k) \leqslant \log\sqrt{\lambda}$ for a.e.\ $k\in \RR$
holds under more general circumstances, that is, possibly for all $v$,
and also beyond the Meyer set case.

\begin{theorem}\label{thm:geq}
   Let\/ $\varrho$ be a primitive inflation rule, with multiplier\/
   $\lambda$ and Fourier matrix\/ $B (k)$. Then, for a.e.\ $k\in\RR$,
   one has\/ $\chi^{}_{\min} (k) \geqslant 0$ or, equivalently, that\/
   $\chi^{B} (k) \leqslant \log \sqrt{\lambda}$.
\end{theorem}

\begin{proof}
   Choosing the Frobenius norm for convenience,
   the mean of $\log \| B^{(N)} \|^{}_{\mathrm{F}}$ can be calculated
   as explained above around Eq.~\eqref{eq:mean}.
   Now, via Jensen's inequality, we get
\[
    \exp \Bigl( \MM \bigl( \log \| B^{(N)} (.) 
    \|^2_{\mathrm{F}} \bigr) \Bigr) \, \leqslant 
    \int_{\TT^D} \sum_{i,j=1}^{\aaa} \big\lvert
    P^{(N)}_{ij} \bigl( \tilde{k} \bigr) \big\rvert^2 
    \dd \tilde{k} \, = \sum_{i,j} \big\| P^{(N)}_{ij} 
    \big\|^{2}_{2} \, = \sum_{i,j} \bigl( M^N \bigr)_{ij}
\]
with $M$ being the substitution matrix of $\varrho$. The last step
follows from Parseval's equation and the observation that the
coefficients of the trigonometric polynomials can only be $0$ or $1$,
due to the nature of the control points (see below for more), and
$B^{(N)} (0) = M^N$.

Now, since $\varrho$ is primitive, we know that
$\bigl( M^N \bigr)_{ij} \sim C \, \lambda^N v^{}_{i}\, u^{}_{j}$ holds
simultaneously for all $i,j$ as $N\to\infty$, where $C>0$ is some
constant, and $u$ and $v$ are the left and right PF eigenvectors of
$M$, both strictly positive and conveniently normalised (via
$\sum_i v^{}_{i} = \sum_i u^{}_{i} \ts v^{}_{i} = 1$, say). In fact,
the error term of this estimate is exponentially small because all
other eigenvalues of $M$ are strictly smaller than $\lambda$ in
modulus. This gives
\[
     \myfrac{1}{N}\ts \MM \bigl( \log \| B^{(N)} (.) 
   \|^2_{\mathrm{F}} \bigr)
   \, \leqslant \, \myfrac{1}{N} \log \bigl( C'
   \lambda^N \bigr) \, = \, \log (\lambda) +
   \myfrac{1}{N} \log (C')
\]
for some $C' >0$ and all sufficiently large $N$. Consequently,
we also have
\[
    \chi^{B} (k) \, \leqslant \, \liminf_{N\to\infty}
    \myfrac{1}{2 N} \ts \MM \bigl( \log \| B^{(N)} (.) 
   \|^2_{\mathrm{F}} \bigr) \, \leqslant \,
   \myfrac{1}{2} \log (\lambda) \, = \, 
   \log \sqrt{\lambda}
\]
from Lemma~\ref{lem:bound} for a.e.\ $k\in\RR$ as claimed.
\end{proof}

The previous theorem is a strong almost everywhere result that has to
be satisfied by generic primitive substitutions.  This, together with
Theorem~\ref{thm:1D}, provides a necessary criterion for the presence
of AC spectral components as follows. 

\begin{coro}\label{coro:ac-cond}
  Let\/ $\varrho$ be a primitive inflation rule, with multiplier\/
  $\lambda$ and Fourier matrix\/ $B (k)$, where we assume that\/
  $\det (B(k))\ne 0$ for some\/ $k$. Let\/ $B^{(n)} (k)$ be the
  Fourier matrix cocycle as before. If the diffraction measure of the
  hull defined by\/ $\varrho$ comprises a non-trivial absolutely
  continuous component, one has\/ $\chi^{B} (k) = \log \sqrt{\lambda}$
  for a subset of\/ $\RR$ of positive measure, or for a.e.\
  $k\in\RR$ whenever\/ $\chi^{B} (k)$ is almost surely constant.  \qed
\end{coro}

To date, few examples are known where a non-trivial AC component
occurs at all, namely the Rudin--Shapiro sequence \cite{Q,TAO} and its
relatives, which all have a similar structure \cite{NF-Hadamard}. To
the best of our knowledge, no example outside the class of
constant-length substitutions (and their topological conjugates, which
need not be of constant length) is known.

Let us briefly compare Corollary~\ref{coro:ac-cond} with another
necessary criterion, taken from \cite{BS}, for constant-length
substitutions, which reads as follows.

\begin{theorem}[{\cite[Thm.~1.1]{BS}}]\label{thm:Boris}
  Let\/ $\varrho$ be a primitive substitution of constant length,
  and let\/ $M$ be its substitution matrix, with PF eigenvalue\/
  $\lambda$. If\/ $\widehat{\gamma}^{}_{\mathsf{ac}} \ne 0$, there
  is an eigenvalue\/ $\alpha$ of\/ $M$ with\/ $\lvert \alpha \rvert
  = \sqrt{\lambda}$.  \qed
\end{theorem}  

In hindsight, this criterion is a restriction on the Lyapunov
exponents for the \emph{inward} iteration from Eq.~\eqref{eq:inward}.
What we have in Corollary~\ref{coro:ac-cond}, under the
non-degeneracy assumption for $\det(B(k))$, is a condition on the
\emph{outward} iteration --- hence a criterion that applies to
primitive inflations in general.  Taken together, this suggests even
more that absolutely continuous components are only possible under
very restrictive conditions.

\section{Abelian bijective substitutions}\label{sec:abelian}

It is natural to ask whether the conditions in Theorem~\ref{thm:1D}
can be confirmed on a larger scale; that is, whether it can be proved
that such a bound exists for an entire class of substitutions without
computing the exponents explicitly. Fortunately, this is the case for
a specific class, which we elaborate here. In particular, we will show
the following result.

\begin{theorem}\label{thm:pos-abelian}
  Let\/ $\varrho$ be a primitive, bijective constant-length
  substitution that is aperiodic and whose IDA\/ $\cB$ is
  Abelian. Then, all Lyapunov exponents of the outward iteration
  \eqref{eq:iter-out} of\/ $\varrho$ are strictly positive, and the
  geometric realisation of the inflation tiling has singular
  diffraction. Moreover, all spectral measures of the dynamical
  spectrum are singular.
\end{theorem}

Throughout this section, we assume $\varrho$ to be an aperiodic,
primitive, bijective substitution of length $L$ on an $\aaa$-letter
alphabet $\cA^{}_{\aaa}$, with corresponding Fourier matrix $B ( k )$
and associated IDA $\cB$.  Due to Fact~\ref{fact:IDA}, we make no
distinction between $\cB$ and the algebra generated by the digit
matrices $\{ D_{x} \mid x\in S_T \} $, which, in this setting, is the
algebra generated by the matrix representation of the permutations
$ \{ g^{}_{0}, g^{}_{1}, \ldots , g^{}_{L-1} \}$, hence
\[
   \cB \, = \, \big\langle \{D_{x} \mid x \in S_T \} \big\rangle
   \, = \, \big\langle \varPhi (G ) \big\rangle  ,
\]
where $g^{}_r$ is the inverse of the $r$-th column of the word vector
$\bigl(\varrho (a_i)\bigr)_{1 \leqslant i \leqslant \aaa}$, viewed as
an element of the permutation group $\vS^{}_{\aaa}$ of $\aaa$
elements,
$G= \langle g^{}_{0}, g^{}_{1}, \ldots , g^{}_{L-1} \rangle $ and
$\varPhi$ is the canonical representation via permutation
matrices. This follows since $P^T=P^{-1}$ for any permutation matrix
$P$, and hence
$D_x=\ \bigl(\varPhi (g^{}_x )\bigr)^T = \varPhi (g^{-1}_x ) $.  We
call $G$ a \emph{generating subgroup} for the algebra $\cB$, where it
is understood that $G$ is a subgroup of $\vS^{}_{\aaa}$.  The
primitivity condition on $\varrho$ translates to a condition on its
generating subgroup as follows.

\begin{lemma}\label{lem:generating}
  Any  generating subgroup\/  $G$ for\/ $\cB$ must
  be a transitive subgroup of\/ $\vS^{}_{\aaa}$.
\end{lemma}

\begin{proof}
Assume that $G$ is \emph{not} transitive. Then, there
are $a^{}_i, a^{}_{j} \in \cA^{}_{\aaa}$ such that $\sigma (a_i) = a_j$
cannot hold for \emph{any} $\sigma\in G$. Consequently, the
representation matrices will be $0$ in positions $i,j$, as are
all linear combinations of them, and hence all elements of
$\cB$ by Fact~\ref{fact:IDA}.

Now, this implies that  $a_j$ can never appear in any word of
the form $\varrho^n (a_i)$ with $n\in \NN$. This contradicts the 
assumed primitivity of $\varrho$, and our claim follows.
\end{proof}

The following property of Abelian subgroups of $\vS^{}_{\aaa}$ is well
known; see \cite[Cor.~10.3.3 and Thm.~10.3.4]{Scott}.

\begin{fact}\label{fact:trans}
  Any transitive Abelian subgroup of\/ $\vS^{}_{n}$
  must be of order\/ $n$. So, if\/ $G$ is an Abelian subgroup
  of\/ $\vS^{}_{\aaa}$ that is generating for the IDA\/ $\cB$ of\/ 
  $\varrho$, it must be of order\/ $\aaa$.   \qed
\end{fact}

Bijective substitutions  have a rich structure due to the algebraic 
properties of their columns. These can be exploited to shed light 
on the multiplicity and mutual singularity of the spectral measures 
of the associated dynamical system; see \cite{Bart,Q}.
When the generating group is Abelian, these measures can be written 
down explicitly as Riesz products of polynomials arising from the 
characters $\rho\in \widehat{G}$ evaluated on the columns of $\varrho$. 
The following important result (actually, also its higher-dimensional 
analogue) was outlined in \cite{Q}, and was formally proved 
in \cite{Bart}. For binary block substitutions, it also follows
from \cite{NF,Nat-review}, and it was shown in \cite{squiral} by 
a different method.

\begin{theorem}[{\cite[Thm.~4.19]{Bart}}]\label{thm:Bart}
  Any primitive, bijective constant-length substitution that is
  aperiodic and Abelian has purely singular dynamical spectrum.  \qed
\end{theorem}

In what follows, we prove that Theorem~\ref{thm:1D} holds for this
class, thus giving an independent proof of Theorem~\ref{thm:Bart} by
yet another method, which also extends to higher-dimensional block
substitutions \cite{BG-block}.  Note that we impose no assumptions on
the length or the height of $\varrho$.

\begin{proof}[Proof of Theorem~\textnormal{\ref{thm:pos-abelian}}]
  By assumption, the generating subgroup $G$ is Abelian, and all digit
  matrices $D_x$ commute with one another. Being permutation matrices,
  they are thus simultaneously diagonalisable, by a unitary matrix $U$
  say. In this case, the diagonal entries of $U D_x \ts U^{-1}$
  are values of characters of $G$, written as $\rho_i (g)$. Note that
  the $\rho_i$ are the irreducible representations of $G$ because the
  latter is Abelian.
  
  The matrix representation we start from is completely reducible by
  standard results \cite{JL}, as a sum of one-dimensional irreducible
  representations, which are the group characters in this case. In its
  diagonalised version, all values that occur are thus roots of unity,
  so $\lvert \rho_i (g) \rvert = 1$ for all $i$ and all $g\in G$. If
  $\varrho$ is of length $L$, the eigenvalues of $B (k)$ are then of
  the form
\[
      \beta_j  (k) \, = \sum_{m=0}^{L-1} 
      \, \overline{\rho_j (g_m)} \, u^m \ts ,
\]  
  which is a polynomial in $u=\ee^{2 \pi \ii k}$ of degree $L-1$
  whose coefficients are all on the unit circle.
  
  Since we are dealing with a constant-length substitution,
  $B^{(n)} (k)$ is $1$-periodic and defines a matrix cocycle over the
  compact dynamical system given by $k \mapsto L k$ modulo $1$ on
  $[0,1)$, which is ergodic with respect to Lebesgue
  measure. Oseledec's multiplicative ergodic theorem, compare
  \cite{BaPe,Viana}, then guarantees the existence of the Lyapunov
  exponents of our matrix cocycle for a.e.\ $k\in \RR$. Moreover,
  still for a set of full measure, the exponents are constant. They
  can be computed explicitly for each invariant subspace, where one
  obtains
\[
     \chi^{}_{j} \, = \, \log \sqrt{L} \, - 
     \lim_{n\to\infty} \frac{1}{n} \sum_{\ell =0}^{n-1}
     \log \lvert \beta_j (L^{\ell} k )\rvert
     \; \underset{k\in\RR}{\overset{\text{a.e.}}{=}} \;
     \log \sqrt{L} \, - \int_{0}^{1} \log
     \lvert \beta_j (k) \rvert \dd k  \, > \, 0 \ts .
\]  
Here, the integral is strictly less than $\log \sqrt{L}$ because
\[
     \exp \left( \int_{0}^{1} \log
     \lvert \beta_j (k) \rvert \dd k \right) \, < \,
     \int_{0}^{1}  \lvert \beta_j (k) \rvert \dd k
     \, = \, \| \beta_j \|^{}_{1}
     \, < \, \| \beta_j \|^{}_{2} \, = \, \sqrt{L} \ts ,
\]
where the first estimate follows from Jensen's inequality and is
strict, as is the second because $\beta_j $ is not a monomial
in $u$. The last equality is Parseval's identity; compare \cite{Neil}.

The diffraction measures of constant-length substitutions are
closely related to the spectral measures of characteristic functions.
More precisely, choosing $u^{}_{i} = \delta^{}_{i, \ell}$, the restriction
of $\widehat{\gamma}$ to $[0,1)$ is the spectral measure of the
characteristic function for the presence of a tile (letter)
of type $\ell$ at $0$; see \cite{BLvE} for details. Since our system
is self-similar, this argument extends to characteristic functions
of supertiles by a simple scaling argument.

Now, by \cite[Prop.~7.2]{Q} or \cite[Thm.~4.4]{Bart}, a spectral
measure of maximal type can be constructed by a linear combination of
spectral measures of characteristic functions of tiles and
supertiles. But if none of them comprise an absolutely continuous
component, the spectral measure of maximal type must be singular as
well.
\end{proof}

\begin{remark}
  The integral to which the second summand of $\chi^{}_j$ converges is
  known to be the \emph{logarithmic Mahler measure} of the polynomial
  $\beta_j$, which we denote by $\mathfrak{m} (\beta_j)$. Here, the
  logarithmic Mahler measure of a (complex) polynomial $p$ is defined
  as
\[
    \mathfrak{m} (p) \, \defeq \int_{0}^{1}
    \log \big\lvert p\bigl( \ee^{2 \pi \ii t} \bigr) 
    \big\rvert \dd t \ts .
\]
When $p(z) = a \prod_{i=1}^{n} (z-\alpha_i)$ is a non-zero polynomial
of degree $n$, one has
\[
   \mathfrak{m} (p)\, = \, \log \lvert a \rvert
   \, + \sum_{i=1}^{n} \log \bigl(\max \big\{ 1, 
   \lvert \alpha_i\rvert \big\} \bigr)
\]
by Jensen's formula. In particular, for monic polynomials,
$\mathfrak{m} (p)$ only depends on the roots of $p$ that lie outside
the unit circle; see \cite{EvWa} for general background.

This connection provides an effective tool to obtain good upper bounds
of Lyapunov exponents for constant-length substitutions, in particular
Abelian ones; compare \cite{BGM}. This step conveniently generalises
to higher-dimensional analogues of Abelian substitutions,\footnote{We
  refer to Example~\ref{ex:block} in the next section for an
  illustration.} then called Abelian \emph{block substitutions};
compare \cite{BG-block}. With that, one completely recovers Bartlett's
singularity result \cite{Bart}, for any dimension; see
Theorem~\ref{thm:higher-abelian} below. \exend
\end{remark}

\begin{example}
   Consider $\varrho^2$  with $\varrho$ from Example~\ref{ex:D4-IDA}, 
   which is a substitution on $\cA^{}_4$ with associated generating 
   subgroup $G=C_2\times C_2$. The Fourier matrix $B(k)$ reads
\[
   B(k) \, = \, \begin{pmatrix}
           1 & u^3 & u^2 & u \\ u^3 & 1 & u & u^2 \\
           u^2 & u & 1 & u^3 \\ u & u^2 & u^3 & 1
    \end{pmatrix} ,
\]
where $u=\ee^{2\pi \ii k}$, while the corresponding eigenvalues 
 are
\begin{align*}
    \beta_1 (k) &=1-u-u^2+u^3  , &  
    \beta_2 (k) &=1+u-u^2-u^3  , \\
    \beta_3 (k) &=1-u+u^2-u^3  , &  
    \beta_4 (k) &=1+u+u^2+u^3  ,
\end{align*}
with corresponding eigenvectors that are $k$-independent.  These four
polynomials are products of cyclotomic polynomials, and hence
$\mathfrak{m}(\beta_j)=0$ for $1\leqslant j \leqslant 4$.  This
results in a degenerate Lyapunov spectrum for a.e.\ $k\in\RR$, and
hence in $\chi^B = 0 < \log\sqrt{L} = \log (2)$.  \exend
\end{example}

Determining which substitutions have the same Lyapunov exponents is
generally difficult, especially since the equality of Mahler measures,
which only depend on the roots of a polynomial outside the unit
circle, does not imply that they come from the same polynomial.
However, as we shall see in the next result, a certain dichotomy gives
rise to families of substitutions that share the same Lyapunov
spectrum (before normalisation).

\begin{theorem}\label{thm:extend}
  Consider the\/ $\aaa$-letter constant-length substitution
  $\varrho\! : \, a_i \mapsto w_i$, with\/ $\lvert w_{i} \rvert = L$
  for all\/ $i$, and assume that the columns are either bijective or
  constant.  Suppose further that the
  group\/ $G'$ generated by the bijective columns is Abelian $($but
  not necessarily transitive in\/ $\vS^{}_{\aaa} \ts )$. Then, all
  Lyapunov exponents associated to\/ $\varrho$ are strictly positive.
\end{theorem}

\begin{remark}
  We stress that the set of constant-length substitutions satisfying
  the conditions of Theorem~\ref{thm:extend} is a subset of the
  substitutions with at least one constant column (also known as a
  \emph{coincidence}). Such substitutions, by Dekking's criterion
  \cite{Dekking}, have pure point spectrum. What we have confirmed
  here, using our method via Lyapunov exponents, is the singularity of
  the spectrum for this specific subset. Though this is a weaker
  result, it is interesting in its own right, as it extends to other
  cases where Dekking's criterion yields no answer.  \exend
\end{remark}

\begin{proof}[Proof of Theorem~\textnormal{\ref{thm:extend}}]
  From the premise, the Fourier matrix of $\varrho$ can be decomposed
  into
\[
   B ( k ) \, = \, B_{\mathsf{b}} (k) + B_{\mathsf{c}} (k) \ts ,
\]
where $B_{\mathsf{b}} (k )$ and $B_{\mathsf{c}} (k)$ are generated by
the bijective and constant columns, respectively. This
gives a partition of the positions as
$\{ 0, \ldots, L-1\} = S_{\mathsf{b}} \, \dot{\cup} \,
S_{\mathsf{c}}$.
The idea of the proof now is to show that all eigenvalues of
$B_{\mathsf{b}}(k)$ except one (and their corresponding eigenvectors)
are essentially inherited by $B(k)$. We begin by illustrating how this
works for cases when $G'$ is transitive, and later describe what
changes in the case when it is not.

It follows from Theorem~\ref{thm:pos-abelian} that
$B_{\mathsf{b}} (k)$ has $\aaa$ linearly independent eigenvectors that
do not depend on $k$. Furthermore, $\aaa -1$ of them have a component
sum equal to zero, with the remaining eigenvector being
$v^{}_{\aaa}= (1,1, \ldots,1)^T$. This property follows from the facts
that these eigenvectors can directly be constructed from the character
table of $G'$ and that $\sum_{g\in G} \rho(g)=0$ for all irreducible
representations $\rho\in\widehat{G'}$ except the trivial one.

Consider any eigenvector $v$ of $B_{\mathsf{b}} (k )$, with eigenvalue
$\beta (k)$ say, with zero component sum. Observe
that we can write $B_{\mathsf{c}} (k)$ as 
\[
   B_{\mathsf{c}} (k) \, = \sum_{z\in S_{\mathsf{c}}}
   \ee^{2\pi \ii z k} R_{a(z)} \ts , \quad
   1 \leqslant a(z) \leqslant \aaa \ts ,
\]
where the matrix $R_m$ has entries $1$ in the $m$-th row and $0$
everywhere else. Consequently, $R_{m} v = 0$ for all
$1\leqslant m \leqslant \aaa$, which implies
$B_{\mathsf{c}} (k ) \ts v =0$. But this means that $v$ is also an
eigenvector of $B (k )$, with the same eigenvalue $\beta (k)$.

As in Theorem~\ref{thm:pos-abelian}, the eigenvalues of 
$B_{\mathsf{b}}$ can be written in terms of the characters of $G'$,
\[
   \beta_{j} (k) \, = \sum_{m\in S_{\mathsf{b}}}
   \overline{\rho^{}_{j} (g^{}_{m} )} \, u^m ,
\]
which is always a polynomial in $u=\ee^{2\pi \ii k}$ of degree at most
$L-1$. All its coefficients are either $0$ or have modulus $1$.
Parseval's equation then once again guarantees that the Lyapunov
exponents arising from these eigenvalues are strictly less than
$\log\sqrt{L}$. The maximal Lyapunov exponent is achieved for
some $j$, which in turn satisfies
\[
   \chi^{B} \, = \, \chi^{}_j \, = \,
   \mathfrak{m} ( \beta_{j} ) \,  < \,  \log\sqrt{L} \ts .
\]
The $\aaa -1 $ exponents shared by $B$ and $B_{\mathsf{b}}$ clearly
satisfy this bound. The idea is now to invoke Lyapunov forward
regularity to show that the last exponent is zero, which is done prior
to (additively) normalising with $\log \sqrt{L}$. This will confirm
that $B$ and $B_{\mathsf{b}}$ indeed share the same set of exponents.

To this end, we note that the $\aaa$-th eigenvalue of $B(k)$ is
$\beta^{\ts\prime}_{\aaa}(u) = \sum_{m=0}^{L-1} u^m$, which easily
follows from the trace formula. By Lyapunov forward regularity
\cite{BaPe}, we see that the sum of the exponents under the outward
iteration \eqref{eq:iter-out}, without the prefactor $\sqrt{\lambda}$,
is given by
\[
\begin{split}
     \sum_{m=1}^{\aaa}  \chi^{\ts\prime}_{m} \, & = \,  -\int_{0}^{1}
         \log \left| \det \bigl(B(k)\bigr) \right| \dd k 
    \, = \, -\sum_{m=1}^{\aaa} \int_{0}^{1} \log
      \left| \beta_{m}^{\ts\prime} (k ) \right| \dd k \\[2mm]
    & = \,  \chi_{1}^{\ts\prime}+ \chi_{2}^{\ts\prime} + \cdots +
        \chi_{\aaa -1}^{\ts\prime} - \mathfrak{m} 
        ( \beta^{\ts\prime}_{\aaa} ) \ts  ,
\end{split}
\]
from which it is clear that
$\chi^{\ts\prime}_{\aaa}=-\mathfrak{m} (\beta^{\ts\prime}_{\aaa})=0$
since $\beta^{\ts\prime}_{\aaa}$ is cyclotomic.  This completes the
argument for the transitive case.

When $G'$ fails to be transitive, we can still use the decomposition
$B = B_{\mathsf{b}} + B_{\mathsf{c}}$, where $B_{\mathsf{b}}$ now has
to be put into block diagonal form via some elementary matrix
operations that partition
$\cA^{}_{\aaa}= \{a^{}_1, \ldots , a^{}_{\aaa} \}$ into orbits of
$G'$. A particularly useful decomposition of $G'$ is
$G' \simeq G'_1 \times \cdots \times G'_s $, wherein each subgroup
$G'_{\ell}$ (which can be the trivial subgroup) acts transitively on
the $s$ orbits in $\cA_{\aaa}$. Furthermore, each nontrivial
$G'_{\ell}$ can be written as a finite product of cyclic groups by the
fundamental theorem of finite Abelian groups.  This also means that
the digit matrices afford the splitting 
\[
  D_m \, = \, \varPhi \bigl( g^{-1}_m \bigr) \, = \:
  \bigoplus_{\ell=1}^{s} \overline{\varPhi^{}_{\ell} 
  \bigl( g^{(\ell)}_{m} \bigr)} 
\]  
with $g^{}_m = \bigoplus_{\ell=1}^{s} \, g^{(\ell)}_{m} $, where 
$\varPhi^{}_{\ell}$ is the permutation representation on $G'_{\ell}$.

With this, we recover the eigenvalues of $B_{\mathsf{b}}$ from each 
block as
\[
   \beta^{(\ell)}_{j} (k) \, = \sum_{m \in S_{\mathsf{b}}}
   \overline{ \rho^{(\ell)}_{j} \bigl( g^{(\ell)}_{m} \bigr)} \, u^m ,
\]
where $\rho^{(\ell)}_{j}$ is an irreducible character of $G'_{\ell}$.
An immediate consequence is that $\sum_{m \in S_{\mathsf{b}}} u^m$ has
multiplicity $s$ as an eigenvalue of $B_{\mathsf{b}}$ (corresponding
to different eigenvectors) since all blocks naturally admit the
trivial representation. Note that non-transitivity in conjunction with
primitivity of the substitution implies that at least one coincidence
must be present, which implies $\card (S_{\mathsf{b}}) < L$.

Similar to the transitive case, any eigenvector of $B_{\mathsf{b}}$
with zero component sum remains an eigenvector of $B$, with the same
eigenvalue. All but one copy of the polynomial
$\sum_{m \in S_{\mathsf{b}}} u^m$ also remain eigenvalues, but this
time with the corresponding eigenvectors being linear combinations of
eigenvectors from different blocks. Finally, the uninherited
eigenvalue (the one with a $k$-dependent eigenvector) is the
cyclotomic polynomial $\sum_{m=0}^{L-1} u^m$, which can be computed
from the trace. It is easy to see that the same arguments
unambiguously apply as in the transitive case, since the eigenvalues
are polynomials in $u$ with coefficients of unit modulus.
\end{proof}

At this point, some examples are in order. 

\begin{example}[$\cA_3$, transitive, $G' \simeq C_3$]
  Consider the substitution $\varrho^{}_3$, with Fourier matrix
  $B^{}_{3}(k)$, given by
\[
\varrho^{}_3: \left\{ \begin{array}{ccc}
          0& \mapsto & 0022 ,\\
         1 & \mapsto & 1002 ,\\
         2 & \mapsto & 2012 ,
         \end{array} \right.
\quad \text{and} \quad
   B^{}_{3}(k) \, = \, \begin{pmatrix}
      1+u & u+u^2 & u \\ 0 & 1 & u^2 \\
      u^2+u^3 & u^3 & 1+u^3
      \end{pmatrix} ,
\]
where $u=\ee^{2\pi \ii k}$ as usual.  The eigenvalues and eigenvectors
of $B^{}_3 (k)$, which derive from the digit matrices that generate
the Abelian IDA, are given by
\begin{gather*}
     \beta_1(k) \, = \, 1+\omega^2 \ts u^2 
      \quad \text{with} \quad 
      v^{}_1 \, = \, (\omega^2,\omega,1)^{\text{T}} , \\ 
    \beta_2 (k) \, = \, 1+\omega \ts u^2 
     \quad \text{with} \quad 
      v^{}_2 \, = \, (\omega,\omega^2,1)^{\text{T}} ,
\end{gather*}
where $\omega=\ee^{\frac{2\pi \ii}{3}}$. The third eigenvalue is given
by $\beta_3(k)= 1+u+u^2+u^3$. The logarithmic Mahler measures
$\mathfrak{m} (\beta_i)$ are all zero since the respective roots all
lie on the unit circle. This implies that the (almost surely constant)
exponents are
$\chi^{}_{j} = \log (2) - \mathfrak{m} (\beta^{}_{j}) = \log (2)$ and
thus all strictly positive. \exend
\end{example}

\begin{example}[$\cA_4$, non-transitive, $G' \simeq C_2\times C_2$]
  The converse of Fact~\ref{fact:trans} is not true. There are Abelian
  subgroups of $\vS^{}_n$ of order $n$ that are not transitive. In
  $\vS_4$, there are seven subgroups isomorphic to Klein's $4$-group
  $C_2\times C_2$, only three of which are transitive. Here, we
  select a substitution where $G'$ has two disjoint orbits. Consider
  the substitution $\varrho^{}_{V}$, alongside with its corresponding
  Fourier matrix,
\[
    \varrho^{}_V:  \left\{
       \begin{array}{ccc}
          0& \mapsto & 0112 ,  \\
          1& \mapsto & 1012 , \\
          2& \mapsto & 3212 , \\
          3& \mapsto & 2312 , 
       \end{array} \right.
\quad \text{and} \quad 
     B(k) \, = \, \underbrace{\begin{pmatrix}
        1 & u  & 0 & 0 \\ u & 1  & 0 & 0\\
        0 & 0  & u & 1 \\ 0 & 0  & 1 & u 
     \end{pmatrix}}_{B_{\mathsf{b}}} +
     \underbrace{\begin{pmatrix}
        0 & 0 & 0 & 0 \\ u^2 & u^2 & u^2 & u^2 \\
       u^3 & u^3 & u^3 & u^3 \\ 0 & 0 & 0 & 0 
     \end{pmatrix}}_{B_{\mathsf{c}}}.
\]
The eigenvalues of $B$ that correspond to three $k$-independent 
eigenvectors of $B_{\mathsf{b}}$ are
\begin{align*}
    \beta_1(k) & =1-u  \ts ,  &   
          v^{}_1 & =(-1,1,0,0)^{T} , \\
    \beta_2(k) & =-1+u \ts , &  
          v^{}_2  & =(0,0,-1,1)^{T} , \\
    \beta_3(k) & =1+u  \ts , &   
          v^{}_3 & =(-1,-1,1,1)^{T} ,
\end{align*}
with the last eigenvalue being $\beta_4 (k) =1+u+u^2+u^3$. Here, one
sees that $v^{}_3$ is a linear combination of the eigenvectors from
the two separate blocks of $B_{\mathsf{b}}$ corresponding to the same
eigenvalue $\beta_3 (k)$. The positivity of the Lyapunov exponents
follows from the same arguments as in our previous examples.

As mentioned above in Theorem~\ref{thm:Boris}, it is a necessary criterion
\cite{BS} for a primitive substitution $\varrho$ of constant length
$L$ to have an absolutely continuous component in its dynamical
spectrum that the substitution matrix of $\varrho$ must have an
eigenvalue of modulus $\sqrt{L}$.  One can easily check that the
substitution matrix $M=B(0)$ of $\varrho^{}_{V}$ has eigenvalues
$\{4,2,0,0 \}$, so it satisfies the criterion. However, $\varrho^{}_V$
contains coincidences, and hence has pure point spectrum by Dekking's
criterion \cite{Dekking}. The absence of AC spectral components is
also rederived here via the positivity of the Lyapunov exponents.
\exend
\end{example} 

Another example that satisfies the $\sqrt{L}$-criterion but evades AC
spectral measures was analysed in \cite{CG}. Based on Bartlett's
approach, it was shown to have singular continuous spectrum. The
absence of AC spectral components was independently shown via Lyapunov
exponents in \cite{BG-block}. Let us comment on how one can
systematically construct examples that satisfy the
$\sqrt{\lambda}$-criterion, but do not have absolutely continuous
spectrum.

In general, one can begin with a (non-primitive) substitution of
length $L$ in $\aaa$ letters, whose columns are bijective and whose
generating subgroup $G'$ is a non-transitive subgroup of
$\vS^{}_{\aaa}$.  From the proof of Theorem~\ref{thm:extend},
$\beta(k)=1+u+\ldots+u^{L-1}$ is always an eigenvalue of
$B_{\mathsf{b}}$, and at least one copy of it survives to be an
eigenvalue of $B$, which means that $\beta(0)=L$ is an eigenvalue of
$M=B(0)$ of the new substitution formed by adding coincidences. One
can then choose to add appropriate columns, so that the resulting
substitution is primitive, and enough columns, so that it is of length
$L^2$.

\section{Inflation tilings in higher dimensions}\label{sec:higher}

At first sight, Theorem~\ref{thm:1D} and the methods employed to
derive it appear as a one-dimensional affair, which fortunately is not
the case. In fact, it is precisely the separation of the geometry of
the underlying space from the geometry of the inflation structure via
the Fourier matrix that can be transferred to higher dimensions, which
is one of the key advantages of the underlying geometric
self-similarity.

The general setting for inflation tilings in $\RR^d$ is as follows. We
assume we are given $\aaa$ prototiles $\cT_1, \ldots, \cT_{\aaa}$
(that is, tiles up to translations), together with an expansive linear
map $Q$ such that, for each $i$, the image $Q (\cT_i)$ is a union of
non-overlapping translates of prototiles. This is called a \emph{stone
  inflation} in \cite{TAO}, and we will once again use $\varrho$ for
the explicit inflation rule. The latter is called \emph{primitive} if
the corresponding incidence (or substitution) matrix is a primitive
matrix. An important subclass is provided by inflations that lead to
tilings of \emph{finite local complexity} (FLC) with respect to
translations. If one wants to go beyond this case, some extra
precaution will be required.

Essentially as in the one-dimensional case, any primitive
inflation defines a unique hull, called $\YY$, which can be obtained
as the orbit closure of a fixed point tiling of $\varrho$ (or of
$\varrho^m$ for a suitable $m\in\NN$) under the translation action of
$\RR^d$. One has the following classic result; compare
\cite{Boris,Robbie,TAO,FR,MR} and references therein.

\begin{fact}
  The hull\/ $\YY$ of a primitive stone inflation\/ $\varrho$ 
  is compact in the local rubber topology,
  consists of a single LI class and gives rise to a minimal
  topological dynamical system\/ $(\YY, \RR^d)$. The latter is
  strictly ergodic, where the unique invariant probability
  measure is the patch frequency measure.  \qed
\end{fact}

\begin{remark}
   The term `patch frequency' is the extension of a standard notion
   (word frequency) from symbolic dynamics to FLC tilings. Patch
   frequencies are used to define a translation-invariant measure on
   the tiling hull (equipped with the local topology) via cylinder
   sets; compare \cite{Q,TAO}.
   
   Here, we also need the analogous concept for hulls of non-FLC
   tilings, now equipped with the \emph{local rubber topology}
   \cite{BL,TAO}. The latter is a Fell topology and reduces to the
   local topology in the FLC case. The construction of a measure
   via cylinder sets now requires the consideration of patches
   (or clusters) up to $\varepsilon$-deformations for small
   $\varepsilon > 0$. For primitive inflation rules with finitely many
   prototiles up to translations, the resulting properties are very
   similar to the FLC case; see \cite{FR,LS} for details.
\exend
\end{remark}

Let us now assume that each prototile is equipped with a \emph{control
  point} in such a way that any tiling in the hull $\YY$ is MLD with
the corresponding set of control points. In this context, it is often
convenient to distinguish (or colour) the points according to the
prototile type. Unlike the situation in one dimension, there may not
be an obvious or canonical choice for the control points. This does
not matter much because the tiling hull and the corresponding set of
point sets define topologically conjugate dynamical systems, wherefore
we identify the two points of view. It will be clear from the context
whether we speak about tilings or about their MLD control point sets.

\begin{remark}
  Two (possibly coloured) point sets that are MLD need not (and
  generally do not) have the same pair correlations. Consequently, the
  diffraction measures will generally also differ, which reflects the
  fact that the diffraction measure neither is an invariant under
  topological conjugacy nor under metric isomorphism. Nevertheless,
  the \emph{type} of the dynamical spectrum is invariant, and our approach
  via Lyapunov exponents thus aims at statements about the presence or
  absence of absolutely continuous spectral components.  \exend
\end{remark}

Let us now develop the higher-dimensional analogue of the displacement
matrix and how it can be used to derive Lyapunov exponents together
with suitable estimates.

\subsection{Displacement and Fourier matrices}

Let $T_{ij}$ be the set of control point positions of tiles of type
$i$ in the supertile of type $j$, relative to the control point
position of the latter, where the location of the control points in
the supertiles are determined by the action of the linear map $Q$. As
before, $M = \card (T)$ is the incidence matrix, where
$T= (T_{ij})^{}_{1\leqslant i,j \leqslant \aaa}$ is the set-valued
displacement matrix. Likewise,
$Q T = ( Q T_{ij} )^{}_{1\leqslant i,j \leqslant \aaa}$ with
$Q T_{ij} \defeq \{ Q (t) \mid t \in T_{ij} \}$ is the displacement matrix
for the relative positions of supertiles in level{\ts}-$2$ supertiles.
Now, let $T^{(n)}$ denote the displacement matrix for the relative
positions of tiles in the supertiles of level $n$, which is the
displacement matrix for the inflation rule defined by
$\varrho^n$. Clearly, $T^{(1)} = T$ and, adding the relative
displacements by one additional inflation step, one recursively gets
\begin{equation}\label{eq:rel-displace}
     T^{(n+1)}_{ij} \, = \, \bigcup_{\ell=1}^{\aaa} \ts 
     \bigl( T^{}_{i \ell} + QT^{(n)}_{\ell j} \bigr) ,
\end{equation}
where $+$ denotes the Minkowski sum of two point sets, as defined
by 
\[
    U \nts + V \, = \, \{ u+v \mid u \in U, v \in V \} \ts .
\]
It is easy to check that $\card (T^{(n)}) = M^n$.

As before, with $\delta^{}_{T} \defeq \bigl( \delta^{}_{T_{ij}}
\bigr)_{1\leqslant i,j \leqslant \aaa}$,
the \emph{Fourier matrix} of the inflation $\varrho$ is
\begin{equation}\label{eq:F-mat-gen}
     B (k) \, \defeq  \overline{\widehat{\delta^{}_{T}}} (k)
     \, = \, \widehat{\delta^{}_{T}} (-k) \ts 
\end{equation}
with $k \in \RR^d$. For each $k$, one has $B (k) \in \Mat (\aaa, \CC)$,
and each matrix element of $B$, as a function of $k$, is a
multivariate trigonometric polynomial.

\begin{lemma}\label{lem:gen-F-matrix}
  Let\/ $\varrho$ define a primitive stone inflation, with finitely
  many translational prototiles and linear expansion\/ $Q$. Further,
  let\/ $B(k)$ be the Fourier matrix from
  Eq.~\eqref{eq:F-mat-gen}. Then, for\/ $n\in\NN$, the Fourier matrix
  of\/ $\varrho^n$ is given by
\[
    B^{(n)} (k) \, = \, B(k) \ts  B (Q^T k) 
   \cdots B \bigl( (Q^T)^{n-1} k \bigr) 
\]
  and satisfies\/ $B^{(1)} = B$ together with\/ $B^{(n+1)} (k) =
  B(k) \, B^{(n)} (Q^T k)$ for\/ $n\in \NN$.
\end{lemma}

\begin{proof}
 This is a consequence of Eq.~\eqref{eq:rel-displace}
 in conjunction with the observation that
\[
    \ee^{2 \pi \ii k (u + Q v)} \, = \, \ee^{2 \pi \ii k u}\,
    \ee^{2 \pi \ii k \ts Qv} \, = \, \ee^{2 \pi \ii k u} \,
    \ee^{2 \pi \ii (Q^T \nts k) \ts v} .
%    \, = \, \ee^{2 \pi \ii ( k u \, + \, (Q^T \nts k) v )}.
\]
Now, with $B^{(n)} (k) = \widehat{\delta^{}_{T^{(n)}}} (-k)$, the
structure of the Minkowski sum in \eqref{eq:rel-displace} together
with $\delta_{u+v} = \delta_{u} * \delta_{v}$ and the convolution
theorem for Fourier transforms give the claim for $n=2$ by a simple
calculation. The general formula with its recursive structure is then
obvious.
\end{proof}

\subsection{Renormalisation relations for pair correlations}

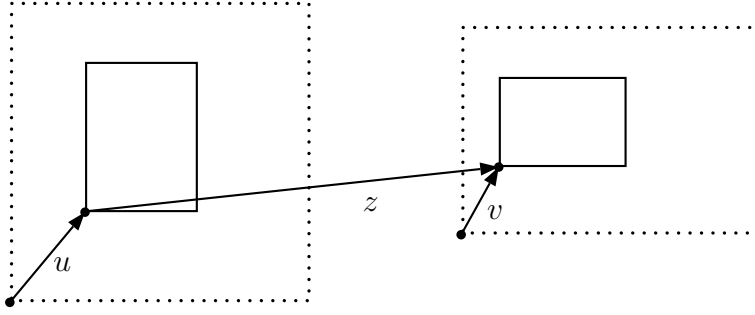
\begin{figure}[t]
\begin{pspicture}(10,4.5)
\psframe[linewidth=1.3pt,linestyle=dotted](0,0)(4,4)
\psdot(0,0)
\psframe[linewidth=1.3pt,linestyle=dotted](6,0.9)(10,3.68)
\psdot(6,0.9)
\psframe[linestyle=solid](1,1.2)(2.5,3.2)
\psdot(1,1.2)
\psframe[linestyle=solid](6.5,1.8)(8.2,3)
\psdot(6.5,1.8)
\psset{arrowsize=5pt,arrowinset=0}
\psline{->}(0,0)(1,1.2)      \rput(0.7,0.5){\large $u$}
\psline{->}(6,0.9)(6.5,1.8)  \rput(6.45,1.2){\large $v$}
\psline{->}(1,1.21)(6.5,1.8)  \rput(4.8,1.3){\large $z$}
\end{pspicture}
\caption{If the two tiles (solid lines) at distance $z$ have offsets
  $u$ and $v$ within their covering supertiles (dotted lines), the
  latter have distance $z+u-v$.\label{fig:two} Here, the distances are
  always defined via the control points of the tiles.}
\end{figure}

Let us begin with the case that the primitive inflation defines an FLC
tiling hull.  In complete analogy to one dimension, $\nu^{}_{ij} (z)$
is defined as the relative frequency of finding a point of type $i$
and one of type $j$ separated by $z\in\RR^d$, which is to be read as
the vector from positions $i$ to $j$. In other words, one has
\begin{equation}\label{eq:def-pc-gen}
    \nu^{}_{ij} (z) \, = \,   \frac{\dens \bigl( \vL_i
     \cap (\vL_j - z ) \bigr)}{\dens (\vL) } \ts ,
\end{equation}
which is once again independent of the choice of $\vL$ from the hull,
with $\supp (\nu^{}_{ij}) = \vL^{}_j - \vL^{}_i$.

The exact renormalisation relations, first announced in
\cite{Neil-MFO}, now read
\begin{equation}\label{eq:nu-reno}
   \nu^{}_{ij} (z) \, = \, \myfrac{1}{\lvert\det(Q)\rvert} 
   \sum_{m,n = 1}^{\aaa} \, \sum_{\substack{u\in T_{im} \\ v \in T_{jn}}}
   \nu^{}_{mn} \bigl( Q^{-1} (z + u - v) \bigr) .
\end{equation}
The derivation is once again based on recognisability in aperiodic
inflation tilings, which holds here as well \cite{Boris-2}, and works
exactly as in the one-dimensional case covered by
Lemma~\ref{lem:frequ-eqs}; see Figure~\ref{fig:two} for an
illustration with block tiles whose lower left corners are chosen as
their reference points, which explains the meaning of $u$, $v$ and $z$
in \eqref{eq:nu-reno}. The extension to also cover the periodic cases
is a consequence of Remark~\ref{rem:periodic}.

With $\vU^{}_{\nts ij} = \sum_{z} \nu^{}_{ij} (z) \, \delta^{}_{z}$,
which can alternatively be defined as in Eq.~\eqref{eq:corr-as-conv},
this leads to the corresponding relations for the pair correlation
measures, namely
\begin{equation}\label{eq:ups-reno}
    \vU^{}_{\nts ij} \, = \, \myfrac{1}{\lvert\det(Q)\rvert} 
    \sum_{m,n=1}^{\aaa}
    \widetilde{\delta^{}_{T_{im}}} * \ts
    \delta^{}_{T_{jn}} * \bigl( Q.\vU^{}_{mn}\bigr) .
\end{equation}
Its derivation from \eqref{eq:nu-reno} is based on the same
calculation as in the one-dimensional case.  At this point, using the
results from \cite{BL,FR} on Delone dynamical systems in the local
rubber topology, it is not difficult to see that the pair correlation
measures $\vU_{ij}$ are well-defined, and that Eq.~\eqref{eq:ups-reno}
is the correct relation among them, also in the non-FLC case, provided
we start from a primitive inflation rule with \emph{finitely} many
translational prototiles. The crucial step here is that $T$ then still
is a finite matrix of displacement sets.

\begin{prop}
    Let\/ $\YY$ be the tiling hull of a primitive stone inflation with
    finitely many translational prototiles. Let\/ $Q$ denote the
    corresponding expansive linear map and\/ $T$ the displacement
    matrix. Then, the pair correlation measures\/ $\vU_{ij}$ are
    well defined and satisfy the renormalisation 
    relations~\eqref{eq:ups-reno}.

    Moreover, whenever\/ $\YY$ is FLC, each\/ $\vU_{ij}$ is a pure 
    point measure, and the pair correlation functions\/ $\nu^{}_{ij}$ 
    defined in \eqref{eq:def-pc-gen} satisfy the 
    relations~\eqref{eq:nu-reno}.  \qed
\end{prop}

\begin{remark}
  While the formulation in terms of a stone inflation is the closest
  analogue to the one-dimensional case, it is by no means necessary.
  Clearly, one can also consider a more general inflation scheme with
  a rule that guarantees the gapless and overlap{\ts}-free cover of
  space in the limit. Various scenarios, and how they are related, are
  described in \cite[Ch.~6]{TAO}, including many classic examples such
  as the Ammann--Beenker and the Penrose tilings, but also more
  complicated ones that can be reformulated as a stone inflation via
  fractiles (tiles with fractal boundaries); see \cite{TAO,Dirk} for
  various examples. We leave further details to the interested reader.
  \exend
\end{remark}

\subsection{Analysis after Fourier transform}

Let $Q^* = (Q^T)^{-1}$ be the \emph{dual} matrix and observe that
$\widehat{Q.\ts \mu} = \lvert\det(Q)\rvert^{-1} Q^* \! . \ts
\widehat{\mu}\ts $, as shown in \cite[Lemma~2.5]{BGM}.
Then, Fourier transform turns relation \eqref{eq:ups-reno} into
\begin{equation}\label{eq:gen-FT-eq}
   \widehat{\vU}^{}_{\nts ij} \, = \, \det(Q)^{-2}
   \sum_{m,n=1}^{\aaa} B^{}_{i m} (.) \, \overline{B^{}_{j n} (.)}
   \, \bigl( Q^* \! . \widehat{\vU}^{}_{m n} \bigr) .
\end{equation}
By Lemma~\ref{lem:type-recursion}, which remains valid here without
any change, Eq.~\eqref{eq:gen-FT-eq} gives rise to three separate
relations for the spectral types.

The analysis of the pure point part can be done in complete analogy
to Section~\ref{sec:pp-part} and leads to
\[
     \bs{I} (k) \, = \, \det (Q)^{-2}
     \bs{A} (k) \ts \bs{I} (Q^T k)
\]
together with the eigenvector relation
$\bs{A} (0) \ts \bs{I} (0) = \det (Q)^2 \ts \bs{I} (0)$. Since
$\lvert \det (Q) \rvert$ is the PF eigenvalue of $M \defeq B(0)$,
which is the inflation or incidence matrix of our system, the entries
$I_{ij} (0)$ of $\bs{I} (0)$ once again satisfy \eqref{eq:pp-dens}, so
$I_{ij} (0) = \dens (\vL)^{-2} \dens (\vL_i) \dens (\vL_j)$.

Likewise, when we again represent the absolutely continuous part of
$\widehat{\vU}$ by the vector $\bs{h}$ of Radon--Nikodym densities,
Lemma~\ref{lem:ac-rec} still holds, now with the relation
\[
    \bs{h} (k) \, = \, \myfrac{1}{\lvert\det (Q)\rvert} \bs{A} (k) 
    \ts \bs{h} (Q^T k) \ts .
\]
Clearly, also the argument with the dimensional reduction can be
applied here. Consequently, we may consider the cocycle $B^{(n)} (k)$
from Lemma~\ref{lem:gen-F-matrix} in conjunction with the Lyapunov
exponents
\begin{equation}\label{eq:L-general}
   \chi^{B} (k) \, \defeq \, \limsup_{n\to\infty}
   \myfrac{1}{n} \log \big \| B^{(n)} (k) \big \|
   \quad \text{and} \quad
   \chi^{}_{\min} (k) \, = \, \log
   \sqrt{\lvert \det (Q) \rvert}\, - \chi^{B} (k) \ts .
\end{equation}
A rather straight-forward generalisation of our
previous proof in one dimension results in the following
higher-dimensional counterpart.

\begin{theorem}\label{thm:higher-D}
  Let\/ $\varrho$ be a primitive inflation rule in\/ $\RR^d$, with
  finitely many translational prototiles and expansive linear map\/
  $Q$.  Let\/ $B(k)$ be the corresponding Fourier matrix, with\/
  $\det (B(k)) \ne 0$ for at least one\/ $k\in\RR^d$.  If there is
  an\/ $\varepsilon > 0$ such that
\[  \chi^{B} (k) \, \leqslant \,
    \log \sqrt{\lvert \det(Q)\rvert} \, - \varepsilon
\]
  holds for a.e.\ $k\in\RR^d$, where\/ $\chi^{B} (k)$ is the maximal
  Lyapunov exponent from \eqref{eq:L-general}, the diffraction
  measure of the system cannot have an absolutely continuous part.
  \qed
\end{theorem}

In analogy to the one-dimensional case, $\chi^{}_{\min} (k) >0$ for a
subset of full measure within an open neighbourhood of $0$ already
suffices to rule out an absolutely continuous diffraction
component. In fact, we can also repeat the proof of
Theorem~\ref{thm:geq}, with minor modifications, so that we obtain the
following result.

\begin{coro}\label{coro:higher-D}
  Let\/ $\varrho$ be a primitive inflation rule with the conditions as
  in Theorem~\textnormal{\ref{thm:higher-D}}. Then, one has\/
  $\chi^{}_{\min} (k) \geqslant 0$ for a.e.\ $k\in\RR^d$.  If the
  system displays a non-trivial diffraction component of absolutely
  continuous type, one must have\/ $\chi^{}_{\min} (k) = 0$ for a
  subset of\/ $\RR^d$ of positive measure, which has full
  measure when\/ $\chi^{}_{\min} (k)$ is constant for a.e.\
  $k\in\RR^d$. \qed
\end{coro}

Let us now apply the theory to two planar examples of rather different
nature. Further cases are discussed in detail in \cite{BG-block}.

\subsection{Block substitutions}\label{sec:block}

Here, we consider unit cubes of finitely many types, say white and
black in the binary case, together with a diagonal matrix
$Q = \diag (q^{}_{1}, \ldots, q^{}_{d} )$, where all
$q^{}_i \geqslant 2$ for expansiveness. Using the key argument around
Jensen's inequality and Parseval's equation from the proof of
Theorem~\ref{thm:pos-abelian}, it is not difficult to prove the
following higher-dimensional version of \cite[Cor.~9]{Neil}; see
\cite{BG-block,Neil-diss} for details and further consequences.

\begin{theorem}\label{thm:gen-absence}
  Let\/ $\varrho$ be a primitive, binary block substitution in\/ $d$
  dimensions, with expansion matrix\/
  $Q = \diag (q^{}_{1}, \ldots, q^{}_{d} )$ where\/
  $q^{}_i \geqslant 2$ for\/ $1 \leqslant i \leqslant d$. Then, the
  minimal Lyapunov exponent, for a.e.\ $k\in\RR^d$, is bounded away
  from\/ $0$, and we have absence of absolutely continuous
  diffraction. Moreover, all dynamical spectral measures are
  singular. \qed
\end{theorem}

Likewise, we easily obtain the following version
of Theorem~\ref{thm:pos-abelian} in higher dimensions.

\begin{theorem}\label{thm:higher-abelian}
  Let\/ $\varrho$ be a primitive, bijective block substitution in\/
  $d$ dimensions, whose IDA is Abelian, with\/ $Q$ as in
  Theorem~\textnormal{\ref{thm:gen-absence}}. Then,
  $\chi^{}_{\min} (k) > 0$ holds for a.e.\ $k\in\RR^{d}$, and we have
  absence of absolutely continuous diffraction. As before, all
  spectral measures are singular. \qed
\end{theorem}

\begin{remark}
  In both theorems, we are in the situation that our cocycle is
  effectively defined over a compact dynamical system, namely the one
  induced by the action of $Q$ on the $d$-torus. Consequently, by
  Oseledec's theorem, we know that the Lyapunov exponents almost
  surely exist as limits, and are constant on a set of full measure.

  Also, in both cases, we know that the treatment of the diffraction
  measure gives access to a representative spectral measure of maximal
  type, via the lookup functions of tiles and supertiles at the
  origin, and thus determines its spectral type as well. In this sense,
  the answer for these cases is complete.
  \exend
  \end{remark}

Let us illustrate Theorem~\ref{thm:higher-abelian} with a bijective
ternary block substitution that also relates to our previous treatment
of Abelian substitutions.

\begin{example}[Planar block substitution with three tiles]\label{ex:block}
  Consider the inflation rule 
\[
    \includegraphics[width=0.9\textwidth]{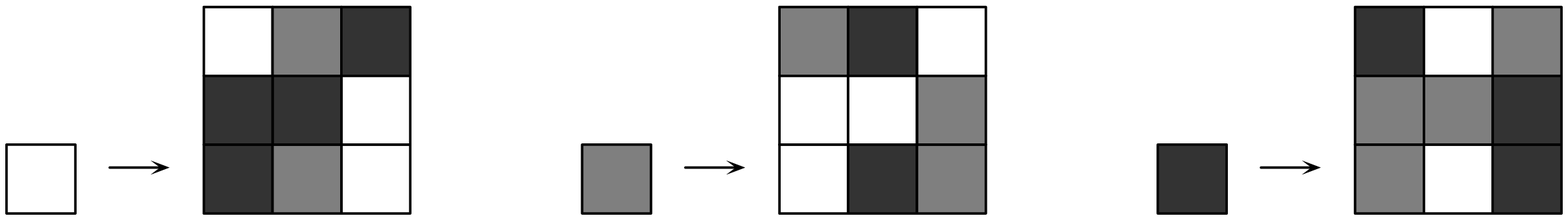}
\]
which is clearly primitive and bijective, with $Q = \diag (3,3)$.  By
inspection, one verifies that the permutation subgroup is isomorphic
to $C_3$.

Its Fourier matrix reads
\[
    B (k_1 , k_2 ) \, = \, \begin{pmatrix}
    x^2 (1+y) + y^2 & 1+y+xy+x^2y^2 & x (1+y^2) \\
    x (1+y^2) & x^2 (1+y) + y^2 & 1+y+xy+x^2y^2  \\
    1+y+xy+x^2y^2 & x (1+y^2) & x^2 (1+y) + y^2 
    \end{pmatrix}
\]
with $x=\ee^{2\pi \ii k_1}$ and $y=\ee^{2\pi \ii k_2}$.  The
eigenvalues of the Fourier matrix are given by
\begin{align*}
  \beta_1(k_1,k_2) &= (1+x+x^2)(1+y+y^2) \ts , \\
  \beta_2(k_1,k_2) &= (x^2+x^2y+y^2 )+\omega(1+y+xy+x^2y^2)
     +\omega^2(x+xy^2) \ts , \\
  \beta_3(k_1,k_2) &= (x^2+x^2y+y^2 )+\omega^2(1+y+xy+x^2y^2)
     +\omega(x+xy^2) \ts ,
\end{align*}
where $\omega=\ee^{\frac{2\pi \ii}{3}}$. As polynomials in two
variables with complex coefficients, all of them are of height $1$,
and have logarithmic Mahler measures strictly less than
$\log \sqrt{\lvert \det (Q)\rvert} = \log (3)$. In line with
Theorem~\ref{thm:higher-D} and the comment following it, we thus see
that the diffraction measures are singular, as are all spectral
measures of the dynamical spectrum.

It is clear that there is an abundance of similar examples.
As long as the permutation subgroup is Abelian,
the corresponding bounds do not require $Q$ to be a
homothety. 
\exend
\end{example}

\subsection{The Godr\`{e}che--Lan\c{c}on--Billard
     tiling}\label{sec:LB}

\begin{figure}[t]
 \includegraphics[width=0.83\textwidth]{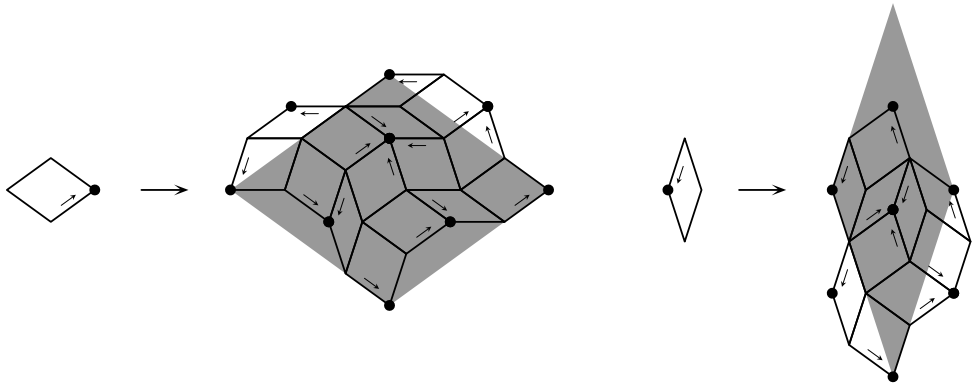}
\caption{Modified inflation rule for the
    Godr\`{e}che--Lan\c{c}on--Billiard tiling, with the control points
    marked by black dots.\label{fig:LB}}
\end{figure}

Here, we briefly consider a non-PV inflation tiling of the plane that
was described by Godr\`{e}che and Lan\c{c}on in \cite{GL}, following
up an idea from Lan\c{c}on and Billiard \cite[Fig.~4]{LB}.  It uses
the two rhombuses of the Penrose tiling for a primitive inflation rule
with expansive map $Q = \lambda R$, where
\begin{equation}\label{eq:lb-lam}
    \lambda \, = \, 2 \cos \bigl( \myfrac{\pi}{10}\bigr)
    \, = \, \sqrt{\myfrac{1}{2}\bigl( 5 + \sqrt{5}\, \bigr) }
    \, \approx \, 1.902 \ts 
\end{equation}
and $R$ is a rotation through $\frac{\pi}{10}$; see
\cite[Sec.~6.5.1]{TAO} for a detailed account, including its
reformulation as a stone inflation with two fractiles, which we do not
use here.  This inflation rule defines an FLC tiling hull as usual,
and any two elements are locally indistinguishable. We call the
elements GLB tilings from now on. Note that $\lambda$ is not a PV
number, so that one cannot have non-trivial eigenfunctions by
\cite{Boris}, which also implies that the pure point part of the
diffraction is trivial.  As we shall see, we cannot have absolutely
continuous diffraction either, which means that the diffraction is
essentially singular continuous.

For our purposes here, it is a little easier to harvest the underlying
fivefold rotation symmetry, and to work with the square of the
inflation in conjunction with an added rotation through $\pi$, as
illustrated in Figure~\ref{fig:LB}.  As one can quickly check, this
modified rule defines the same tiling hull as the original
one,\footnote{Some statistical data, such as the relative frequency of
  the prototiles or the frequency of vertices that carry a control
  point, can be extracted with standard Perron--Frobenius arguments
  from either rule.} now with the expansive map
$\tilde{Q} = \lambda^2 \one$.  Each prototile contains an orientation
vector as a marker, and the unique vertex it points to is chosen as
its control point. Note that control points of different tiles can
coincide, and that the vertex points of any GLB tiling differ in
whether they carry a control point or not. Giving weight $\frac{1}{5}$
and $\frac{2}{5}$ to the control points of thick and thin rhombuses,
respectively, one checks that these weights always add up to $1$ in an
occupied vertex point; compare \cite[Figs.~6.27 and 6.28]{TAO}. The
relative frequency of control points among the vertices is
$\frac{5-\sqrt{5}}{10} \approx 0.276 \ts\ts 393$.

Let $B^{(n)} (k)$ denote the Fourier matrix cocycle for the modified
inflation rule. One can easily check numerically that
$\det ( B^{(1)} (k)) \ne 0$ for some $k=k_0$, hence by continuity for
all $k$ in a small ball around $k_0$, and thus a.e.\ $k\in\RR^2$
by analyticity. Then, to apply Theorem~\ref{thm:higher-D}, we need to
compare $\frac{1}{N} \MM \bigl( \log \| B^{(N)} (.) \| \bigr)$ with
$\frac{1}{2} \log \bigl(\det \bigl(\tilde{Q} \bigr)\bigr) = 2 \log
(\lambda)$, with $\lambda$ as in \eqref{eq:lb-lam}.  Employing the
Frobenius norm, we may equivalently compare
$\frac{1}{N} \MM \bigl( \log \| B^{(N)} (.) \|^{2}_{\mathrm{F}}
\bigr)$ with $4 \log (\lambda) \approx 2.571 {\ts} 862$. Here,
$\| B^{(N)} (k) \|^{2}_{\mathrm{F}}$ is a non-negative trigonometric
polynomial in $k$ that is also quasiperiodic and, for any fixed $N$,
bounded away from $0$. Its logarithm is then a Bohr almost periodic
function that can be represented as a section of a periodic function
on the $4$-torus, and its mean is then given as an integral over
$\TT^4$. The numerical calculation leads to the values shown in
Table~\ref{tab:LB}, and thus to the following conclusion.

\begin{table}
\caption{Some upper bounds for $\chi^{B} (k)$ via means as
  explained in the main text, which are to be compared with
  the threshold value $4 \log (\lambda) \approx 2.571 {\ts}862$.
  The numerical error is less than $0.005$ in all cases 
  listed.\label{tab:LB}}
\renewcommand{\arraystretch}{1.2}
\begin{tabular}{|c|c@{\;\,\;}c@{\;\,\;}c@{\;\,\;}c@{\;\,\;}c
      @{\;\,\;}c@{\;\,\;}c@{\;\,\;}c|}
\hline
$N$ & 6 & 7 & 8 & 9 & 10 & 11 & 12 & 13 \\
\hline
$\tfrac{1}{N} \MM \bigl( \log \| B^{(N)}(.) \|^{2}_{\mathrm{F}} \bigr)$ & 
 2.643 & 2.572 & 2.517 & 2.474 & 2.440 & 2.411 & 2.387 & 2.367  \\[0.5mm]
  \hline
\end{tabular}
\end{table}

\begin{theorem}
  The Fourier transform\/ $\ts\widehat{\vU}$ of the pair correlation
  measures for any GLB tiling with the control points as defined above
  is a vector of singular measures, and\/ $\widehat{\vU}$ is the same
  for all elements of the tiling hull.
    
  For any GLB tiling, with vertex set $\vL$, the measure\/
  $\omega = \sum_{x\in \vL}\, w(x) \ts \delta_x$ with the coincidence
  weights\/ $w(x)$ from above has singular diffraction with trivial
  pure point part.  \qed
\end{theorem}

Indeed, as mentioned above, we know that the pure point part is the
trivial one.  The singular diffraction measure
$\widehat{\gamma^{}_{\omega}}$ is thus of the form
\[
    \widehat{\gamma^{}_{\omega}} \, = \, I^{}_{0} \, \delta^{}_{0}
    + \bigl(  \widehat{\gamma^{}_{\omega}}\bigr)_{\mathsf{sc}} 
    \quad \text{with} \quad  I^{}_{0} \, = \,  \left(
    \frac{5-\sqrt{5}}{10} \, \dens (\vL) \right)^{\! 2} ,
\]
where
$\dens (\vL) = \frac{1+\sqrt{5}}{5 } \lambda = \frac{2}{5} \sqrt{5+2
  \sqrt{5}\ts } \approx 1.231$, if we work with rhombuses of unit edge
length; compare \cite[Cor.~9.1]{TAO}.  Now, it is not difficult to see
that the vertex set with uniform weights is locally derivable from the
pattern with the weights used above and vice versa, which means they
are MLD. Although this will lead to different pair correlation and
diffraction measures, their spectral type remains unchanged, with the
following consequence.

\begin{coro}
  The diffraction measure of the uniform Dirac comb\/ $\delta_{\vL}$
  on the vertex set\/ $\vL$ of any GLB tiling is singular, and of the
  form\/
  $\,\widehat{\gamma} = \dens (\vL)^2 \ts \delta^{}_{0} +
  (\widehat{\gamma})^{}_{\mathsf{sc}}$.  \qed
\end{coro}

It remains to extend this result to the spectral measures of the
corresponding translational dynamical system, where we expect
the spectral measure of maximal type to be singular as well, with
a trivial pure point part due to the constant eigenfunction.

\section*{Appendix: The (skew) Kronecker product algebra}

The structure of the correlation measures relies on some properties of
the Kronecker product matrices
\begin{equation}\label{eq:def-A}
   \bs{A} (k) \, \defeq \, B(k) \otimes \overline{B(k)} \ts ,
\end{equation}
defined for $k\in\RR$. Obviously, one has
$\overline{\bs{A} (k)} = \bs{A} (-k)$ and
$\det (\bs{A} (k)) = \lvert \det (B(k))\rvert^{2 \aaa}$.

In view of the structure of Eq.~\eqref{eq:def-A}, let us now consider
the $\RR$-algebra $\bs{\cA}$ that is generated by the matrix family
$\{ \bs{A} (k) \mid k\in\RR \}$.  Due to the Kronecker product
structure, $\bs{\cA}$ fails to be irreducible, no matter what the
structure of the IDA $\cB$ is. Let us explore this in some more
detail. Let $V=\CC^{\aaa}$ and consider $W\defeq V \otimes^{}_{\CC} V$,
the (complex) tensor product, which is a vector space over $\CC$ of
dimension $\aaa^{2}$, but also one over $\RR$, then of dimension
$2\ts \aaa^{2}$. In the latter setting, consider the involution
$C \! : \, W \!\longrightarrow W$ defined by
\[
      x \otimes y \; \longmapsto \; C (x \otimes y)
      \, \defeq \, \overline{y \otimes x} \, = \,
      \bar{y} \otimes \bar{x}
\]
together with its unique extension to an $\RR$-linear mapping on $W$.
Observe that there is no \mbox{$\CC$-linear} extension, because
$C \bigl(a \ts (x \otimes y)\bigr) = \bar{a}\; C (x\otimes y)$ for
$a \in\CC$. With this definition of $C$, one finds for an arbitrary
$k\in\RR$ that
\[
\begin{split}
  \bs{A} (k)\, C(x\otimes y) \, & = \, 
   \bigl(B(k) \otimes \,\overline{\! B(k)\! }\, \bigr)
   ( \bar{y} \otimes \bar{x} ) \, = \, 
    \bigl(B(k)\ts \bar{y}\bigr) \otimes 
    \bigl(\, \overline{\! B(k)\ts  x \nts} \, \bigr) \\
   & = \, C \bigl(   \bigl(B(k) \otimes \, 
       \overline{\! B(k)\! }\, \bigr)
       (x \otimes y ) \bigr) \, = \, 
    C \bigl( \bs{A} (k) \, ( x \otimes y ) \bigr),
\end{split}
\]
so $C$ commutes with the linear map defined by $\bs{A} (k)$, for any
$k\in\RR$. The $\RR$-linear mapping $C$ has eigenvalues $\pm 1$ and is
diagonalisable, as follows from the unique splitting of an arbitrary
$w\in W$ as
$w = \frac{1}{2} \bigl( w + C(w)\bigr) + \frac{1}{2} \bigl( w -
C(w)\bigr)$.  So, our vector space splits as
$W \! = W_{\! +} \oplus W_{\! -}$ into \emph{real} vector spaces that
are eigenspaces of $C$. Their dimensions are
\[
   \dim^{}_{\RR} (W_{\! +}) \, = \,  \dim^{}_{\RR} (W_{\! -})
    \, = \, \aaa^{2}
\]
since $W_{\! -} = \ii \ts W_{\! +}$ with
$W_{\! +} \cap W_{\! -} = \{ 0 \}$.  It is thus clear that $W_{\! +}$
and $W_{\! -}$ are invariant (real) subspaces for the $\RR$-algebra
$\bs{\cA}$.

Observe next that we have
\[
   \bs{A} (k) \, = \, B(k) \otimes \overline{B (k)} \; = \!
   \sum_{x,y \in S_{T}}\! \ee^{2 \pi \ii k (x-y)} \, D_{x} \otimes D_{y}
   \; = \sum_{z\in \triangle_{T}}  \ee^{2 \pi \ii k z}\, F_{z}
\]
where $\triangle_{T} \defeq S_{T} - S_{T}$ is the Minkowski difference,
with $-\triangle_{T}=\triangle_{T}$, and
\begin{equation}\label{eq:F-def}
   F_{z} \; = \sum_{\substack{x,y \in S_{T} \\ x-y = z}} 
     D_{x} \otimes D_{y} \ts .
\end{equation}
In analogy to before, $F_{z} = F_{z'}$ is possible for $z\ne z'$.
For instance, if $z=x-y$ with $x\ne y$ and $D_{x} = D_{y}$, one
can get $F_{z} = F_{-z}$ if there is no other way to write $z$ as
a difference of two numbers in $S_{T}$.

For $a\in\CC$, one easily checks that
$C \circ (a F_{z}) = \bar{a} \ts C \circ F_{z} = \bar{a} F_{-z} \circ
C$, which implies $[C, \bs{A}(k)]=0$ for all $k\in\RR$, in line with
our previous derivation.  It is immediate that $\bs{\cA}$ is contained
in the \mbox{$\RR$-algebra} $\bs{\cA}_{F}$ that is generated by the
matrices $\{ F_{z}+F_{-z} \mid 0 \leqslant z\in\triangle_{T}\}$
together with
$\{ \ii\ts ( F_{z} - F_{-z}) \mid 0 \leqslant z\in\triangle_{T}\}$,
and an argument similar to the one previously used for $\cB$ shows
that $ \bs{\cA}_{F} \subseteq \bs{\cA}$, hence
$\bs{\cA} = \bs{\cA}_{F}$. Since
$\dim^{}_{\CC} (\cB) \leqslant \aaa^{2}$, and since we generate the
\emph{real} algebra only after taking the Kronecker product, one has
\[
    \dim^{}_{\RR} (\bs{\cA}) \, \leqslant \, \aaa^{4} \ts ,
\]
which is also clear from $\dim^{}_{\RR} (W_{\! +}) = \aaa^{2}$.
Moreover, one has the following result.

\begin{lemma}
  Let\/ $\varrho$ be a primitive inflation rule on an alphabet with
  $\aaa$ letters, and assume that the IDA\/ $\cB$ of\/ $\varrho$ is
  irreducible over\/ $\CC$.  Then, the induced\/ $\RR$-algebra\/
  $\bs{\cA}$ is isomorphic with\/ $\Mat (\aaa^2, \RR)$, and its action
  on the subspace\/ $W_{\! +}$ is irreducible as well, this time
  over\/ $\RR$.
\end{lemma}

\begin{proof}
  Here, $\cB$ irreducible over $\CC$ means $\cB = \Mat (\aaa , \CC)$.
  With $\varGamma \defeq \cB \otimes^{}_{\CC} \cB$, where
  $\otimes^{}_{\CC}$ denotes the tensor product over $\CC$, one has
  $\varGamma \simeq \Mat (\aaa^{2}, \CC)$ by standard
  arguments. Clearly, $\varGamma$ is a $\CC$-algebra of dimension
  $\aaa^{4}$, but also an $\RR$-algebra, then of dimension
  $2 \ts \aaa^{4}$.  Now, using the Kronecker product as
  representation of the tensor product,
  $M\otimes N \mapsto \overline{N} \otimes \overline{M}$ defines a
  mapping that has a unique extension to an automorphism $\sigma$ of
  $\varGamma$ as an $\RR$-algebra.

  Our $\RR$-algebra $\bs{\cA}$ consists of all fixed points of
  $\sigma$, so $\bs{\cA} = \{ Q \in \varGamma \mid \sigma (Q) = Q
  \}$. Employing the elementary matrices $E_{ij}$ from
  $\Mat (\aaa , \RR)$ together with
  $E_{ij,k\ell} \defeq E_{ik} \otimes E_{j \ell}$, we can give a basis of
  $\bs{\cA}$, seen as a vector space over $\RR$, by
\[
   \big\{ \tfrac{1}{2} (E_{ij,k\ell} + E_{k\ell,ij}) \mid 
   (i,j) \leqslant (k,\ell) \big\}
   \cup \big\{ \tfrac{\ii}{2} (E_{ij,k\ell} - E_{k\ell,ij})
   \mid (i,j) < (k,\ell) \big\} ,
\] 
where lexicographic ordering is used for the double indices.  Note
that the cardinalities are $\frac{1}{2} \aaa^{2} (\aaa^{2} + 1)$ and
$\frac{1}{2} \aaa^{2} (\aaa^{2} - 1)$, which add up to
$\dim^{}_{\RR} (\bs{\cA}) = \aaa^{4}$.

Next, observe that we can get $E_{ij,k\ell}$ and $E_{k\ell,ij}$ by a
simple (complex) linear combination of
$ \tfrac{1}{2} (E_{ij,k\ell} + E_{k\ell,ij})$ and
$\tfrac{\ii}{2} (E_{ij,k\ell} - E_{k\ell,ij})$, and vise versa. Put
together, this defines a (complex) inner automorphism of
$\varGamma$. Observing that
$\Mat (\aaa^{2}, \RR) = \{ Q \in \varGamma \mid \overline{Q} = Q \}$,
this construction can now be used to show that
$\bs{\cA} \simeq \Mat (\aaa^{2}, \RR)$, which is a central simple
algebra. Since $\bs{\cA} W_{\! +} \subseteq W_{\! +}$ and
$\dim^{}_{\RR} (W_{\! +}) = \aaa^{2}$, the claimed irreducibility over
the reals follows.
\end{proof}

Note that all $F_{z}$ are non-negative, integer matrices. They clearly
satisfy the relation
$\sum_{z\in\triangle_{T}} F_{z} = \bs{A} (0) = M_{\varrho} \otimes
M_{\varrho}$.  Moreover, under some mild conditions, the spectral
radius of $F^{}_0$ is $\lambda$, while the other matrices $F_z$ have
smaller spectral radius.

Let us come back to Eq.~\eqref{eq:def-A}, which implies
\[
    \| \bs{A} (k) \|^{}_{\mathrm{F}} \, = \,
    \| B (k) \|^{2}_{\mathrm{F}} \ts .
\]
If we consider the matrix cocycle defined by
$\bs{A}^{(n)} (k) = B^{(n)} (k) \otimes \overline{B^{(n)} (k)}$, it is
immediate that the maximal Lyapunov exponents, for all $k\in\RR$, are
related by
\begin{equation}\label{eq:A-versus-B}
   \chi^{\bs{A}} (k) \, = \, 2 \ts \chi^{B} (k) \ts ,
\end{equation}
which also holds for the higher-dimensional case with $k\in\RR^d$.
Clearly, one can now reformulate Theorems~\ref{thm:1D} and
\ref{thm:higher-D} in terms of $\chi^{\bs{A}}$.  In particular, one
has the following reformulation of Theorem~\ref{thm:geq} and
Corollary~\ref{coro:ac-cond} and their higher-dimensional analogues.

\begin{coro}
  Let\/ $\varrho$ be a primitive inflation rule in\/ $\RR^d$ with
  finitely many translational prototiles and expansive map\/
  $Q$. Let\/ $B^{(n)} (.)$ be its Fourier matrix cocycle, with\/
  $\det(B(k))\ne 0$ for at least one\/ $k\in\RR^d$, and\/
  $\bs{A}^{(n)} (.) = B^{(n)} (.) \otimes \overline{B^{(n)} (.)}$ the
  corresponding Kronecker product cocycle. If the diffraction measure
  of the hull defined by\/ $\varrho$ contains a non-trivial absolutely
  continuous component, one has\/
  $\chi^{\bs{A}} (k) = \log\ts\ts \lvert \det(Q) \rvert$ for a subset
  of\/ $\RR^d$ of positive measure, which has full measure when\/
  $\chi^{\bs{A}} (k)$ is constant for a.e.\ $k\in\RR^d$.  \qed
\end{coro}

\section*{Acknowledgements}

It is a pleasure to thank Frederic Alberti, Alan Bartlett, Scott
Balchin, Natalie Frank, Uwe Grimm, Andrew Hubery, Robbie Robinson,
Boris Solomyak and Nicolae Strungaru for helpful discussions. We also
thank two anonymous reviewers for their thoughtful comments.  This
work was supported by the German Research Foundation (DFG), within the
CRC 1283.


\bigskip
\begin{thebibliography}{99}
\small

\bibitem{Aki}
S.~Akiyama, M.~Barge, V.~Berth\'{e}, J.-Y.~Lee and A.~Siegel,
On the Pisot substitution conjecture, in \cite{KLS}, pp.~33--72.

\bibitem{BFGR}
M.~Baake, N.P.~Frank, U.~Grimm and E.A.~Robinson,
Geometric properties of a binary non-Pisot inflation and absence 
of absolutely continuous diffraction,
\textit{Studia Math.} \textbf{247} (2019) 109--154;
\texttt{arXiv:1706.03976}.

\bibitem{renex}
M.~Baake and F.~G\"{a}hler,
Pair correlations of aperiodic inflation rules via
renormalisation: Some interesting examples,
\textit{Topol.\  \& Appl.} \textbf{205} (2016) 4--27;
\texttt{arXiv:1511.00885}.

\bibitem{squiral}
M.~Baake and U.~Grimm,
Squirals and beyond: Substitution tilings with singular 
continuous spectrum,
\textit{Ergodic Th.\ \& Dynam.\ Syst.} \textbf{34} 
(2014) 1077--1102; \texttt{arXiv:1205.1384}.

\bibitem{TAO}
M.~Baake and U.~Grimm,
\textit{Aperiodic Order. Vol.\ $1$: A Mathematical Invitation},
Cambridge Univ.\ Press, Cambridge (2013).

\bibitem{TAO2}
M.~Baake and U.~Grimm (eds.),
\textit{Aperiodic Order. Vol.\ $2$: Crystallography and
Almost Periodicity}, Cambridge University Press,
Cambridge (2017).

\bibitem{BG-block}
M.~Baake and U.~Grimm,
Renormalisation of pair correlations and their Fourier
transforms for primitive block substitutions,
in \textit{Tiling and Discrete Geometry},
eds.\ S.~Akiyama and P.~Arnoux,
Springer, Berlin, in press;
\texttt{arXiv:1906.10484}.

\bibitem{BGM}
M.~Baake, U.~Grimm and N.~Ma\~{n}ibo,
Spectral analysis of a family of binary inflation rules,
\textit{Lett.\ Math.\ Phys.} \textbf{108} (2018) 1783--1805;
\texttt{arXiv:1709.09083}.

\bibitem{BHL}
M.~Baake, A.~Haynes and D.~Lenz,
Averaging almost periodic functions along exponential
sequences, in \cite{TAO2}, pp.~343--362;
\texttt{arXiv:1704.08120}.

\bibitem{BL}
M.~Baake and D.~Lenz,
Dynamical systems on translation bounded measures:\
Pure point dynamical and diffraction spectra,
\textit{Ergodic Th.\ \& Dynam.\ Syst.} \textbf{24} (2004)
1867--1893; \newline
\texttt{arXiv:math.DS/0302231}.

\bibitem{BL-rev}
M.~Baake and D.~Lenz,
Spectral notions of aperiodic order,
\textit{Discr.\ Cont.\ Dynam.\ Syst.\ S}
\textbf{10} (2017) 161--190;
\texttt{arXiv:1601.06629}.

\bibitem{BLvE}
M.~Baake, D.~Lenz and A.C.D.~van Enter,
Dynamical versus diffraction spectrum for structures
with finite local complexity,
\textit{Ergodic Th.\ \& Dynam.\ Syst.}
\textbf{35} (2015) 2017--2043;
\texttt{arXiv:1307.7518}.

\bibitem{BM}
M.~Baake and R.V.~Moody,
Weighted Dirac combs with pure point diffraction,
\textit{J.\ Reine Angew.\ Math.\ (Crelle)} \textbf{573}
(2004) 61--94; \texttt{arXiv:math.MG/0203030}.

\bibitem{BaPe}
L.~Barreira and Y.~Pesin, 
\textit{Nonuniform Hyperbolicity}, 
Cambridge University Press, Cambridge (2007).

\bibitem{Bart}
A.~Bartlett,
Spectral theory of $\ZZ^d$ substitutions,
\textit{Ergodic Th.\ \& Dynam.\ Syst.} \textbf{38}
(2018) 1289--1341; \texttt{arXiv:1410.8106}.

\bibitem{BF}
C.~Berg and G.~Forst,
\textit{Potential Theory on Locally Compact Abelian Groups},
Springer, Berlin (1975).

\bibitem{BS}
A.~Berlinkov and B.~Solomyak,
Singular substitutions of constant length,
\textit{Ergodic Th.\ \& Dynam.\ Syst.}, in press;
\texttt{arXiv:1705.00899}.
% \marginpar{online erschienen}

\bibitem{BS1}
A.I.~Bufetov and B.~Solomyak,
On the modulus of continuity for spectral measures
in substitution dynamics, \textit{Adv.\ Math.}
\textbf{260} (2014) 84--129;
\texttt{arXiv:1305.7373}.
  
\bibitem{BS2}
A.I.~Bufetov and B.~Solomyak,
A spectral cocycle for substitution systems and translation
flows, \textit{preprint}
\texttt{arXiv:1802.04783}.

\bibitem{CG}
L.~Chan and U.~Grimm,
Spectrum of a Rudin--Shapiro-like sequence,
\textit{Adv.\ Appl.\ Math.} \textbf{87} (2017) 16--23;
\texttt{arXiv:1611.04446}.

\bibitem{CGS}
L.~Chan, U.~Grimm and I.~Short,
Substitution-based structures with absolutely continuous spectrum,
\textit{Indag.\ Math.} \textbf{29} (2018) 1072--1086;
\texttt{arXiv:1706.05289}.

\bibitem{CS}
A.~Clark and L.~Sadun,
When size matters:\ Subshifts and their related tiling spaces,
\textit{Ergodic Th.\ \& Dynam.\ Syst.} \textbf{23} (2003) 1043--57; 
\texttt{arXiv:math.DS/0201152}.

\bibitem{David}
D.~Damanik, R.~Sims and G.~Stolz,
Localization for one-dimensional, continuum, 
Bernoulli--Anderson models,
\textit{Duke Math.\ J.} \textbf{114} (2002)
59--100;
\texttt{arXiv:math-ph/0010016}.

\bibitem{Dekking}
F.M.~Dekking,
The spectrum of dynamical systems arising from substitutions 
of constant length,
\textit{Z.\ Wahrscheinlichkeitsth.\ Verw.\ Geb.} \textbf{41} 
(1978) 221--239.

\bibitem{Durand}
F.~Durand,
A characterization of substitutive sequences using return words,
\textit{Discr.\ Math.} \textbf{179} (1998) 89--101;
\texttt{arXiv:0807.3322}.

\bibitem{EW}
M.~Einsiedler and T.~Ward,
\textit{Ergodic Theory --- with a View towards Number Theory},
GTM 259, Springer, London (2011).

\bibitem{EvWa}
G.~Everest and T.~Ward,
\textit{Heights of Polynomials and Entropy in
Algebraic Dynamics}, Springer, London (1999).

\bibitem{Schmeling}
A.-H.~Fan, B.~Saussol and J.~Schmeling,
Products of non-stationary random matrices and multiperiodic
equations of several scaling factors,
\textit{Pacific J.\ Math.} \textbf{214} (2004) 31--54; \newline
\texttt{arXiv:math.DS/0210347}.

\bibitem{NF-Hadamard}
N.P.~Frank,
Substitution sequences in $\ZZ^d$ with a nonsimple Lebesgue 
component in the spectrum,
\textit{Ergodic Th.\ \& Dynam.\ Syst.} \textbf{23} % no 2: 
(2003) 519--532.

\bibitem{NF}
N.P.~Frank,
Multi-dimensional constant-length substitution sequences,
\textit{Topol.\  \& Appl.} \textbf{152} (2005) 44--69.

\bibitem{Nat-review}
N.P.~Frank,
Introduction to hierarchical tiling dynamical systems,
in \textit{Tiling and Discrete Geometry},
eds.\ S.~Akiyama and P.~Arnoux, Springer,
Berlin, in press; % \marginpar{update}
\textit{preprint} \texttt{arXiv:1802.09956}.

\bibitem{Dirk}
D.~Frettl\"{o}h,
More inflation tilings, in \cite{TAO2}, pp.~1--37.

\bibitem{FR}
D.~Frettl\"{o}h and C.~Richard,
Dynamical properties of almost repetitive Delone sets,
\textit{Discr.\ Cont.\ Dynam.\ Syst.\ A} \textbf{34}
(2014) 531--556; \texttt{arXiv:1210.2955}.

\bibitem{G}
F.~Gantmacher, 
\textit{Matrizentheorie},
Springer, Berlin (1986).

\bibitem{GLA}
J.~Gil de Lamadrid and L.N.~Argabright,
Almost periodic measures,
\textit{Memoirs Amer.\ Math.\ Soc.} \textbf{85} (1990) no.~428 
(AMS, Providence, RI).

\bibitem{GL}
C.~Godr\`{e}che and F.~Lan\c{c}on,
A simple example of a non-Pisot tiling with
five-fold symmetry,
\textit{J.\ Phys.\ I (France)} \textbf{2} (1992) 207--220.

\bibitem{Hof}
A.~Hof,
On diffraction by aperiodic structures, 
\textit{Commun.\ Math.\ Phys.} \textbf{169} (1995) 25--43.

\bibitem{JL}
G.~James and M.~Liebeck,
\textit{Representations and Characters of Groups},
2nd ed., Cambridge University Press, Cambridge (2001).

\bibitem{KLS}
J.~Kellendonk, D.~Lenz and J.~Savinien (eds.),
\textit{Mathematics of Aperiodic Order},
Birkh\"{a}user, Basel (2015).

\bibitem{LB}
F.~Lan\c{c}on and L.~Billiard,
Two-dimensional system with a quasicrystalline ground state,
\textit{J.\ Phys.\ (France)} \textbf{49} (1988) 249--256.

\bibitem{Lang}
S.~Lang,
\textit{Algebra},
rev.\ 3rd ed., Springer, New York (2002).

\bibitem{LMS}
J.-Y.~Lee, R.V.~Moody and B.~Solomyak,
Pure point dynamical and diffraction spectra,
\textit{Ann.\ Henri Poincar\'e} \textbf{3} (2002) 1003--1018;
\texttt{arXiv:0910.4809}.

\bibitem{Lenz}
D.~Lenz,
Continuity of eigenfunctions of uniquely ergodic
dynamical systems and intensity of Bragg peaks,
\textit{Commun.\ Math.\ Phys.} \textbf{287} (2009) 225--258;
\texttt{arXiv:math-ph/0608026}.

\bibitem{LS}
D.~Lenz and N.~Strungaru,
Pure point spectrum for measure dynamical systems on locally
compact Abelian groups,
\textit{J.\ Math.\ Pures Appl.} \textbf{92} (2009) 323--341;
\texttt{arXiv:0704.2498}.

\bibitem{LR}
V.~Lomonosov and P.~Rosenthal,
The simplest proof of Burnside's theorem on matrix algebras,
\textit{Lin.\ Alg.\ Appl.} \textbf{383} (2004) 45--47.

\bibitem{Neil}
N.~Ma\~{n}ibo,
Lyapunov exponents for binary substitutions of constant length,
\textit{J.\ Math.\ Phys.} \textbf{58} (2017) 113504:1--9;
\texttt{arXiv:1706.00451}. 

\bibitem{Neil-MFO}
N.~Ma\~{n}ibo,
Spectral analysis of primitive inflation rules,
\textit{Oberwolfach Rep.} \textbf{14} (2017)
2830--2832.

\bibitem{Neil-diss}
N.~Ma\~{n}ibo,
\textit{Lyapunov Exponents in the Spectral Theory of
  Primitive Inflation Systems}, PhD thesis, Bielefeld
University (2019); available at 
\texttt{https://pub.uni-bielefeld.de/record/2935972}.

\bibitem{MS}
R.V.~Moody and N.~Strungaru,
Almost periodic measures and their Fourier transforms,
in \cite{TAO2}, pp.~173--270.

\bibitem{MR}
P.~M\"{u}ller and C.~Richard,
Ergodic properties of randomly coloured point sets,
\textit{Can.\ J.\ Math.} \textbf{65} (2013) 349--402;
\texttt{arXiv:1005.4884}.

\bibitem{Q}
M.~Queff\'{e}lec,
\textit{Substitution Dynamical Systems --- Spectral Analysis},
2nd ed., LNM 1294, Springer, Berlin (2010).

\bibitem{Robbie}
E.A.~Robinson Jr.,
Symbolic dynamics and tilings of $\RR^{d}$,
\textit{Proc.\ Sympos.\ Appl.\ Math.} \textbf{60} (2004) 81--119.

\bibitem{Rudin}
W.~Rudin,
\textit{Fourier Analysis on Groups},
Wiley, New York (1962).

\bibitem{Scott}
W.R.~Scott,
\textit{Group Theory}, Prentice Hall, Englewood Cliffs, NJ (1964).

\bibitem{Bernd}
B.~Sing,
\textit{Pisot Substitutions and Beyond},
PhD thesis, Bielefeld University (2006);
available at \newline
\texttt{https://pub.uni-bielefeld.de/record/2302336}.

\bibitem{Boris}
B.~Solomyak,
Dynamics of self-similar tilings,
\textit{Ergodic Th.\ \& Dynam.\ Syst.} \textbf{17} (1997) 695--738
and \textit{Ergodic Th.\ \& Dynam.\ Syst.} \textbf{19} (1999)
1685 (erratum).

\bibitem{Boris-2}
B.~Solomyak,
Nonperiodicity implies unique composition for self-similar
translationally finite tilings,
\textit{Discr.\ Comput.\ Geom.} \textbf{20} (1998) 265--278.

\bibitem{S}
N.~Strungaru, 
On the Fourier analysis of measures with Meyer set support, 
\textit{preprint} \newline
\texttt{arXiv:1807.03815}.

\bibitem{Trebin}
H.-R.~Trebin (ed.),
\textit{Quasicrystals --- Structure and Physical Properties},
Wiley-VCH, Weinheim (2003).

\bibitem{Viana}
M.~Viana,
\textit{Lectures on Lyapunov Exponents},
Cambridge University Press, Cambridge (2013).

\bibitem{Vince}
A.~Vince,
Digit tiling of Euclidean space,
in:\ M.\ Baake and R.V.\ Moody (eds.),
\textit{Directions in} \textit{Mathematical Quasicrystals},
CRM Monograph Series, vol.\ \textbf{13}, AMS, Providence, RI (2000),
\mbox{pp.~329--370}.

\end{thebibliography}
\end{document}